\documentclass[11pt,arial,letterpaper]{article}
\usepackage[letterpaper,margin=1in,footskip=0.25in]{geometry}

\usepackage{latexsym}
\usepackage{amsmath}
\usepackage{rotating} 
\usepackage{theorem}
\usepackage{epsfig,subfigure}
\usepackage{graphicx} 
\usepackage{upgreek}
\usepackage{txfonts}
\usepackage{pxfonts}
\usepackage{pstricks}
\usepackage{multirow} 
\usepackage{epsfig}  
\usepackage{rotating}
\usepackage{cite}
\usepackage{color, colortbl}
\usepackage[square,numbers]{natbib}

\usepackage[refpage]{nomencl}

\makenomenclature

\usepackage{makeidx}

\makeindex

\newtheorem{thm}{Theorem}[section]
\newtheorem{lem}[thm]{Lemma}
\newtheorem{cor}[thm]{Corollary}
\newtheorem{proposition}[thm]{Proposition}
\newtheorem{Def}[thm]{Definition}
\newtheorem{rmk}[thm]{Remark}

\newtheorem{ex}[thm]{Example}

\newenvironment{proof}[1][Proof]{\begin{trivlist}
\item[\hskip \labelsep {\bfseries #1}]}{\end{trivlist}}

\newcommand{\qed}{\hfill $\Box$ 
}

\newcommand{\END}{\hfill\mbox{\raggedright$\Diamond$}}

\newcommand{\etal}{{\em et al.}}

\newcommand{\bc}{\mbox{$\mathbb C$}}
\newcommand{\bn}{\mbox{$\mathbb N$}}

\newcommand{\noi}{\noindent}
\newcommand{\q}{\quad}

\def\la{\langle}
\def\ra{\rangle}

\def\lam{\lambda}

\def\cL{\mathcal{L}}

\def\cT{\mathcal{T}}
\def\cP{\mathcal{P}}

\usepackage{multirow}
\usepackage{arydshln}

\def\ind{\text{\rm ind\,}}

\def\Ind{\text{\rm Ind\,}}

\usepackage{xcolor}

\usepackage[normalem]{ulem}

\definecolor{Gray}{gray}{0.9}

\title{Reduced Lattices of Synchrony Subspaces and their Indices}

\author{
	Hiroko Kamei$^{1}$ and 	Haibo Ruan$^{2}$
	\\
	$^1$ {\small Division of Mathematics, University of Dundee}\\
	$^2$ {\small Institute of Mathematics, Technical University of Hamburg-Harburg}\\
	{\small E-mail: hkamei@dundee.ac.uk \quad haibo.ruan@tuhh.de}}

\date{} 

\begin{document} 
	
\maketitle

\begin{abstract}
For a regular coupled cell network, synchrony subspaces are the polydiagonal subspaces that are invariant under the network adjacency matrix. The complete lattice of synchrony subspaces of an $n$-cell regular network can be seen as an intersection of the partition lattice of $n$ elements and a lattice of invariant subspaces of the associated adjacency matrix. 
We assign integer tuples with synchrony subspaces, and use them for identifying equivalent synchrony subspaces to be merged. Based on this equivalence, the initial lattice of synchrony subspaces can be reduced to a lattice of synchrony subspaces which corresponds to a simple eigenvalue case discussed in our previous work. The result is a reduced lattice of synchrony subspaces, which affords a well-defined non-negative integer index that leads to bifurcation analysis in regular coupled cell networks.
\end{abstract}

\section*{}
\noindent
AMS classification numbers: 15A72, 06A06, 06B23, 37C25

\noindent
Keywords: coupled cell network, Jordan normal form, synchrony subspaces, lattice, index

\section{Introduction}
A coupled cell system is a finite collection of individual dynamical systems (or {\it cells}) that  are coupled together through mutual interactions. Coupled cell systems can be used to model a wide variety of phenomena in many scientific fields, ranging from physics, biology, chemistry, to engineering, social science and climatology (cf. Aguiar \etal~\cite{Aguiar-2016}, Golubitsky \etal~\cite{GS06} and references therein).

Topological configuration of a coupled cell system can be described by a directed graph, a {\it coupled cell network}, whose nodes correspond to the cells and whose edges represent the interactions. A coupled cell network is called {\it regular}, if it has a single edge type and consists of identical cells (having the same phase space and the same internal dynamics) which all have the same number of input edges. 

Dynamics and bifurcations of coupled cell systems have been much studied, for example Leite \etal~\cite{LG06}, Elmhirst \etal~\cite{EG06}, Aguiar \etal~\cite{ADGL09}, Golubitsky \etal~\cite{Golubitsky-2009}, Stewart \etal~\cite{SG11} and Soares \cite{Soares17} for regular networks; and Stewart \etal~\cite{SGP03}, Golubitsky \etal~\cite{GST05} and Gandhi \etal~\cite{GGPSW20} and Aguiar \etal~\cite{ADS2020} for more general coupled cell. 

One of the key properties of a regular network is the existence of synchrony subspaces, that are the polydiagonals defined in terms of equalities of cell coordinates and are left invariant under any coupled cell systems consistent with the network structure. Their existence depends solely on the network structure and not on given admissible vector field. Synchrony subspaces range from the full-synchrony subspace $\Delta$, formed by setting all cell coordinates equal in the total phase space, towards subspaces with less synchrony, up to the total phase space $P$, formed by setting all cell coordinates distinct. We say that the system undergoes a {\it local synchrony-breaking steady-state bifurcation}, if the synchronous equilibrium in $\Delta$ changes its stability and bifurcates to a steady state with less synchrony, as a bifurcation parameter crosses some critical value. If it bifurcates to a periodic state with less synchrony, we call it a {\it local synchrony-breaking Hopf bifurcation}. 

Synchrony subspaces are polydiagonal subspaces that are left invariant under the network adjacency matrix $A$. Thus a lattice of synchrony subspaces inherits properties as a lattice of invariant subspaces of a linear map represented by $A$ as studied by Brickman \etal~\cite{BF67} and Longstaff~\cite{Longstaff-1984}. It is impossible to draw a Hasse diagram of a lattice of invariant subspaces in general, however, we can picture a lattice of synchrony subspaces using the property of flow-invariant polydiagonals which are in one-to-one correspondence to balanced colorings of cells in the network that come from certain partitions of cells. All possible partitions of $n$ cells can be represented by the partition lattice of $n$ elements where partitions are partially ordered. As a result, the partition lattice of $n$ elements gives the maximum possible lattice of synchrony subspaces of an $n$-cell regular network. This enables us to visualise a finite lattice of synchrony subspaces as a combination of an infinite lattice of invariant subspaces with a finite partition lattice. Construction and properties of a lattice of synchrony subspaces have been studied \cite{Stewart-2007,Kamei-2009-part1,Kamei-2013,Aguiar-2014,Moreira15,NSS2020,ADS2020,Soares17}. When the adjacency matrix of a given regular network has non-simple eigenvalues, the size of a lattice of synchrony subspaces increases due to higher dimension eigenspaces. It is our interest to simplify the representations of such large lattices by identifying the key building blocks.

The concept of {\it lattice index} was introduced in \cite{Kamei-2009-part1}  for lattices of synchrony subspaces for regular networks whose adjacency matrix  $A$ has only simple eigenvalues, and later used for synchrony-breaking bifurcation analysis in \cite{Kamei-2009-part2}. An important feature of such lattice index is that  they are non-negative and their sum equals the total number of cells in the underlying network, which enabled enumeration of all possible lattice structures for a given network size and, furthermore, positive lattice indices were used to predict the existence of synchrony-breaking bifurcating branches.

Coupled cell systems are known to be capable of enforcing multiple eigenvalues of adjacency matrices in a generic manner due to the underlying network structure, which determines, even at linear level, the kind of generic transitions from a synchronous equilibrium that can occur as the parameter is varied. While the bifurcation analysis for simple critical eigenvalues is straightforward,  multiple eigenvalues can lead to complicated bifurcating behaviour of the system such as multiple bifurcations and secondary bifurcations. Our aim of this paper is to extend the existing concept of lattice index for lattices of synchrony subspaces for regular networks whose adjacency matrices may have not only simple eigenvalues, which can be used for bifurcation analysis in coupled cell systems.

The structure of the paper is as follows. In Section\ref{sec:LM}, we construct a lattice $\mathcal{L}_{M}$, which is uniquely determined by the Jordan normal form of the adjacency matrix of a given regular network. It will be shown that any lattice of synchrony subspaces with simple eigenvalues or any reduced lattice of synchrony subspaces with non-simple eigenvalues is a closed subset of $\mathcal{L}_{M}$ {(Lemma \ref{lem:closed_PA})}. In Section \ref{sec:PA}, we construct a poset (partially order set) $\mathcal{P}_{A}$ from the lattice of synchrony subspaces $V_{\mathcal{G}}^{P}$, where some nodes of $\mathcal{P}_{A}$ can have identical tuple representations. We construct a reduced lattice $\mathcal{P}_{A}/{=}$ using that equality as an equivalence relation. We aim to construct the same structure as $\mathcal{P}_{A}/{=}$ starting from an alternative representation of the lattice of synchrony subspaces, the poset $\mathcal{E}_{A}$. In Section \ref{sec:QA}, we construct the poset $\mathcal{E}_{A}$, where $s\in \mathcal{E}_{A}$ represents a set of distinct eigenvalues of a quotient network and the connectivity of $\mathcal{E}_{A}$ is the same as the lattice of synchrony subspaces $V_{\mathcal{G}}^{P}$. It will be proved that if $\mathcal{P}_{A}$ and $V_{\mathcal{G}}^{P}$ have the same covering relation (Definition \ref{def:covering_relation}), then $\mathcal{P}_{A}/{=}\, \subseteq \,\mathcal{E}_{A}/{\sim}$, which implies that reduced lattices {$\mathcal{P}_{A}/=$} can be identified using quotient posets $\mathcal{E}_{A}/{\sim}$ (Lemma \ref{lem:balanced_connectivity_matrix_QA}). Moreover, it will be shown that under the same covering relation assumption, the reduced lattice $\mathcal{E}_{A}/{\sim}$ indeed corresponds to a lattice structure for regular networks whose coupling matrices have only simple eigenvalues (Theorem \ref{thm:reduced_EA}). In Section \ref{sec:examples}, we present examples of regular networks to  illustrate our lattice reduction algorithm, taking into account the covering relation condition.

\section{Preliminaries}
In this section, we review basic concepts of coupled cell networks and their lattices of synchrony subspaces. For details, we refer to {\cite{Stewart-2007,GS06}}.

\subsection{Regular Networks}
In  what follows, we consider an $n$-cell regular network $\mathcal{G}$ with total phase space $P$ and the associated $n\times n$ adjacency matrix $A$.

\begin{Def}
\rm
	A {\it coupled cell network} consists of a finite nonempty set $\mathcal C$ of {\it nodes} or {\it cells} and a finite nonempty set $\mathcal E= \{ (c,d):\ c,d \in \mathcal C\}$ of {\it edges} or {\it arrows} and two equivalence relations: $\sim_{C}$ on $\mathcal C$ and $\sim_{E}$ on $\mathcal E$ such that the {\it consistency condition} is satisfied: if $(c_1,d_1) \sim_{E} (c_2,d_2)$, then $c_1 \sim_{C} c_2$ and $d_1 \sim_{C} d_2$.
	We write $\mathcal G = (\mathcal C,\mathcal E,\sim_C, \sim_E)$. 
\END
\end{Def} 
A coupled cell network can be represented by a directed graph, where the cells are placed at vertices, edges are depicted by directed arrows and the equivalence relations are indicated by different types of vertices or edges in the graph.

\begin{Def} 
\rm  
\label{defregnet}
A coupled cell network is called  {\it regular}, if it has only  one cell-equivalent class and one edge-equivalent class and moreover, every cell receives the same number of input edges. The number, which is the cardinality of the input set for every cell, is called the {\it valency} of the regular network.
\END
\end{Def}

The coupling structure of a regular network is given by an {\it adjacency matrix} $A=(a_{ij})$, where $a_{ij}$ is the number of input edges of the $i$-th cell from the $j$-th cell. The sum of the $i$-th row of $A$ is then equal to the number of total input edges received by the $i$-th cell.

\subsection{Lattice Theory} 
We summarise basic results from lattice theory in the following. See~\cite{Davey} for more details.

A \textit{lattice} is a partially ordered set $X$ such that any two elements $x,y\in X$ have a unique \textit{greatest lower bound} or \textit{meet}, and a unique \textit{least upper bound} or \textit{join}, denoted respectively by
\[x\wedge y\quad\quad x\vee y.\]

A complete lattice has a \textit{top} (maximal) element, denoted $\top$, and \textit{bottom} (minimal) element, denoted $\bot$. These are respectively the unique minimal and maximal elements of $X$.  

Let $L$ be a lattice and $\emptyset\neq M\subseteq L$. Then $M$ is a \textit{sublattice}  of $L$ if
\[a\wedge b\in M \quad\textrm{and}\quad a\vee b\in M\quad \forall a,b\in M.\]

The following defines structure-preserving maps between ordered sets.
\begin{Def}
\rm
Let $P$ and $Q$ be ordered sets. A map $\varphi:P\rightarrow Q$ is said to be
\begin{itemize}
\item[(i)]
\textit{order-preserving} (or, alternatively, \textit{monotone}) if $x\leq y$ in $P$ implies $\varphi(x)\leq\varphi(y)$ in $Q$;

\item[(ii)]
an \textit{order-embedding} if $x\leq y$ in $P$ if and only if $\varphi(x)\leq\varphi(y)$ in $Q$;

\item[(iii)]
an \textit{order-isomorphism} if it is an order-embedding mapping $P$ onto $Q$.
\end{itemize}
When there exists an order-isomorphism from $P$ to $Q$, we say that $P$ and $Q$ are \textit{order-isomorphic} and write $P\cong Q$.
\END
\end{Def}

\subsection{Lattice of Synchrony Subspaces}
The following is a summary of results from \cite{Stewart-2007}.

Suppose that $\mathcal{G}$ is a coupled cell network with cells $\mathcal{C}$ and a choice of total phase space $P$. Let $M_{\mathcal{G}}$ be the complete lattice of all equivalence relations on $\mathcal{C}$. Associated with each equivalence relation $\bowtie\in M_{\mathcal{G}}$ is a subspace $\Delta_{\bowtie}\subseteq P$, called the {\it polydiagonal} corresponding to $\bowtie$. It is defined by
\[\Delta_{\bowtie} = \{x\in P: c,d\in\mathcal{C}\quad\textrm{and}\quad c\bowtie d\Rightarrow x_{c}=x_{d}\}. \]
It consists of all $x$ whose components are equal for $\bowtie$-equivalent coordinates.

Define $W_{\mathcal{G}}^{P}$ to be the set of all polydiagonals for this choice of $P$ and $\mathcal{G}$. There is a natural bijection map
\[\delta: M_{\mathcal{G}}\rightarrow W_{\mathcal{G}}^{P}\quad \delta(\bowtie)=\Delta_{\bowtie},\]
where $M_{\mathcal{G}}$ and $W_{\mathcal{G}}^{P}$ are both complete lattices.

\begin{lem}
\label{lem:reverse-order}
The map $\delta$ is a lattice anti-isomorphism, that is, an isomorphism that reverses order, hence interchanges meet and join.
\end{lem}

An equivalence relation $\bowtie$ on $\mathcal{C}$ can be interpreted as a coloring of $\mathcal{C}$ in which $\bowtie$-equivalent cells receive the same color. For regular networks, a coloring is \textit{balanced} if any pair of cells with color $r$ have the same number of inputs from cells of color $b$ for each $b$. If $\bowtie$ is balanced, we call $\Delta_{\bowtie}$ a \textit{balanced polydiagonal} (or \textit{synchrony subspace}).

Let $\Lambda_{\mathcal{G}}$ be the set of all balanced equivalence relations for $\mathcal{G}$, and denote the set of all synchrony subspaces for $\mathcal{G}$ by $V_{\mathcal{G}}^{P}$. Then
\[\Lambda_{\mathcal{G}}\subseteq M_{\mathcal{G}}\quad V_{\mathcal{G}}^{P}\subseteq W_{\mathcal{G}}^{P}.\]

Let $\bowtie_{1},\bowtie_{2}\in\Lambda_{\mathcal{G}}$. We say that $\bowtie_{1}$ refines $\bowtie_{2}$, denoted by $\bowtie_{1}\prec\bowtie_{2}$, if and only if
\[c\bowtie_{1}d\Rightarrow c\bowtie_{2}d.\]
That is, the partition of $\mathcal{C}$ defined by $\bowtie_{1}$ is finer than that defined by $\bowtie_{2}$ in the sense that for any $c\in\mathcal{C}$
\[[c]_{1}\subseteq [c]_{2}\]
where $[c]_{j}$ is the $\bowtie_{j}$-equivalence class of $c$ for $j=1,2$. Observe that $\prec$ is a partial ordering on $\Lambda_{\mathcal{G}}$. As in Lemma \ref{lem:reverse-order}, forming polydiagonals reverses order:
\[\bowtie_{1}\prec\bowtie_{2}\Leftrightarrow \Delta_{\bowtie_{1}}\supseteq\Delta_{\bowtie_{2}}.\]

In \cite{Stewart-2007}, it has been shown that $\Lambda_{\mathcal{G}}$ is not a sublattice of $M_{\mathcal{G}}$, however, these two lattices share the same join operation.  

{\it Admissible} maps for a given coupled cell network $\mathcal{G}$ are defined as maps that are compatible with the network structure, which is characterised by the corresponding adjacency matrix. We say that a subspace $V$ of the total phase space $P$ is \textit{admissibly invariant} if $f(V)\subseteq V$ for every admissible map $f$ on $P$. A crucial property of a balanced equivalence relation is given in the following Theorem.

\begin{thm}(see [\cite{Stewart-2007}, Theorem 5.6] and [\cite{GST05}, Theorem 4.3])
Let $\bowtie$ be an equivalence relation on a coupled cell network. Then $\Delta_{\bowtie}$ is admissibly invariant if and only if $\bowtie$ is balanced.
\end{thm}

Since the network structure is characterized by the corresponding adjacency matrix, we immediately obtain the following result.

\begin{proposition}(see [~\cite{Kamei-2013}, Proposition 3.5])
Let $\mathcal{G}$ be a coupled cell network associated with the admissible map $f$. Let $A$ be the adjacency matrix of $\mathcal{G}$. $\bowtie$ is balanced if and only if $A(\Delta_{\bowtie})\subseteq \Delta_{\bowtie}$ where $\Delta_{\bowtie}$ is a synchrony subspace associated with $\bowtie$.
\end{proposition}

\subsection{Quotients}
\label{subsec:quotinets}
Any balanced coloring $\bowtie$ on a regular network $\mathcal{G}$ determines a quotient network $\mathcal{G}_{\bowtie}$. A quotient network $\mathcal{G}_{\bowtie}$ is a regular network with a valency $r$ if, and only if, the original network $\mathcal{G}$ is a regular network with a valency $r$.  The set of cells of $\mathcal{G}_{\bowtie}$ is formed by one cell of each color and the edges in the quotient network are the projection of edges in the original network. The dynamics of $\mathcal{G}_{\bowtie}$ corresponds to synchronous dynamics of $\mathcal{G}$; that is, dynamics restricted to $\Delta_{\bowtie}$. 

\subsection{Young Tableau}
Suppose $A$ consists of $m$ Jordan blocks with sizes $k_{1},\ldots, k_{m}$, where $m\leq n$. Young tableau is a way to represent the partitions of a positive integer $k$, i.e., different ways to write $k$ as a a sum of positive integers:
\[k=k_{1}+k_{2}+\cdots + k_{m}.\]
The Young tableau (diagram) of $(k_{1}, k_{2},\ldots,k_{m})$ is a diagram of $m$ rows of square boxes with $k_{i}$ boxes on the $i$-th row for $i=1,2,\ldots,m$. For example Young tableau of $(4,2,2,1)$ is shown in Figure \ref{fig:Young_Tableau}.

\begin{figure}[h!]
\begin{center}
\includegraphics[scale=0.3]{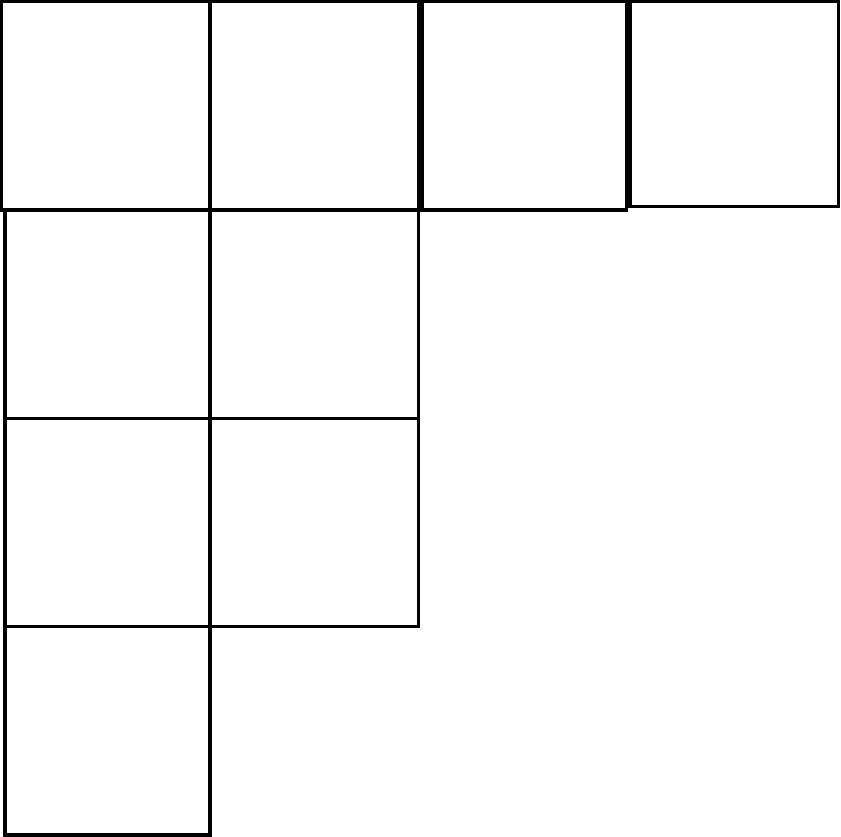}
\end{center}
\caption{Young tableau of $(4,2,2,1)$.}
\label{fig:Young_Tableau}
\end{figure}

\section{Lattice of Integer Tuples $\mathcal{L}_{M}$}
\label{sec:LM}
We construct a lattice $\mathcal{L}_{M}$, which is uniquely determined by positive integers $k_{1},\ldots, k_{m}$. We associate a lattice $\mathcal{L}_{M}$ with a lattice of synchrony subspaces of an $n$-cell regular network, where the Jordan normal {form} of the adjacency matrix $A$ has $m$ blocks with sizes $k_{1},\ldots, k_{m}$ with $m\leq n$. 

Let $L$ be a lattice of synchrony subspaces of a regular network with simple eigenvalues. We show that $L\subseteq\mathcal{L}_{M}$ with $r\wedge s\in L$ for all $r,s\in L$, termed \textit{closed subset} of $\mathcal{L}_{M}$. In Subsection\ref{subsec:PA_reduction}, we construct a reduced structure of lattice of synchrony subspaces of a regular network with non-simple eigenvalues. Lemma \ref{lem:closed_PA} shows that such a reduced structure for non-simple eigenvalue case is also represented as a closed subset of $\mathcal{L}_{M}$, which leads to the conclusion that a lattice of synchrony subspaces of a regular network with non-simple eigenvalue can be reduced into a lattice structure of simple eigenvalue case.

\subsection{Construction of $\mathcal{L}_{M}$}
\label{subsec:LM}

\begin{Def}
\rm
\label{def:LM}  
Let $\bn$ be the set of positive integers and $m\in \bn$. Consider $m$ positive integers $k_1,\dots, k_m\in \bn$ and define 
\begin{equation*}
\cL_M(k_1,\dots, k_m):=\{(r_1,\dots, r_m): 0\le r_i\le k_i,\,\forall\, i\}
\end{equation*}
together with a partial order
\begin{equation}\label{eq:LM_po}
(r_1,\dots, r_m)\le (s_1,\dots, s_m) \q \Longleftrightarrow\q r_i\le s_i\q\forall\, i. 
\end{equation}
The meet ``$\wedge$'' and the join ``$\vee$'' operations are given by taking minimum and maximum, respectively. That is,
\begin{align}
&(r_1,\dots, r_m)\wedge (s_1,\dots, s_m)=(\min(r_1,s_1),\dots, \min(r_m,s_m))\label{eq:tuple_wedge} \\ 
&(r_1,\dots, r_m)\vee (s_1,\dots, s_m)=(\max(r_1,s_1),\dots, \max(r_m,s_m)). \label{eq:tuple_vee}
\end{align}
The set $\cL_M(k_1,\dots, k_m)$ equipped with (\ref{eq:tuple_wedge})--(\ref{eq:tuple_vee}) is called the {\it lattice of $m$-tuples} of non-negative integers bounded by $k:=(k_1,\dots, k_m)$. 
\END
\end{Def}
 
\subsection{Index on $\mathcal{L}_{M}$} 
For two elements $r,s\in \cL_M(k_1,\dots, k_m)$ such that $r\le s$ and $r\ne s$, we write $r<s$ and call $s$ to be a {\it follower} of $r$, or equivalently, we call $r$ to be a {\it leader} of $s$. If there are no elements $t$ being different from $r,s$ such that $r<t<s$, then we say $s$ is an {\it immediate follower} of $r$ and $r$ is an {\it immediate leader} of $s$. For $r=(r_1,\dots, r_m)\in \cL_M(k_1,\dots, k_m)$, denote by 
$$|r|:=\sum_{i=1}^m r_i.$$

\begin{Def}
\rm 
\label{def:ind} 
An index ``$\ind$'' can be defined on $\cL_M(k_1,\dots, k_m)$ using ``$|\cdot|$''.  For $s\in \cL_M(k_1,\dots, k_m)$ and its immediate leaders $r_1,\dots, r_n$, define 
\[\ind s=|s|-|r_1\vee \cdots \vee r_n|, \q \forall\, s\le k \]
\END
\end{Def}


\begin{lem}
\label{lem:ind}
\begin{itemize}
\item[(i)] $\ind s=\begin{cases}1,\q\text{if $s$ has one non-zero component}\\ 0,\q\text{otherwise}\end{cases}$;
\item[(ii)] $\underset{s\le a}{\sum} \ind s=|a|$, $\forall\, a\le k$.
\end{itemize}
\end{lem} 

\begin{proof} (i) Let $s=(s_1,\dots, s_m)\in \cL_M(k_1,\dots, k_m)$. If $s$ has one non-zero component, say $s_i$, then $s$ has a unique immediate leader $(s_1,\dots, s_i-1,\dots, s_m)$. Thus, $\ind s=|s|-|(s_1,\dots, s_i-1,\dots, s_m)|=1$. If $s$ has no non-zero components, then  $s=(0,\dots, 0)$, so $\ind s =0$. Otherwise, assume without loss of generality that $s$ has two non-zero components $s_i$ and $s_j$. Then, $s$ has two immediate leaders:
\[r:=(s_1,\dots,s_i-1,\dots, s_m),\q  t:=(s_1,\dots,s_j-1,\dots, s_m).\]
It follows that $r\vee t=s$ and $\ind s=|s|-|r\vee t|=0$.

\noi (ii)  Let $a:=(a_1,\dots, a_m)$. There are precisely $|a|$ elements having one non-zero components such that $s\le a$. They are $(1,0,\dots, 0),$ $(2,0,\dots, 0)$, $\dots$, $(a_1,0,\dots, 0)$, $\dots$, $(0,0,\dots, 1)$, $(0,0,\dots, 2)$, $\dots$, $(0,0,\dots, a_m)$. The statement (ii) follows from (i).
\qed
\end{proof}

\subsection{Closed Subset of $\mathcal{L}_{M}$}
\label{subsec:closed_subset}
\begin{Def}
\rm 
\label{def:ind_L} 
Let $L=\cL_M(k_1,\dots, k_m)\setminus \{r_1,\dots, r_n\}$ be a subset. For $s\in L$, denote by $F_s$ the set of $r\in\{r_1,\dots, r_n\}$ where $s$ is an immediate follower of $r$. Define an {\it index on $L$} by
 \begin{equation}
 \label{eq:ind_L} 
 \Ind s= \ind s+\underset{r\in F_s}{\sum} \ind r,  \q \forall\, s\in L.
 \end{equation}
\END
\end{Def}

\begin{Def}
\rm 
\label{def:closed}
A subset $L\subset \cL_M(k_1,\dots, k_m)$ is called {\it closed}, if $(k_1,\dots, k_m)\in L$ and $r\wedge s\in L$ whenever $r,s\in L$. 
\END
\end{Def}

\begin{lem}
\label{lem:rem_cl}
Let $L_1,L_2\subset \cL_M(k_1,\dots, k_m)$ be two closed subsets such that $L_2=L_1\setminus \{r_1,\dots, r_n\}$ for some $r_1,\dots, r_n\in \cL_M(k_1,\dots, k_m)$. Then, every $r_i$ has a unique immediate follower in $L_1$ that is also contained in $L_2$. 
\end{lem}
\begin{proof} Let $r\in\{r_1,\dots,r_n\}$. Since $L_2$ is closed, we have $(k_1,\dots, k_m)\in L_2$, thus $r$ has at least a follower that is contained in $L_2$ such as $(k_1,\dots, k_m)$. Assume that $s,t\in L_2$ are two distinct immediate followers of $r$ that are contained in $L_2$. Then, $r< s$, $r<t$ and $r\le s\wedge t< s$.  Since $s$ is assumed to be an immediate follower of $r$, we have $r=s\wedge t$. This contradicts the fact that $L_2$ is closed, since $s,t\in L_2$ but $r\not\in L_2$.
\qed
\end{proof}

\begin{lem}
\label{lem:indL} 
Let $L\subset \cL_M(k_1,\dots, k_m)$ be a closed subset and ``$\Ind$'' be the index on $L$. Denote by $k=(k_1,\dots, k_m)$. Then,
\begin{itemize}
\item[(i)] $\Ind s\ge 0$;
\item[(ii)] ${\underset{s\in L,s\le a}{\sum}} \Ind s=|a|$, $\forall\, a\le k$.
\item[(iii)] $\underset{s\in L}{\sum} \Ind s=|k|$.
\end{itemize}
\end{lem}
\begin{proof} By Lemma \ref{lem:ind} (i) and (\ref{eq:ind_L}), we have the non-negativity of $\Ind s$ for all $s\in L$. To show (ii), notice that $\cL_M(k_1,\dots, k_m)$ is closed. Suppose that $L=\cL_M(k_1,\dots, k_m)\setminus \{r_1,\dots, r_n\}$. By Lemma \ref{lem:rem_cl}, every $r_i$ has a unique immediate follower $s_i$ in $L$. Let 
\[S:=\{s_i: i=1,\dots, n\},\]
which is also the set of all elements of $L$ that have immediate leaders in $\{r_1,\dots, r_n\}$. Thus, $\Ind s\ne \ind s$ if and only if $s\in S$. Also, note that
\[\{r_1,\dots, r_m\}=\underset{s\in S}{\sqcup} \big( F_s\cap \{r_1,\dots, r_m\}\big)=\underset{s\in S}{\sqcup} F_s\]
is a partition of the set $\{r_1,\dots, r_m\}$. Thus, we have
\begin{align*}
&\underset{s\in S,s\le a}{\sum} \Ind s=\underset{s\in S,s\le a}{\sum} \big(\ind s + \underset{r\in F_s}{\sum} \ind r\big)\\
&=\underset{s\in S,s\le a}{\sum} \ind s +\underset{r \in \{r_1,\dots, r_n\},r\le a}{\sum} \ind r =\underset{s\in L,s\le a}{\sum} \ind s.
\end{align*}
It follows that $\underset{s\in L,s\le a}{\sum} \Ind s=\underset{s\le a}{\sum} \ind s=|a|$ by Lemma \ref{lem:ind} (ii).

The statement (iii) follows from (ii) as a special case.
\qed
\end{proof}

We show that the index defined by (\ref{eq:ind_L}) can be alternatively  computed using immediate leaders. 
\begin{lem}
\label{rmk:ind_P_def2}
Let $L$ be a closed subset of $\mathcal{L}_{M}(k_{1},\ldots,k_{m})$. Then, the index defined by (\ref{eq:ind_L}) satisfies
\begin{equation}\label{eq:ind_L2}
\textrm{Ind}\ s=|s|-|r_{1}\vee \cdots \vee r_n|, \q \forall\, s\le k,
\end{equation}	
where $k:=(k_{1},\ldots,k_{m})$ and $r_{1},\ldots,r_{n}$ are immediate leaders of $s$.
\end{lem}
\begin{proof}
Assume first that $L=\mathcal{L}_{M}(k_{1},\ldots,k_{m})\setminus\{r\}$. If $s$ is not an immediate follower of $r$, then both (\ref{eq:ind_L}) and (\ref{eq:ind_L2}) lead to  $\Ind s=\ind s$, thus coincide. Otherwise, if $s$ is an immediate follower of $r$, then assume that $r, r_2,\dots, r_j$ are all the distinct immediate leaders of $s$. It follows that $(r\wedge r_2),\dots, (r\wedge r_j)$ are necessarily immediate leaders of $r$ in $\mathcal{L}_{M}(k_{1},\ldots,k_{m})$. There can be additional immediate leaders of $r$, say $t_1,\dots, t_l$, in $\mathcal{L}_{M}(k_{1},\ldots,k_{m})$. Then, in the subset $L$, $s$ has $t_1,\dots, t_l$ and $r_2,\dots, t_j$ as immediate leaders.
	By (\ref{eq:ind_L}), we have
	\begin{equation}\label{eq:ind_L_pf}
	\Ind s=\ind s+\ind r=|s|-|r\vee r_2\vee \dots \vee r_j|+|r|-|t_1\vee\dots \vee t_l\vee(r\wedge r_2)\vee\dots \vee(r\wedge r_j)|
	\end{equation}
	By (\ref{eq:ind_L2}), we have
	\begin{equation}\label{eq:ind_L_pf2}
	\Ind s=|s|-|t_1\vee \dots\vee t_l\vee r_2\vee\dots \vee r_j|.
	\end{equation}
	Notice that $r\vee t_i=r$ and $t_i=r\wedge t_i$ for $i=1,\dots, l$ and generally it holds for $m$-tuples $s_1,s_2$ that
	\begin{equation}\label{eq:prop_wedge}
	|s_1|+|s_2|=|s_1\vee s_2|+|s_1\wedge s_2|.
	\end{equation}
	By letting $s_1=r$, $s_2=t_1\vee\dots\vee t_l\vee r_2\vee\dots\vee r_j$, we obtain that (\ref{eq:ind_L_pf}) coincides with (\ref{eq:ind_L_pf2}). For general subset $L$, one can successively repeat the above argument.	
\qed
\end{proof}

\section{Poset of Integer Tuples $\mathcal{P}_{A}$}
\label{sec:PA}
We describe how an invariant subspace of a given linear transformation can be represented by tuple representation and how tuple representations can be assigned to lattice nodes of $V_{\mathcal{G}}^{P}$ using the eigen-structure of the adjacency matrix $A$ (Proposition \ref{prop:map_T} and Corollary \ref{cor:map_T_unique}).  This leads to a construction of an order-preserving poset $\mathcal{P}_{A}$ to a lattice of synchrony subspaces $V_{\mathcal{G}}^{P}$. 

If we assign tuple representations to synchrony subspaces of regular networks with simple eigenvalues, then all tuple representations are distinct. Such networks have all lattice nodes with non-negative index and the sum of indices equals to the size of the network, which enabled us to enumerate all possible lattice structures for a given network size and, furthermore, positive lattice indices were used to predict the existence of synchrony-breaking bifurcating branches {\cite{Kamei-2009-part1, Kamei-2009-part2}. 
	
However, the situation for regular networks with non-simple eigenvalues is different, since some nodes of $\mathcal{P}_{A}$ might have identical tuple representations when the geometric multiplicity of a repeated eigenvalue is more than one. Thus, we generalise the concept of the index in Definition \ref{def:ind_P},  Lemma \ref{lem:closed_PA} and  Lemma \ref{lem:ind_PA}. We then construct a reduced lattice $\mathcal{P}_{A}/{=}$, where nodes with the same tuple representation are merged into one lattice node while preserving the connectivity associated with the partial order among tuple representation. 

{\subsection{Tuple Representation of Invariant Subspaces}
Let $A$ be a linear transformation acting on a finite-dimensional vector space $P$. We consider an $m$-tuple {of} non-negative {integer} representation of an invariant subspace $M$ of $P$ such that $AM\subset M$ as discussed in \cite{Fillmore-1977,Longstaff-1984}. Let $k_{1}\geq k_{2}\geq\cdots\geq k_{m}$ be the size of the Jordan blocks of $A$. We associate an {$m$-tuple} of non-negative integers {$(r_{1}, \ldots,r_{m})$} with $0\leq r_{i}\leq k_{i}$ to the invariant subspace $M$ if $M$ can be spanned by the first $r_{i}$ (generalised) eigenvectors from each Jordan block with size $k_{i}$ for $i=1,2,\ldots,m$.

\begin{ex}
\rm
Consider the following $9\times 9$ Jordan normal form $A$:
\begin{displaymath}
A=
\left(\begin{tabular}{ccccccccc}
\cline{1-4}
  \multicolumn{1}{|c}{$\lambda_{1}$} & 1 & 0 & \multicolumn{1}{c|}{0} & 0 & 0 & 0& 0 & 0 \\
  \multicolumn{1}{|c}{0} & $\lambda_{1}$ & 1 & \multicolumn{1}{c|}{0} & 0 & 0 & 0 & 0& 0 \\
  \multicolumn{1}{|c}{0} & 0 & $\lambda_{1}$ & \multicolumn{1}{c|}{1} & 0 & 0 & 0 & 0 & 0 \\
  \multicolumn{1}{|c}{0} & 0 & 0 & \multicolumn{1}{c|}{$\lambda_{1}$} & 0 & 0 & 0 & 0 & 0 \\
 \cline{1-6}
 0 & 0 & 0 & 0 &  \multicolumn{1}{|c}{$\lambda_{2}$} & \multicolumn{1}{c|}{1} & 0 & 0 & 0 \\
 0 & 0 & 0 & 0 &  \multicolumn{1}{|c}{0} & \multicolumn{1}{c|}{$\lambda_{2}$} & 0 & 0 & 0  \\
\cline{5-8}
0 & 0 & 0 & 0 & 0 & 0 & \multicolumn{1}{|c}{$\lambda_{3}$} & \multicolumn{1}{c|}{1} & 0 \\
0 & 0 & 0 & 0 & 0 & 0 & \multicolumn{1}{|c}{0} & \multicolumn{1}{c|}{$\lambda_{3}$} & 0  \\
\cline{7-9}
0 & 0 & 0 & 0 & 0 & 0 & 0 & 0 & \multicolumn{1}{|c|}{$\lambda_{4}$} \\
\cline{9-9}
\end{tabular}\right)
=
\left(\begin{tabular}{cccc}
\cline{1-1}
  \multicolumn{1}{|c|}{$J_{1}$} & 0 & 0 &  0\\
  \cline{1-2}
 0 & \multicolumn{1}{|c|}{$J_{2}$}  & 0 & 0 \\
 \cline{2-3}
 0 & 0 & \multicolumn{1}{|c|}{$J_{3}$} & 0  \\
 \cline{3-4}
 0 & 0 & 0 & \multicolumn{1}{|c|}{$J_{4}$} \\
 \cline{4-4}
\end{tabular}\right),
\end{displaymath}
where some eigenvalues can be repeated, i.e., $\lambda_{i}=\lambda_{j}$ for some {$i\neq j$}. This Jordan normal form has one Jordan block $J_{1}$ of size $4$, two Jordan blocks $J_{2}$ and $J_{3}$ of size $2$, and one Jordan block $J_{4}$ of size $1$.

Let {$e_{1},e_{2},\ldots,e_{9}$} be linearly independent (generalised) eigenvectors of $A$ with the following Jordan chain structure:
\begin{center}
\begin{tabular}{cccc}
 \multicolumn{1}{r}{$J_1:$} & \multicolumn{1}{l}{$0\leftarrow e_1\leftarrow e_2 \leftarrow e_{3} \leftarrow e_{4}$}\\
 \multicolumn{1}{r}{$J_2:$} & \multicolumn{1}{l}{$0\leftarrow e_5 \leftarrow e_{6}$} \\
 \multicolumn{1}{r}{$J_3:$} & \multicolumn{1}{l}{$0\leftarrow e_7 \leftarrow e_{8}$}\\
 \multicolumn{1}{r}{$J_4:$} & \multicolumn{1}{l}{$0\leftarrow e_{9}$}\\
\end{tabular}
\end{center}
Note that $e_{1}, e_{5}, e_{7}$ and $e_{9}$ are eigenvectors and the rest of $e_{i}$ are generalised eigenvectors.

We use Young tableau to represent an invariant subspace of $A$. In the Young tableau, each row represents one Jordan block, and the number of boxes represent the dimension. We list each block according to its size in a descending order. Suppose $M=\la e_{1},e_{2},e_{3},e_{5},e_{6},e_{9}\ra$. Then we can represent $M$ with the $4$-tuples $(3,2,0,1)$, which can be represented by the Young tableau as shown in Figure \ref{fig:Young_Tableau_Filled}. Black boxes in the Young tableau illustrate the (generalised) eigenvectors from each Jordan block, which are spanning $M$.
\begin{figure}[h!]
\begin{center}
\includegraphics[scale=0.3]{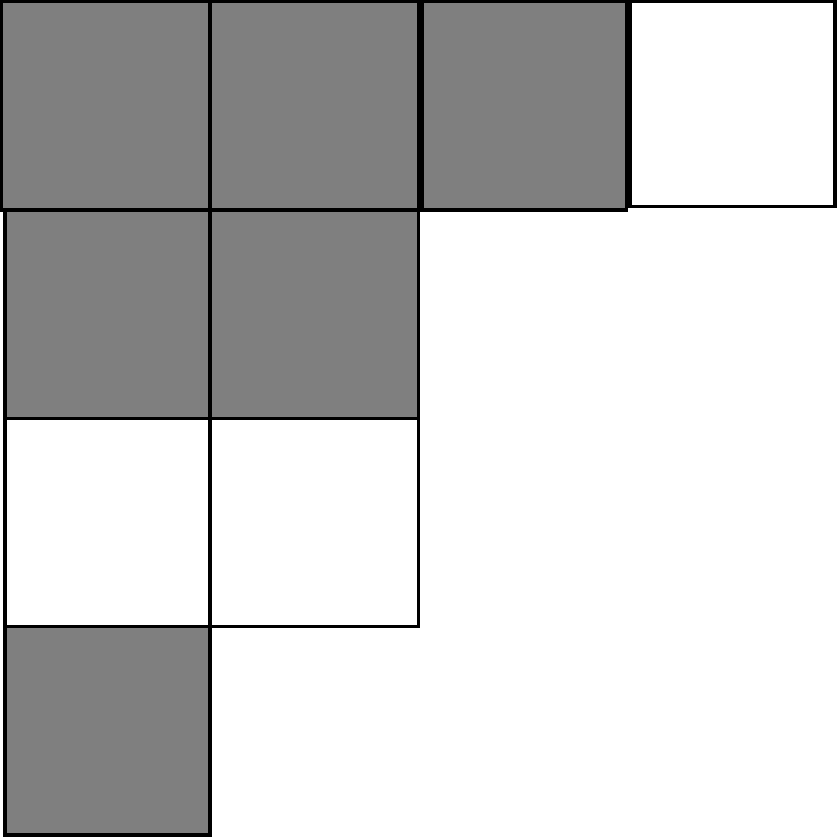}
\end{center}
\caption{Young tableau of the $4$-tuples $(3,2,0,1)$ which represents an invariant subspace $M=\la e_{1},e_{2},e_{3},e_{5},e_{6},e_{9}\ra$.}
\label{fig:Young_Tableau_Filled}
\end{figure}
\END	
\end{ex}

\subsection{Construction of $\mathcal{P}_{A}(k_{1},\ldots,k_{m})$}
\begin{Def}
\rm
\label{def:P_A}  
Let $A:\bc^n\to\bc^n$ be the adjacency matrix of a regular network $\mathcal{G}$. Let $J_{1},\ldots, J_{m}$ be the Jordan blocks of $A$ with  sizes $k_{1},\ldots, k_{m}$, respectively. Define {a multiset}
\begin{equation*}
\mathcal{P}_{A}(k_1,\dots, k_m):=\{(r_1,\dots, r_m): 0\le r_i\le k_i,\,\forall\, i\}
\end{equation*}
together with {a strict partial order
\begin{equation}
\label{eq:partial_order_PA}
(r_1,\dots, r_m) < (s_1,\dots, s_m) \q \Longleftrightarrow\q r_i \leq s_i\q\forall\, i \quad\textrm{and}\quad \exists\, i\textrm{ such that }r_{i}\neq s_{i}.
\end{equation}
}
The meet ``$\wedge$'' and the join ``$\vee$'' operations are given by (\ref{eq:tuple_wedge})--(\ref{eq:tuple_vee}).
The set $\mathcal{P}_{A}(k_1,\dots, k_m)$ equipped with (\ref{eq:tuple_wedge})--(\ref{eq:tuple_vee}) is called the {\it poset of $m$-tuples} of non-negative integers. 
\END
\end{Def}

\begin{proposition}
\label{prop:map_T} 
Let $A:\bc^n\to\bc^n$ be an adjacency matrix of a regular network $\mathcal{G}$ with $m$ Jordan blocks. There exists {a one-to-one} order-preserving map $\cT$ between the lattice of synchrony subspaces $V_{\mathcal{G}}^{P}$ and $\cP_A(k_{1},\ldots,k_{m})$ such that
\begin{align}
\cT:\q V_{\mathcal{G}}^{P}\q &\to \q\cP_A(k_1,\dots, k_m)\notag\\
S\q &\mapsto \q (r_1,\dots, r_m),\label{eq:map_dis}
\end{align}
where $k_i$ is the size of the $i$-th Jordan block of $A$ for $i=1,\dots, m$.
\end{proposition}

\begin{proof}
Let $S_1\subset S_2$ be two synchrony subspaces of dimension $d_{1}$ and $d_{2}$, respectively. We want to show that there exists a map $\mathcal{T}$ such that $\mathcal{T}(S_{1})<\mathcal{T}(S_{2})$. We denote $\mathcal{T}(S_{1})=(r_{1}^{1},\ldots,r_{m}^{1})$ and $\mathcal{T}(S_{2})=(r_{1}^{2},\ldots,r_{m}^{2})$.

Let $A:\bc^n\to\bc^n$ be {the} adjacency matrix of a regular network $\mathcal{G}$ with distinct eigenvalues $\lambda_{1},\ldots,\lambda_{k}$. We decompose
\[\bc^{n}=E_{\lambda_{1}}\oplus\cdots\oplus E_{\lambda_{k}}.\]

Consider 
\[S_2\cap E_{\lambda_{i}}=\bigoplus_{j=1}^{l_{i}}J_{j}^{i},\]
where each $J_{j}^{i}$ is the span of the corresponding Jordan chain, and $l_{i}$ is the number of Jordan blocks corresponding to an eigenvalue $\lambda_{i}$ of $A$. Denote by $k_{1}\geq\cdots\geq k_{l_{i}}$ the size of Jordan blocks in decreasing order. Define a map $\mathcal{T}$ which assigns an non-negative integer $r_{j}^{2}$ to each Jordan block $J_{j}^{i}$ in the following way:
\[\mathcal{T}(S_{2}\cap E_{\lambda_{i}})=(r_{1}^{2},\ldots,r_{l_{i}}^{2}) \quad\textrm{which satisfies}\quad 0\leq r_{j}^{2}\leq k_{j} \quad\forall j=1,\ldots, l_{i}.\] 

This means that the space $S_2\cap E_{\lambda_{i}}$ can be spanned by the first $r_{1}^{2}$ vectors of the Jordan chain of the block $J_{1}^{i}$, to the first $r_{l_{i}}^{2}$ vectors of the Jordan chain of the block $J_{l_{i}}^{i}$.
Note that if $k_{c}=k_{d}$ for some $c,d$, then the assignment of non-negative integers are not unique.

Since $(S_1\cap E_{\lambda_{i}})\subset (S_2\cap E_{\lambda_{i}})$, the map $\mathcal{T}$ assigns a tuple of integers as follows:
\begin{equation}
\label{eq:order}
\mathcal{T}(S_{1}\cap E_{\lambda_{i}})=(r_{1}^{1},\ldots,r_{l_{i}}^{1}),\quad\textrm{which satisfies}\quad (r_{1}^{1},\ldots,r_{l_{i}}^{1})\leq(r_{1}^{2},\ldots,r_{l_{i}}^{2}).
\end{equation}
 
We have
\begin{eqnarray*}
\mathcal{T}(S_{1}) &=& \mathcal{T}((S_{1}\cap E_{\lambda_{1}})\oplus\cdots \oplus (S_{1}\cap E_{\lambda_{k}})) \\
&=& (r_{1}^{1},\ldots,r_{m}^{1}), \\
\mathcal{T}(S_{2}) &=& \mathcal{T}((S_{2}\cap E_{\lambda_{1}})\oplus\cdots\oplus (S_{2}\cap E_{\lambda_{k}})) \\
&=& (r_{1}^{2},\ldots,r_{m}^{2}).
\end{eqnarray*}
Since the relationship (\ref{eq:order}) is true for all eigenvalues, $(r_{1}^{1},\ldots,r_{m}^{1})\leq (r_{1}^{2},\ldots,r_{m}^{2})$. Therefore $\mathcal{T}(S_{1})\leq \mathcal{T}(S_{2})$. Since $S_{1}\neq S_{2}$, there are some $j$ such that $r_{j}^{1}<r_{j}^{2}$. Thus, there exists an order-preserving map such that if $S_{1}\subset S_{2}$ then $\mathcal{T}(S_{1})<\mathcal{T}(S_{2})$. 
\qed
\end{proof}

\begin{rmk}
\label{rmk:oder_preserv}
\rm 

\

\begin{enumerate}
\item[(i)]
A map constructed in Proposition \ref{prop:map_T} may not give the unique $\mathcal{P}_{A}$ in general. However, when Jordan blocks have distinct sizes for any given eigenvalue, the map $\mathcal{T}$ gives the unique $\mathcal{P}_{A}$ as shown in Corollary \ref{cor:map_T_unique}.

\item[(ii)]
There might be $r_{1},r_{2}\in\mathcal{P}_{A}$ with $r_{1}=r_{2}$. As a result $\cP_A(k_{1},k_{2},\ldots,k_{m})$ is a poset, but not a lattice as meet and join are not uniquely defined.

\item[(iii)]
Note that $(r_1,r_2,\dots, r_m) \mapsto S$ is generally not order-preserving. More specifically, $\cT(S_1)\le \cT(S_2)$ does not necessarily imply $S_1\subseteq S_2$. 

\item[(iv)]
There might be additional edges in $\mathcal{P}_{A}$ as a result of partial order among tuple representation, which are not in the lattice of synchrony subspaces $V_{\mathcal{G}}^{P}$. See Remark \ref{rmk:oder_preserv} (iii). These additional edges in $\mathcal{P}_{A}$ plays an important role when we apply Definition \ref{def:covering_relation}.
 
\item[(v)] In the following,  we always reserve the last {element} in the $m$-tuple for {the eigenspace $\la(1,\ldots,1)\ra$ corresponding to the valency of a regular network.}
\end{enumerate}
\END
\end{rmk}

\begin{cor}
\label{cor:map_T_unique} 
Let $A:\bc^n\to\bc^n$ be an adjacency matrix of a regular network $\mathcal{G}$ with distinct eigenvalues $\lambda_{1},\ldots,\lambda_{k}$. Let $m$ be the number of Jordan blocks of $A$. If the Jordan blocks are of distinct size for any given $i=1,\ldots,k$, then the order-preserving map $\mathcal{T}$ defined in (\ref{eq:map_dis}) assigns tuple representations to $V_{\mathcal{G}}^{P}$ uniquely. 
\end{cor}

\begin{proof}
Denote by $k_{1}>\cdots>k_{l_{i}}$ the distinct size of Jordan blocks for the eigenvalue $\lambda_{i}$. Then a map $\mathcal{T}$ assigns an non-negative integer $r_{j}^{2}$ to each Jordan block $J_{j}^{i}$ with one-to-one and onto manner since $k_{1},\ldots,k_{l}$ are all distinct. The rest of the proof is analogous to that of Proposition \ref{prop:map_T}. When the Jordan blocks are of {distinct} size for each eigenvalue,  the order-preserving map $\mathcal{T}$ assigns tuple representations uniquely with $V_{\mathcal{G}}^{P}$.
\qed
\end{proof}

\begin{ex}[Distinct Sized Jordan Blocks]
\rm
\label{ex:4_cell_network343}
Consider a $4$-cell regular network {depicted} in Figure \ref{fig:4_cell_network343} with a repeated eigenvalue with algebraic multiplicity $3$ and geometric multiplicity $2$. 
\begin{figure}[htb!]
\begin{center}
\begin{tabular}{ccll}
Network $\mathcal{G}$ & Adjacency matrix $A$ & Eigenvalues & Eigenvectors\\
\hline
\multirow{4}{*}{\includegraphics[scale=0.25]{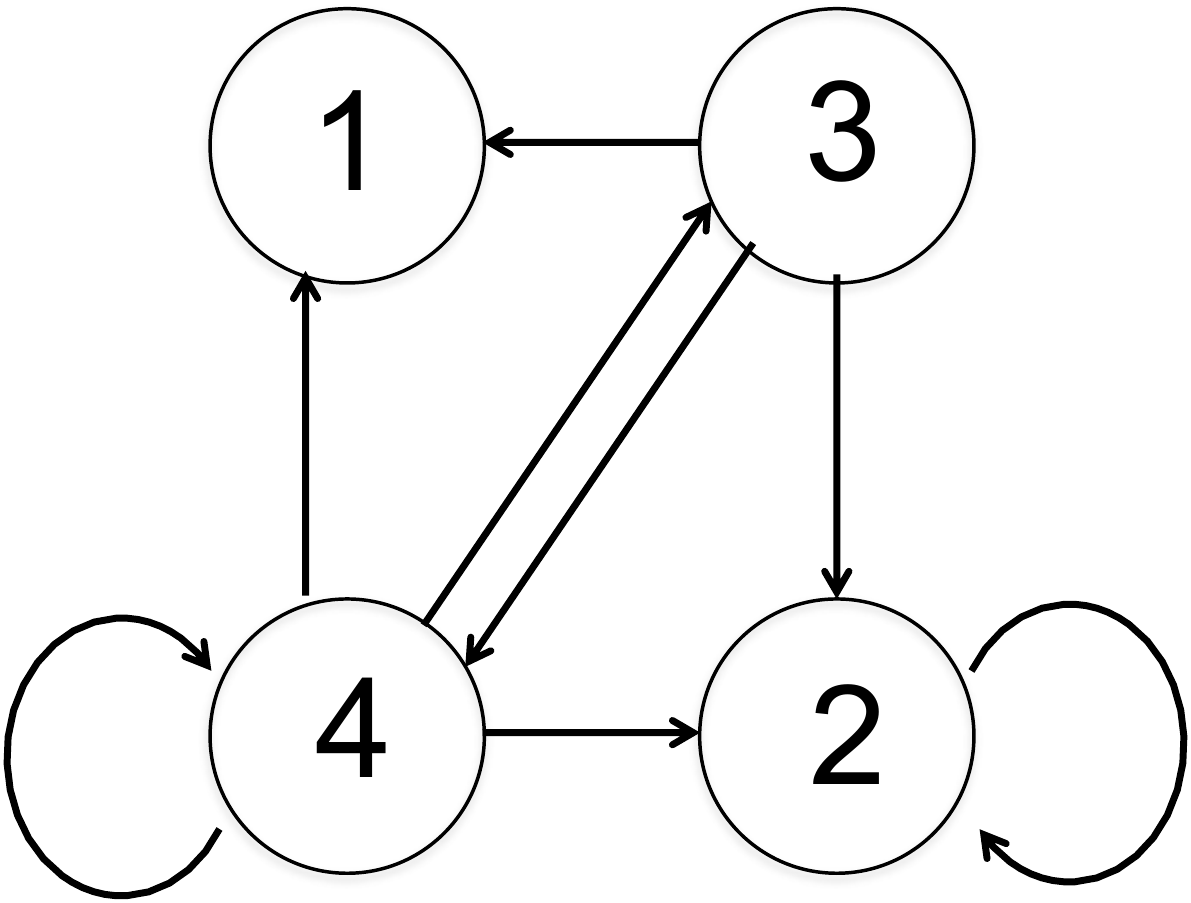}} &
\multirow{4}{*}{$\left(\begin{array}{cccc}
0 & 0 & 1 & 1\\
0 & 1 & 0 & 1 \\
0 & 1 & 0 & 1 \\
0 & 0 & 1 & 1  
\end{array}\right)$} &
$\lambda_{1}=0$ & $e_1=(1/2,-1/2,-1/2,1/2)$\\
& & $\lambda_{2}=0$ & $e_2=(-1/4,-1/4,3/4,-1/4)^{*}$\\
& & $\lambda_{3}=0$ & $e_3=(-1,0,0,0)$\\
& & $\lambda_{4}=2$  & $e_4=(1,1,1,1)$ \\
& &
\end{tabular}
\end{center}
\caption{$4$-cell regular network $\mathcal{G}$ with corresponding adjacency matrix $A$ and its eigenvalues. The repeated eigenvalue $\lambda_{1}=\lambda_{2}=\lambda_{3}=0$ has algebraic multiplicity $3$ and geometric multiplicity $2$. Note that ${e_{2}=}(-1/4,-1/4,3/4,-1/4)^{*}$ is a generalised eigenvector with head eigenvector ${e_{1}=}(1/2,-1/2,-1/2,1/2)$.}
\label{fig:4_cell_network343}
\end{figure}

The matrix $A$ has the following Jordan normal form and Jordan chains: 

\[\left(\begin{tabular}{ccc} 
\cline{1-1}
 \multicolumn{1}{|c|}{$J_1$}&& \multicolumn{1}{c}{}\\\cline{1-2}
\multicolumn{1}{c}{} & \multicolumn{1}{|c|}{$J_2$} &\\\cline{2-3}
&& \multicolumn{1}{|c|}{$J_3$} \\
\cline{3-3}
\end{tabular}\right)
=
\left(\begin{tabular}{cccc} 
\cline{1-2}
 \multicolumn{1}{|c}{$\lam_1$}& \multicolumn{1}{c|}{1} & 0 & 0\\
 \multicolumn{1}{|c}{0} &\multicolumn{1}{c|}{$\lam_2$}& 0 & 0\\ 
\cline{1-3}
0 & 0 & \multicolumn{1}{|c|}{$\lam_3$} & 0 \\
\cline{3-4}
0 & 0 & 0 & \multicolumn{1}{|c|}{$\lam_4$}  \\
\cline{4-4}
\end{tabular}\right),
\q\q\q
\begin{tabular}{cc}
 \multicolumn{1}{r}{$J_1:$}& \multicolumn{1}{l}{$0 \leftarrow e_1\leftarrow e_2 $}\\
 \multicolumn{1}{r}{$J_2:$}& \multicolumn{1}{l}{$0 \leftarrow e_3$}\\
 \multicolumn{1}{r}{$J_3:$}& \multicolumn{1}{l}{$0 \leftarrow e_4$}\\
\end{tabular}.
\]

 We show how to determine $\mathcal{P}_{A}(2,1,1)$ by representing each synchrony subspace with $3$-tuples $(r_{1},r_{2}, r_{3})$ which satisfies $0\leq r_{1}\leq 2$, $0\leq r_{2}\leq 1$ and $0\leq r_{3}\leq 1$. $r_{1}$ is associated with the $2\times 2$ Jordan block for eigenvalue $\lambda_{1}=\lambda_{2}=0$,  $r_{2}$ is associated with the $1\times 1$ Jordan block for eigenvalue $\lambda_{3}=0$, and $r_{3}$ is associated with the $1\times 1$ Jordan block for eigenvalue $\lambda_{4}=2$.

\begin{description}
\item[Step $1$:]
Figure \ref{fig:network343_quotient_jordan_young_tableau} shows the Jordan normal forms of quotient networks, and {the associated Young tableau representations.} The shaded boxes in {Young} tableau represent the cases where the choices are trivial and unique from the Jordan normal forms, {and the question marks are used for the cases we need to determine for filling.}
\begin{figure}[h!]
	\begin{center}
		\begin{tabular}{ccc}
			(a) & (b) & (c) \\
			& & \\
			\includegraphics[scale=0.35]{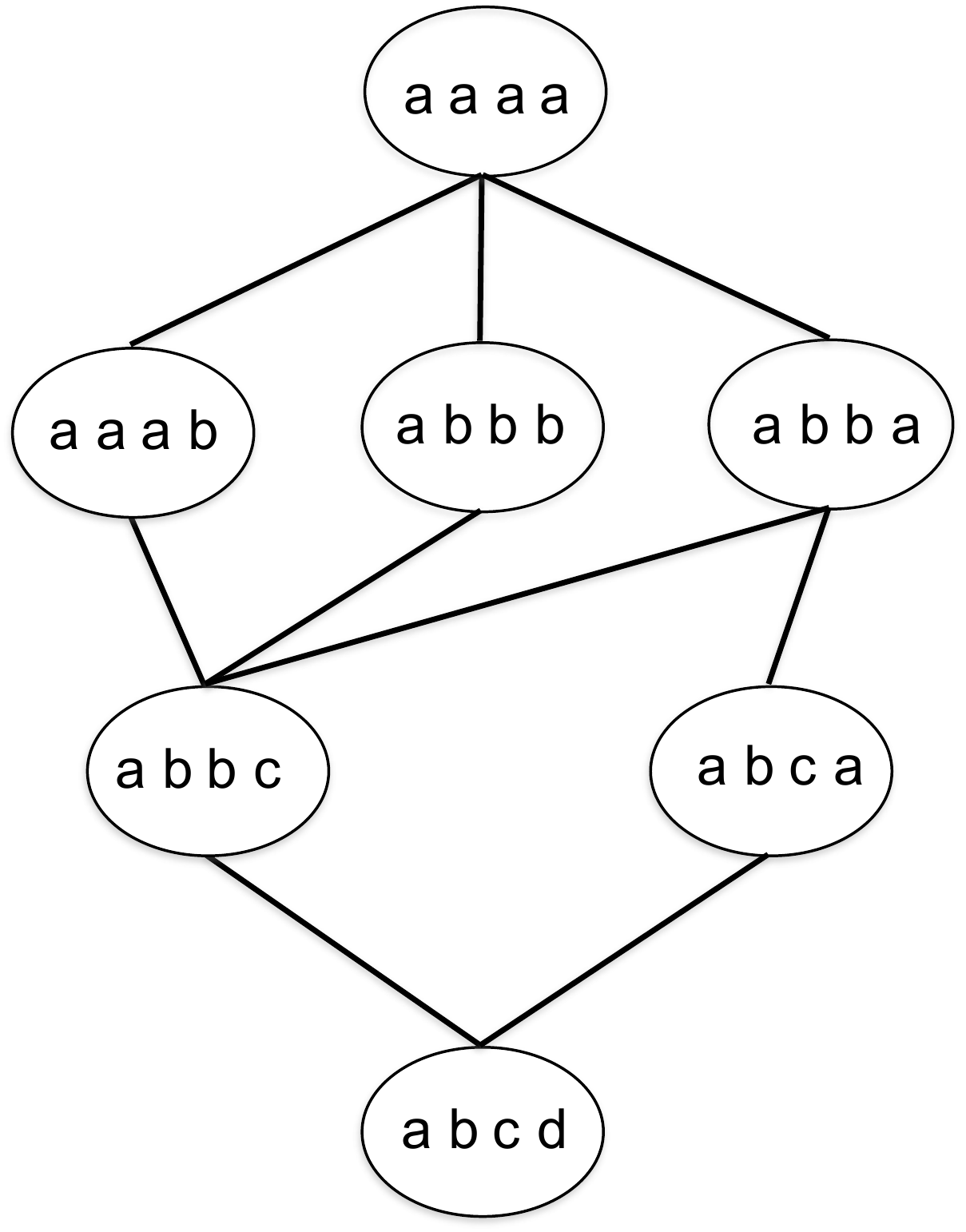} 
			&
			\includegraphics[scale=0.34]{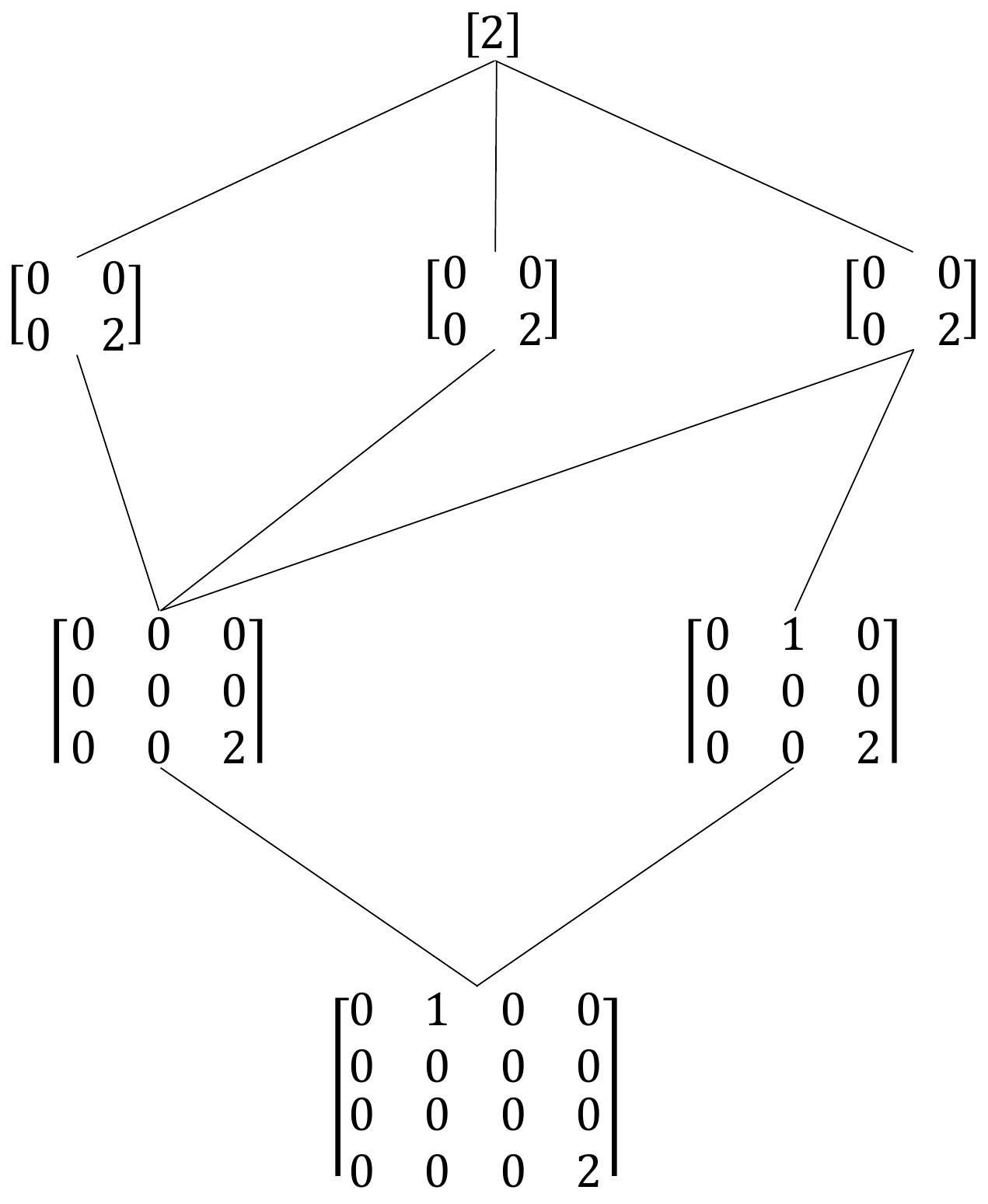} 
			&
			\includegraphics[scale=0.32]{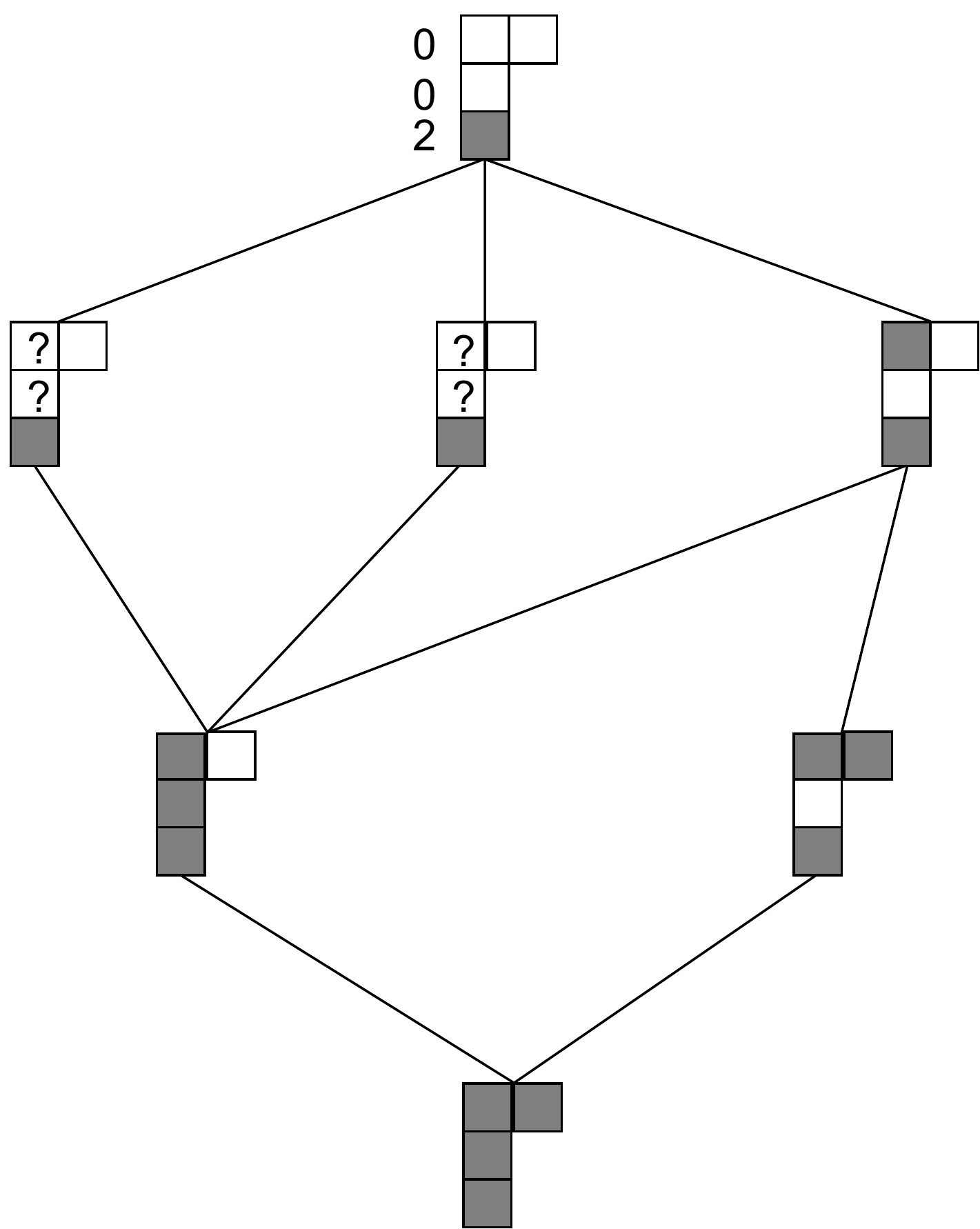}
		\end{tabular}
	\end{center}
	\caption{(a) The lattice of synchrony subspaces $V_{\mathcal{G}}^{P}$. For $(x_{1},x_{2},x_{3},x_{4})\in V_{\mathcal{G}}^{P}$, cell coordinate equality is given by the same symbol, e.g., $(x_{1},x_{2},x_{3},x_{4})=(a,b,b,a)$ means $x_{1}=x_{4}$ and $x_{2}=x_{3}$. (b) Jordan normal forms of quotient networks. (c) Young tableau representations of Jordan normal forms of quotient networks embedded in the total phase space $P$. The top $2$ boxes represents $2\times 2$ Jordan block of the eigenvalue $0$, the second top box represents $1\times 1$ Jordan block of the eigenvalue $0$, and the last box represents $1\times 1$ Jordan block of the eigenvalue $2$.}
	\label{fig:network343_quotient_jordan_young_tableau}
\end{figure}
\item[Step $2$:]
Consider possibilities to fill the remaining boxes. There are in total $4$ different possibilities to fill the remaining boxes based on the information given in the Jordan normal forms of quotient networks. 

Consider possible tuple representations of the two leaders $s_1,s_2$ of  $(1,1,1)$ (left {Young} tableau at rank $3$). For the most left node at rank $2$, {there are} two choices, either $(0,1,1)$ or $(1,0,1)$. If we choose $(0,1,1)$, then it can't be ordered with $(2,0,1)$. This also matches the connectivity of $V_{\mathcal{G}}^{P}$. If we choose $(1,0,1)$, then this is now ordered with $(2,0,1)$ as $(1,0,1)<(2,0,1)$. Similarly, the middle node at rank $2$ has two choices, either $(0,1,0)$ or $(1,0,1)$. If we choose $(1,0,1)$, then this incurs the red edge connecting with $(2,0,1)$. These $4$ possible $\mathcal{P}_{A}(2,1,1)$ can be classified into three distinct types as shown in Figure \ref{fig:network343_PA}.
\begin{figure}
	\begin{center}
		\begin{tabular}{c|cc|c}
			Type $1$ & \multicolumn{2}{c|}{Type $2$} & Type $3$ \\
			\hline
			& & & \\
			\includegraphics[scale=0.24]{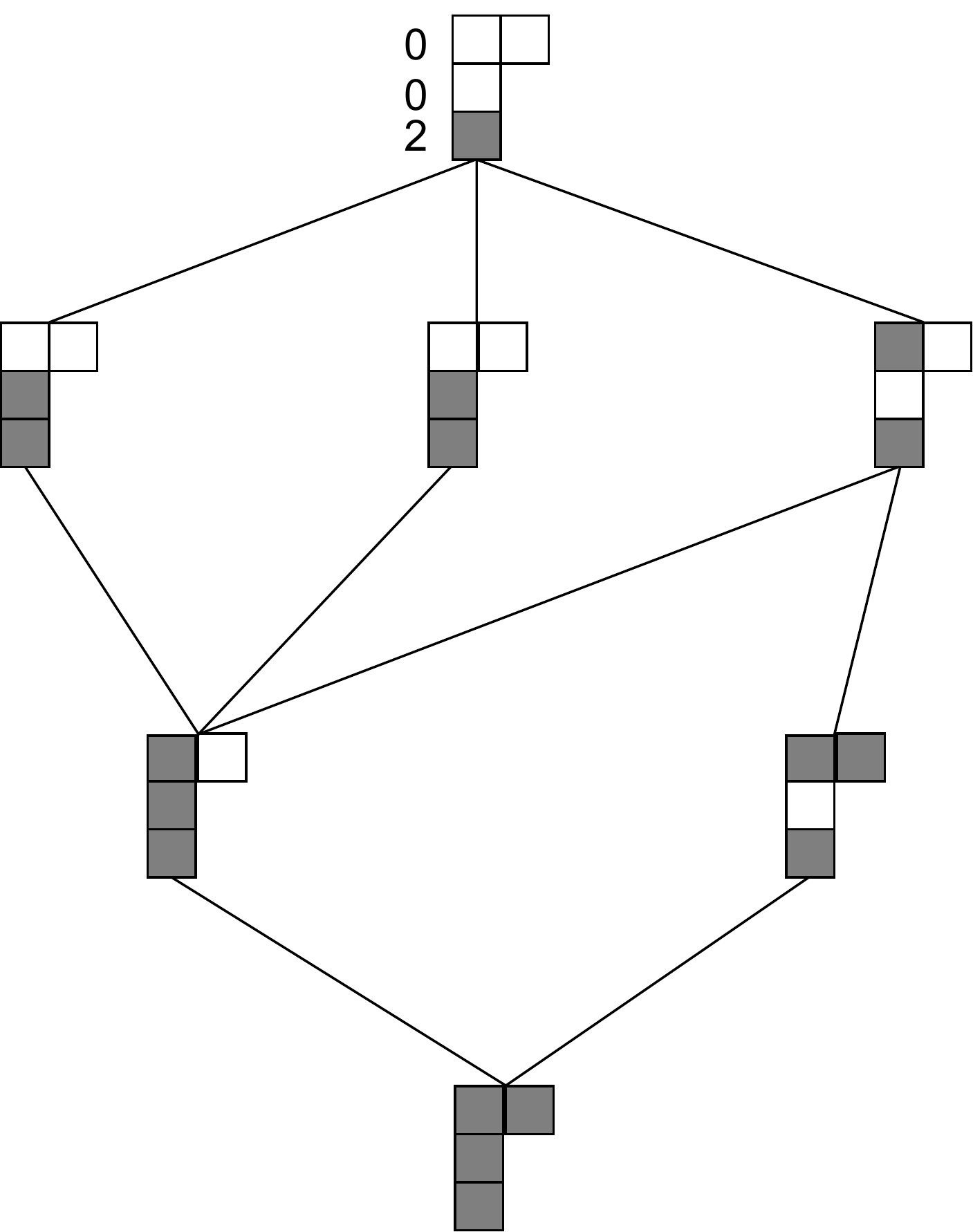} 
			& 
			\includegraphics[scale=0.24]{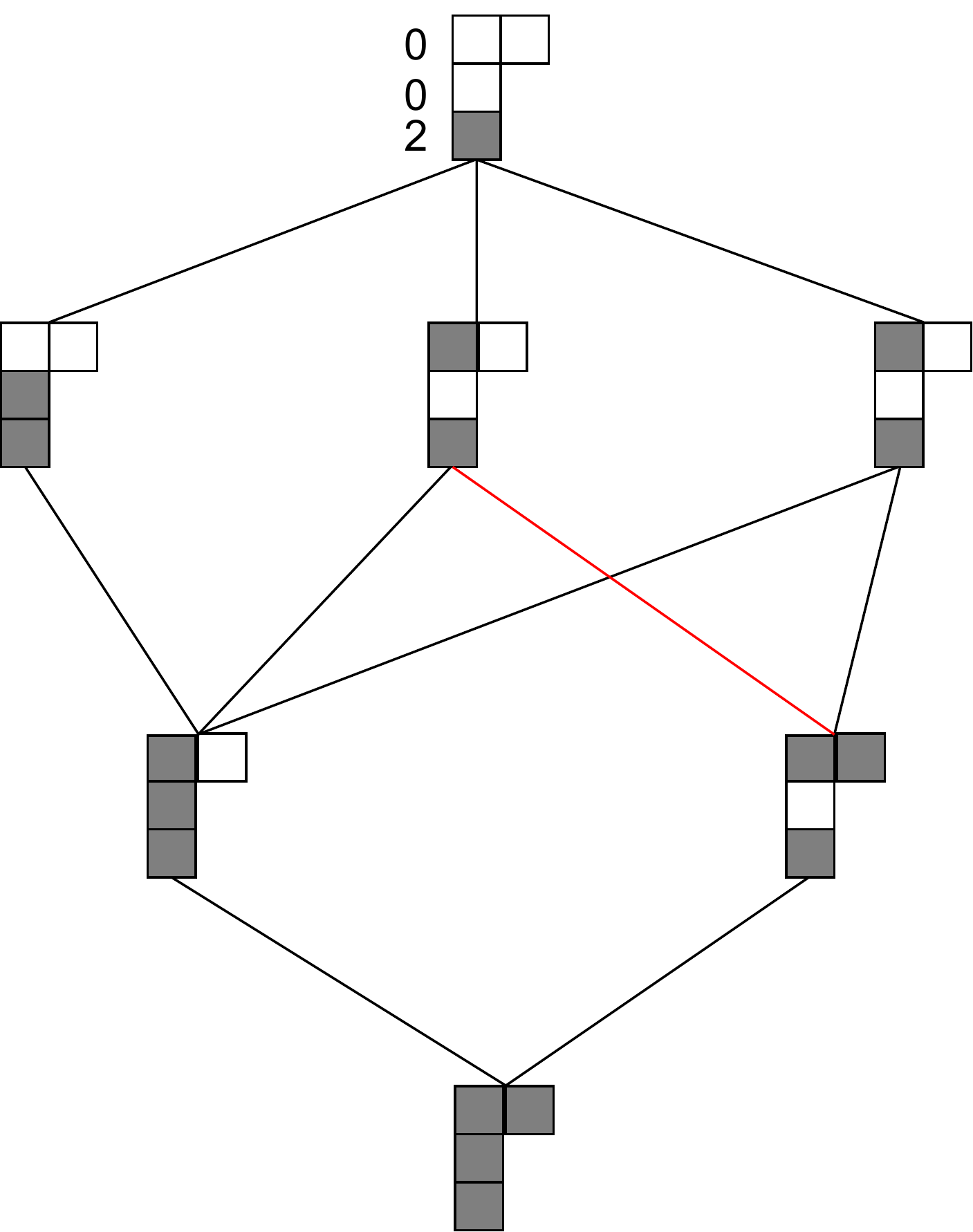} 
			& 
			\includegraphics[scale=0.24]{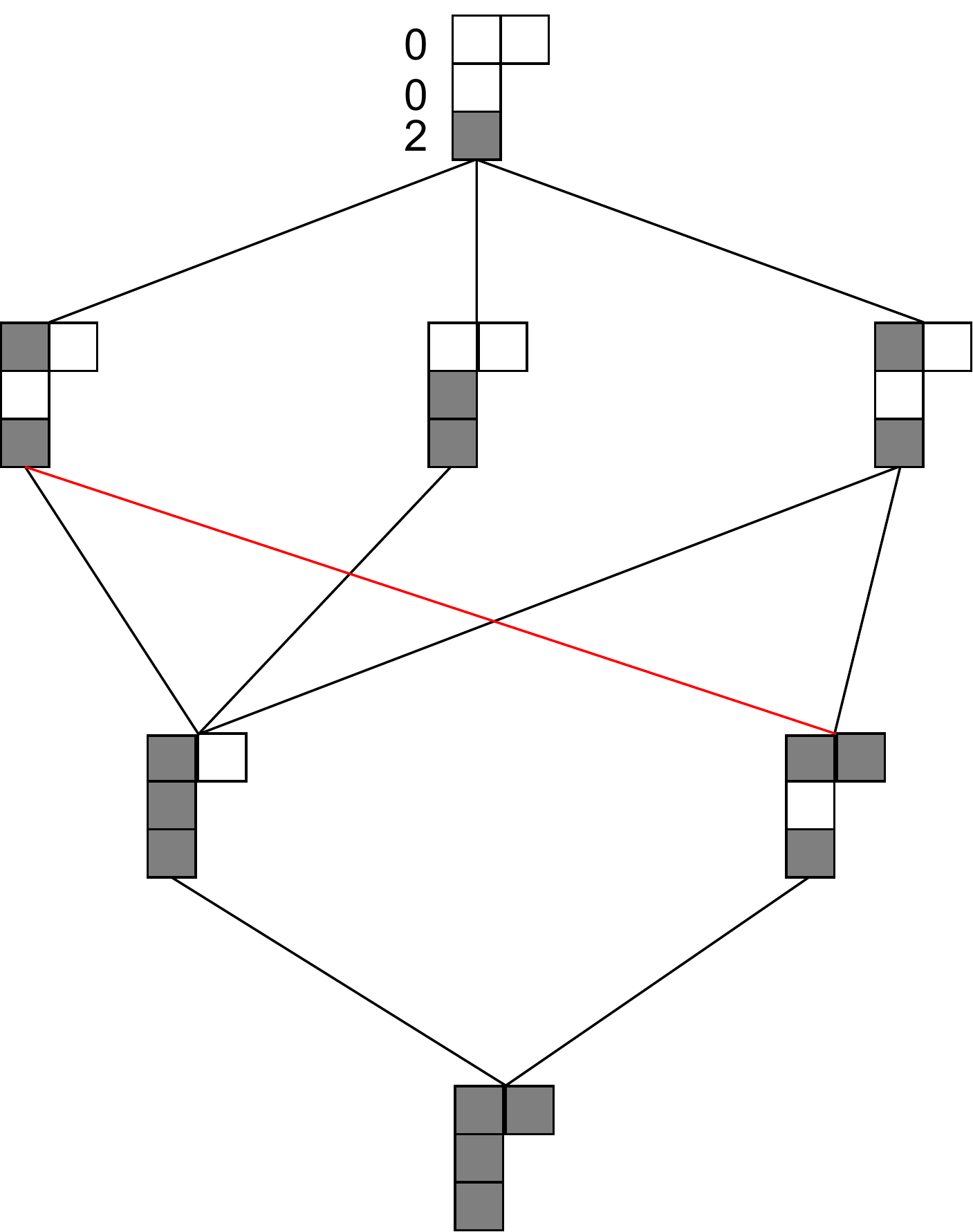} 
			&
			\includegraphics[scale=0.24]{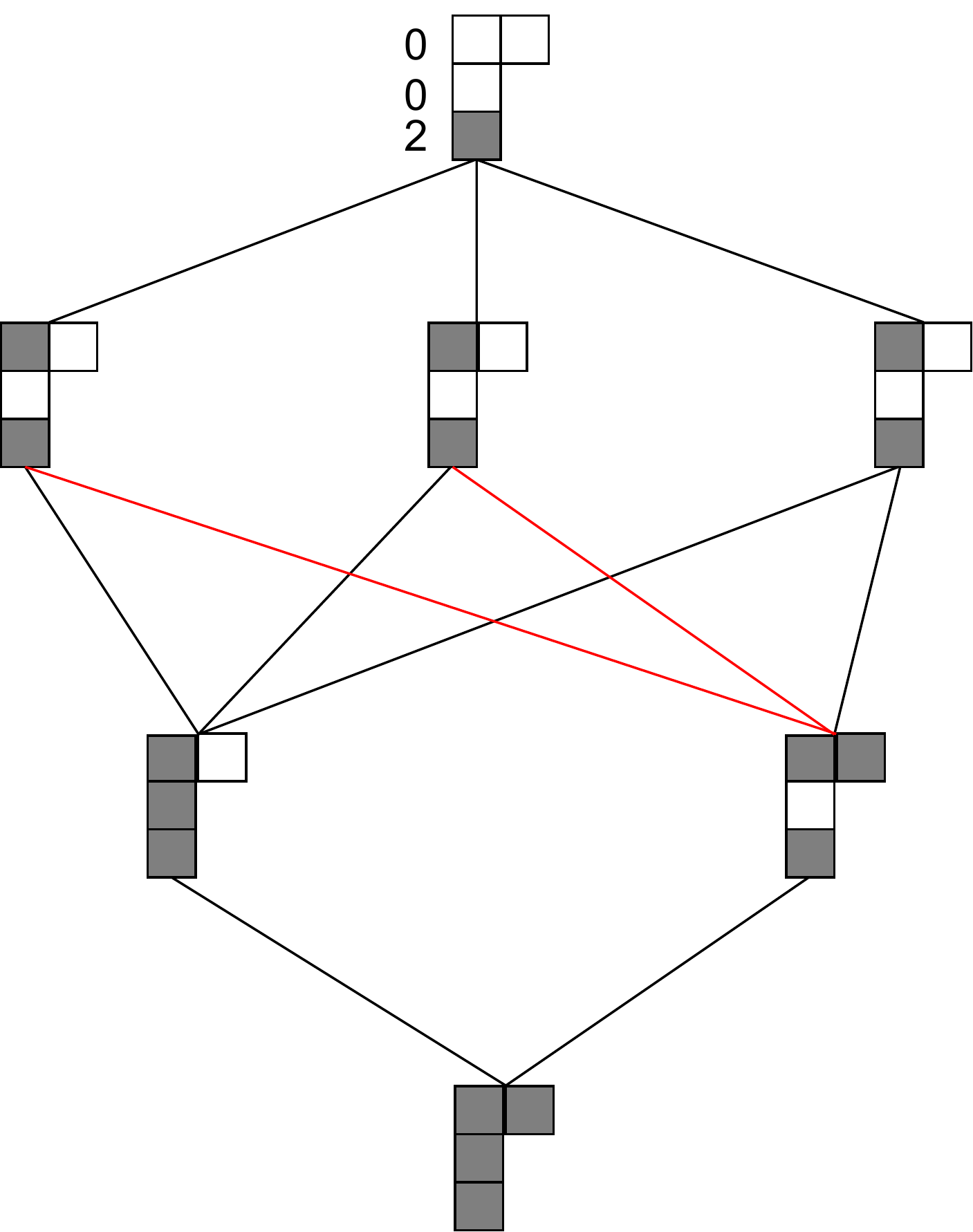}
		\end{tabular}
	\end{center}
	\caption{Type $1$ should represent the uniquely determined $\mathcal{P}_{A}(2,1,1)$. Note that Type $3$ with two additional red edges will give a positive index $\textrm{Ind}_{\mathcal{P}}(1,1,1)=1$, which will be defined in Definition \ref{def:ind_P}. On the other hand, other types give $\textrm{Ind}_{\mathcal{P}}(1,1,1)=0$.}
\label{fig:network343_PA}
\end{figure}
\item[Step $3$:] {To identify the correct type, we compute for individual nodes $s_1,s_2$ using the tuple representation map (\ref{eq:map_dis}) in Proposition \ref{prop:map_T}. More precisely, the nodes  $s_1,s_2$ correspond to the two synchrony subspaces $S_1=(a\ a\ a\ b)$,  $S_2=(a\ b\ b\ b)$ of the network. Using the eigenvectors $e_1,e_2,e_3, e_4$, we have $S_1=\la e_1+e_3,e_4\ra$ and $S_2=\la e_3,e_4\ra$. Note that both $e_3$ and $e_1+e_3$ generate a Jordan chain of length $1$ for the eigenvalue $0$. Thus,  both $s_1$ and $s_2$ have a tuple representation of $(0,1,1)$. It follows that  the Type $1$ gives the correct reduction.} 
\end{description}
\END
\end{ex}

\subsection{Index on $\mathcal{P}_{A}$}
\label{subsec:index_PA}

As we defined index on $\mathcal{L}_{M}$ in Remark \ref{rmk:ind_P_def2}, we define an index {on} $\mathcal{P}_{A}$.

\begin{Def}
\label{def:ind_P}
\rm
Let $r_1,\dots, r_n$ be the immediate leaders of $s\in\mathcal{P}_{A}(k_{1},\ldots,k_{m})$. Let $k:=(k_1,\dots, k_m)$. For all $s\le k$, define an index $\Ind_{\mathcal{P}}(s)$ as follows:
\[\Ind_{\mathcal{P}}(s)=|s|-|r_1\vee \cdots \vee r_n|.\]
\END
\end{Def}

\subsection{Quotient $\mathcal{P}_{A}/{=}$}
\label{subsec:PA_reduction}
We merge elements of the poset $\mathcal{P}_{A}$ by putting an equivalence relation on the poset and then ordering the equivalence classes as follows.

\begin{Def}
\label{def:po_PA}
\rm
Let $=$ be an equivalence relation on $\mathcal{P}_{A}$. Define the quotient $\mathcal{P}_{A}/{=}$ to be the set of equivalence classes with the partial order $\leq$ given by $X\leq Y$ in $\mathcal{P}_{A}/{=}$ if and only if $x\leq y$ in $\mathcal{P}_{A}$ for all $x\in X$ and all $y\in Y$. This means that all nodes with the same tuple representation on $\mathcal{P}_{A}$ are merged into one node.
\END
\end{Def}

\begin{proposition}
\label{prop:sign_preserving_PA}
Let $s\in\mathcal{P}_{A}$ and {$s^{*}\in\mathcal{P}_{A}/=$} where both have the same tuple representation $(r_{1},\ldots,r_{m})$. Then $\Ind_{\mathcal{P}}(s)=\Ind_{\mathcal{P}}(s^{*})$.
\end{proposition}

\begin{proof}
Let $\{r_{1},\ldots,r_{m}\} = \{(r_{11},\ldots,r_{1m}),\ldots,(r_{k1},\ldots,r_{km})\}$ be the set of immediate leaders of $s=(s_{1},\ldots,s_{m})\in\mathcal{P}_{A}$. Note that
\begin{eqnarray*}
\Ind_{\mathcal{P}}(s) &=& |s|-|r_1\vee \cdots \vee r_m| \\
&=& |s_{1}-\max(r_{11},\ldots,r_{k1}),\ldots,s_{m}-\max(r_{1m},\ldots,r_{km})|.
\end{eqnarray*}

Let $\{(r_{11},\ldots,r_{1m}),\ldots,(r_{d1},\ldots,r_{dm})\}$, where $d\leq k$, be the set of all distinct elements of $L$. Then the set of immediate leaders of $s^{*}$ is indeed $\{(r_{11},\ldots,r_{1m}),\ldots,(r_{d1},\ldots,r_{dm})\}$. Since 
\begin{eqnarray*}
\textrm{max}(r_{11},\ldots,r_{k1}) &=& \textrm{max}(r_{11},\ldots,r_{d1}), \\
&\vdots& \\
\textrm{max}(r_{1m},\ldots,r_{km}) &=& \textrm{max}(r_{1m},\ldots,r_{dm}),
\end{eqnarray*}
we have $\Ind_{\mathcal{P}}(s)=\Ind_{\mathcal{P}}(s^{*})$ for $s$ and $s^{*}$ which share the same tuple representation.
\qed
\end{proof}

\begin{lem}
\label{lem:closed_PA} 
$\mathcal{P}_{A}/{=}$ is a closed subset of  $\mathcal{L}_{M}$. 
\end{lem}

\begin{proof}
By construction, $\mathcal{P}_{A}/{=}$ is not a multiset. Moreover, it is by definition a subset of $\mathcal{L}_{A}(k_{1},\ldots,k_{m})$. Now we want to show that $r\wedge s\in\mathcal{P}_{A}/{=}$ for any $r,s\in\mathcal{P}_{A}/{=}$ with $r\ne s$. 

Let $p:\mathcal P_A\to \mathcal P_A/{=}$ be the projection map that maps equal tuples into one element in $\mathcal P_A/{=}$.
Since both maps $p$ and $\mathcal{T}$ are surjective, there exists at least one $R\in V_{\mathcal{G}}^{P}$ with $p(\mathcal{T}(R))=r$. Similarly, there exists at least one $S\in V_{\mathcal{G}}^{P}$ with $p(\mathcal{T}(S))=s$. Note that $p(e)=e$ for all elements $e\in \mathcal P_A$, thus we have $\mathcal{T}(R)=r$ and $\mathcal{T}(S)=s$. Since $R \cap S \in V_{\mathcal{G}}^{P}$ whenever $R, S\in V_{\mathcal{G}}^{P}$   and $\mathcal{T}$ is an order-preserving map, there exists $r\wedge s\in\mathcal{P}_{A}$ (hence $r\wedge s\in \mathcal P_A/{=}$) such that $r \wedge  s\leq  r$ and $r\wedge s\leq  s$. Note that $r\wedge s\neq  r$ as well as $r\wedge s\neq s$ as $r$ and $s$ are distinct. The maximal element $(k_1,\dots, k_m)\in  \mathcal{P}_{A}/{=}$ since  $(k_1,\dots, k_m)\in  \mathcal{P}_{A}$. It follows that $\mathcal{P}_{A}/{=}$ is a closed subset of $\mathcal{L}_{M}$.    
\qed
\end{proof}

\begin{lem}
\label{lem:ind_PA} 
Let $\Ind_{\mathcal{P}}$ be the index on  $\mathcal{P}_{A}$ given by Definition \ref{def:ind_P}. Then, we have
\begin{itemize}
	\item[(i)] $\Ind_{\mathcal{P}} (s)\ge 0$ for all $s\in \mathcal{P}_{A}$;
	\item[(ii)] {${\underset{s\in \mathcal{P}_{A} ,s\le a}{\sum}} \Ind_{\mathcal{P}}(s)\ge |a|$, $\forall\, a\le k=(k_1,\dots, k_m)$;}
	\item[(iii)] {$\underset{s\in \mathcal{P}_{A}}{\sum} \Ind_{\mathcal{P}}(s)\ge  |k|$,}
\end{itemize}
where the equality ''$=$`` in (ii)-(iii) holds on {$\mathcal{P}_{A}/{=}$}.
\end{lem}
\begin{proof} 
Notice that $s\geq r_{1}\vee\cdots\vee r_{n}$. Thus, by Definition \ref{def:ind_P}, we have (i) holds. By Proposition \ref{prop:sign_preserving_PA}, we have
\[ \Ind_{\mathcal{P}}(s) = \Ind_{\mathcal{P}}(s^*).\]
Let $L_{s}$ be a set of immediate leaders of $s\in\mathcal{P}_{A}$, and let $L_{s^{*}}$ be a set of immediate leaders of $s^{*}\in\mathcal{P}_{A}/{=}$. Since $L_{s}\supseteq L_{s^{*}}$, we have
\[{\underset{s\in \mathcal{P}_{A} ,s\le a}{\sum}} \Ind_{\mathcal{P}}(s)\ge {\underset{s^*\in {\mathcal{P}_{A}/=} ,s^*\le a}{\sum}} \Ind_{\mathcal{P}}(s^*)= |a|,\]
where the last equality is given by Lemma \ref{lem:indL} (ii), since {$\mathcal{P}_{A}/{=}$} is closed by Lemma \ref{lem:closed_PA}. Thus, (ii) follows. The statement (iii) is a special case of (ii) by taking $a=k=(k_1,\dots, k_m)$.
\qed	 
\end{proof}

\begin{proposition}
\label{prop:index_Psim}
Let $k:=(k_{1},\ldots,k_{m})$, and let {$s^{\ast}\in\mathcal{P}_{A}/{=}$}. Then an index $\Ind_{\mathcal{P}}(s^{\ast})$ can be recursively assigned as follows:
\[\Ind_{\mathcal{P}}(s^{\ast})=|s^{\ast}|-{\underset{e<s^{\ast}}{\sum}}\Ind_{\mathcal{P}}(e),\quad\forall s^{\ast}\leq k.\]
\end{proposition}

\begin{proof} 
By Lemma \ref{lem:ind_PA} (ii), we have
\begin{eqnarray*}
&& {\underset{s\in \mathcal{P}_{A}/{=} ,s\le a}{\sum}} \Ind_{\mathcal{P}}(s) = |a|\\
&\Leftrightarrow& \Ind_{\mathcal{P}}(a) + {\underset{s\in {\mathcal{P}_{A}/=} ,s < a}{\sum}} \Ind_{\mathcal{P}}(s) = |a|
\end{eqnarray*}
Let {$a=s^{*}\in\mathcal{P}_{A}/{=}$}. Then we obtain the required result
\[\Ind_{\mathcal{P}}(s^{\ast})=|s^{\ast}|-{\underset{e<s^{\ast}}{\sum}}\Ind_{\mathcal{P}}(e),\quad\forall s^{\ast}\leq k.\]
\qed
\end{proof}

\subsubsection{Matrix Representation of $\mathcal{P}_{A}$}
We consider matrix representations of $\mathcal{P}_{A}$ and $\mathcal{P}_{A}/{=}$, which leads to a theoretical justification of our computer algorithm for the lattice reduction. 
\begin{Def}
\label{def:connectivity_matrix}
\rm
Let $s$ be the number of nodes in $\mathcal{P}_{A}$. Define an $s\times s$ adjacency matrix $A_{P}=(a_{ij})=[A_{1} \ldots A_{s}]$, where  $A_{j}$ is the $j$-th column vector of $A_{P}$ for $j=1,\ldots, s$, of $\mathcal{P}_{A}$ with
\[a_{ij}=\begin{cases}1,\q\text{if $i=j$ or node $i$ is immediately connected to node $j$}, \\ 0,\q\text{otherwise.}\end{cases}\]

Define an $s\times s$ matrix $C=(c_{ij})=[C_{1}\ldots C_{s}]$, where  $C_{j}$ is the $j$-th column vector of $C$ for $j=1,\ldots, s$, as a matrix representation of $\mathcal{P}_{A}$ as follows. Let $(r_{1}^{(j)},\ldots, r_{m}^{(j)})$ be the tuple representation of the node $p_{j}\in \mathcal{P}_{A}$. Then $C_{j}=(r_{1}^{(j)},\ldots, r_{m}^{(j)})A_{j}$. More specifically, $i$-th element of a column vector $C_{j}$, $c_{ij}$, is given by
\[
c_{ij} = \begin{cases}
(r_{1}^{(j)},\ldots,r_{m}^{(j)}),\q\text{if $a_{ij}=1$,} \\
(0,\ldots,0), \q\textrm{if $a_{ij}=0$.}
\end{cases}
\]
\END
\end{Def}

\begin{rmk}
\rm
\label{rmk:parents_children_connections}
Note that non-zero tuple representations in the $j$-th column are all identical given by $(r_{1}^{(j)},\ldots,r_{m}^{(j)})$.
With this definition, the lower and the upper triangular matrices describe immediate followers and leaders connections, respectively. 
\END
\end{rmk}

\subsubsection{Construction of {$\mathcal{P}_{A}/{=}$}}
The realisation of a quotient poset can be done through a matrix manipulation on $C=(c_{ij})$ defined in the following.

\begin{Def}
\label{def:addition_operator_on_PA}
\rm
Let $A_{P}=[A_{1}\ldots A_{s}]$ be a $s\times s$ adjacency matrix of $\mathcal{P}_{A}$. Let $I=\{1,\ldots,p\}\subseteq\{1,\ldots,s\}$. Define the $i$-the element of a column vector $\vee_{j\in I}A_{j}=A_{1}\vee\cdots\vee A_{p}$ by
\[
a_{i1}\vee\cdots\vee a_{ip},
\]
where
\[a_{ij}\vee a_{ik}=\max(a_{ij},a_{ik}).\]
\END
\end{Def}

\begin{Def} 
\rm
Let $S=\{1,\ldots,s\}$ be a set of nodes on $\mathcal{P}_{A}$ and let $C$ be the corresponding matrix whose columns we denote by $C_{1},\ldots, C_{s}$. Let $\bowtie$ be an equivalence relation on $S$ with classes $I_{1},\ldots,I_{p}$. Denote by $\overline{C}$ the $s\times p$ matrix whose columns $\overline{C}_{1},\ldots \overline{C}_{p}$ are defined by
\[
\overline{C}_{j}=(r_{1}^{(j)},\ldots,r_{m}^{(j)})\vee_{i\in I_{j}}A_{i},
\]
where $(r_{1}^{(j)},\ldots,r_{m}^{(j)})$ is the {identical tuple representation of all nodes $j\in I_{j}$}, and $A_{i}$ is the $i$-th column of the adjacency matrix $A_{P}$. We say that the matrix $C$ is {\it $\bowtie$-balanced} if for each $j=1,\ldots, p$, the rows for $i\in I_{j}$ of $\overline{C}$ are identical.
\END
\end{Def}

\begin{proposition}
Let $C=(c_{ij})$ be the matrix representation of  $\mathcal{P}_{A}$. Let $\bowtie$ be the equality among tuple representation on $C$ as defined in Definition \ref{def:po_PA}. Then the matrix $C$ is $\bowtie$-balanced. 
\end{proposition}

\begin{proof}
Let $p_{k},p_{l}\in\mathcal{P}_{A}$ with $p_{k}\bowtie p_{l}$ of the corresponding equivalence class $I$. Let $C_{k}=[c_{k1},\ldots, c_{ks}]$ and $C_{l}=[c_{l1},\ldots, c_{ls}]$ be the $k$th and $l$th rows of the $s\times s$ matrix representation $C$, respectively. 

Since $p_{k}$ and $p_{l}$ have the same tuple representation, they have exactly the same set of leaders and follows. Thus,
\[c_{kj}=c_{lj},\quad\forall j\in\{1,\ldots,n\}\setminus\{k,l\}.\] 
Note that the tuple representation equality only occurs when $|p_{k}|=|p_{l}|$, i.e. $p_{k}$ and $p_{l}$ are at the same rank, or equivalently, they not partially ordered. Consequently, 
\begin{eqnarray*}
a_{kk}\vee a_{kl} &=&\textrm{max}(1,0)=1, \\
a_{ll}\vee a_{lk} &=&\textrm{max}(1,0)=1.
\end{eqnarray*}

Since $c_{kk}=c_{ll}$, we obtain
\[\overline{c}_{kk}=\overline{c}_{lk}\]
by choosing $p_{k}$ as a representative element of the equivalence class $I$.

Therefore, we obtain
\[\overline{c}_{kj}=\overline{c}_{lj}\quad\forall j=1,\ldots,p.\]
Thus the $k$-th and $l$-th rows, where $k,l\in I$ of $\overline{C}$, are identical, and $C$ is $\bowtie$-balanced.
\qed
\end{proof}

\begin{proposition}
\label{prop:PA_equal}
Let $C$ be the $\bowtie$-balanced matrix associated with $\mathcal{P}_{A}$, where $\bowtie$ is the equality among tuple representation on $C$ as defined in Definition \ref{def:po_PA}. Let $I_{1},\ldots,I_{p}$ be the equivalence classes where all elements in an equivalence class has the same tuple representation. For each $j=1,\ldots,p$, choose any $j_{i}\in I_{j}$. Then the matrix of the quotient poset $\mathcal{P}_{A}/{=}$ is the $p\times p$ submatrix of $\overline{C}$ whose $j$th row is the row $j_{i}$ of $\overline{C}$. 
\end{proposition}

\begin{proof}
Let
\[q:\q \mathcal{P}_{A}\q \to \q \mathcal{P}_{A}/{=}.\]
This is an order-preserving map and it gives the quotient poset $\mathcal{P}_{A}/{=}$. Since all nodes with the identical tuple representation are merged into one node by $q$, there are $p$ nodes in $\mathcal{P}_{A}/{=}$. We show that each of $p$ node corresponds to $\overline{c}_{ii}$ in $\overline{C}$ and the partial order among $p$ nodes is represented by $\overline{C}$.

It immediately follows that the unique $p$ tuple representation corresponds to diagonal entries $\overline{c}_{ii}$ in $\overline{C}$, where $i=1,\ldots, p$. 
$C$ being $\bowtie$-balanced means that if $k\bowtie l$ then $k$ and $l$ in $\mathcal{P}_{A}$  have the same sets of leaders and followers, respectively. During the matrix manipulation process, the partial order among nodes is preserved. Therefore, the resulting $p\times p$ submatrix is a matrix representation of $\mathcal{P}_{A}/{=}$.
\qed
\end{proof}

\section{Poset of Eigenvalues $\mathcal{E}_{A}$}
\label{sec:QA}
In this section, we present a computer algorithm for the reduction procedure  based on an eigenvalue structure given by  $\mathcal{E}_{A}$. Assigning tuple representations to synchrony subspaces is computationally expensive as this requires detailed information of (generalised) eigenvectors which generate synchrony subspaces (see \cite{Aguiar-2014} for this approach). On the other hand, obtaining a set of eigenvalues of quotient networks, corresponding to synchrony subspace, is computationally less expensive.
	Starting from a lattice of synchrony subspaces $V_{\mathcal{G}}^{P}$, we construct a poset $\mathcal{E}_{A}$, where $s\in \mathcal{E}_{A}$ represents a set of distinct eigenvalues of a quotient network and the connectivity of $\mathcal{E}_{A}$ is the same as $V_{\mathcal{G}}^{P}$. Compared with  $\mathcal{P}_{A}$ whose node tells us how many (generalised) eigenvectors come from each Jordan chain associated with an eigenvalue, each node in $\mathcal{E}_{A}$ tells us how many eigenvalues with algebraic multiplicities a given node has. In other words, $\mathcal{E}_{A}$ doesn't distinguish distinct Jordan chains associated with the same eigenvalue. We discuss the relation between the two and show how to combine them to reduce the lattice in an optimal way.

\subsection{Construction of $\mathcal{E}_{A}$}
We construct a partially ordered set $\mathcal{E}_{A}$ whose nodes are given by a tuple representation which counts the number of each distinct eigenvalues from the adjacency matrix of the corresponding quotient network and whose connectivity is the same as given by the lattice $V_{\mathcal{G}}^{P}$. The partial order on $\mathcal{E}_{A}$ is entirely inherited from $V_{\mathcal{G}}^{P}$, which cannot be determined from sets of eigenvalues alone.

\begin{Def}
\rm
Let $\mathcal{G}$ be an $n$-cell regular network with $n\times n$ adjacency matrix $A$ with the distinct eigenvalues $\lambda_{1},\ldots, \lambda_{k}$. Let $\mathcal{G}_{\bowtie}$ be a $p$-cell quotient network of $\mathcal{G}$ restricted to a $p$-dimensional synchrony subspace $\Delta_{\bowtie}\in V_{\mathcal{G}}^{P}$ associated with a $p\times p$ adjacency matrix $A_{\bowtie}$, where $p<n$. Define a map $\mathcal{V}$ which uniquely assigns the set of eigenvalues of $A_{\bowtie}$ for a given synchrony subspace $\Delta_{\bowtie}$ as
\begin{align}
\mathcal{V}:\q V_{\mathcal{G}}^{P} \q &\to \q \mathcal{E}_{A}\notag\\
\Delta_{\bowtie} \q &\mapsto \q (t_{1},\dots,t_{k}),\label{eq:map_eigenvalue}
\end{align}
where $\displaystyle t_{i}$ counts the number of eigenvalue $\lambda_{i}$ for $i=1,\ldots,k$. We define the covering relations on $\mathcal{E}_{A}$ to be directly inherited from $V_{\mathcal{G}}^{P}$.
\END
\end{Def}

\begin{Def}
\rm
Let $(r_{1},\ldots,r_{m})$ be a $m$-tuple representation on the poset $\mathcal{P}_{A}(k_{1},\ldots,k_{m})$ of a regular network with  adjacency matrix $A$, which has  distinct eigenvalues $\lambda_{1},\ldots,\lambda_{k}$, where $k\leq m$. Let $\lambda_{i}$ be the corresponding eigenvalue of $A$ to $r_{i}$ for $i=1,\ldots,m$, given by the following map:
\[\mathcal{W}:\q (r_{1},\ldots,r_{m})\q \mapsto \q (\lambda_{1},\ldots,\lambda_{m}).\]
Note that $\lambda_{i}$ can be equal to $\lambda_{j}$ for some $i\neq j$. 
\END
\end{Def}

\begin{lem}
\label{lem:EA}
Let $A:\bc^n\to\bc^n$ be an adjacency matrix of a regular network $\mathcal{G}$ with distinct eigenvalues $\lambda_{1},\ldots,\lambda_{k}$. Let $\mathcal{V}(\Delta_{\bowtie})=(t_{1},\ldots,t_{k})\in\mathcal{E}_{A}$ and let $\mathcal{T}(\Delta_{\bowtie})=(r_{1},\ldots,r_{m})\in\mathcal{P}_{A}$ for an arbitrary $\Delta_{\bowtie}\in V_{\mathcal{G}}^{P}$, where $k\leq m\leq n$. Then for any representation $\mathcal{P}_{A}(k_{1},\ldots,k_{m})$ associated with $V_{\mathcal{G}}^{P}$, we have
\[t_{i}=\sum_{\mathcal{W}(r_{j})=\lambda_{i}}r_{j},\quad\textrm{for all}\quad i=1,\ldots,k.\]
\end{lem}
\begin{proof}
We decompose
\[\bc^{n}=E_{\lambda_{1}}\oplus\cdots\oplus E_{\lambda_{k}}.\]

Consider 
\[\Delta_{\bowtie}\cap E_{\lambda_{i}}=\bigoplus_{j=1}^{l_{i}}J_{j}^{i},\]
where each $J_{j}^{i}$ is the span of the corresponding Jordan chain, and $l_{i}$ is the number of Jordan blocks corresponding to an eigenvalue $\lambda_{i}$ of $A$. Denote by $k_{1}\geq\cdots\geq k_{l_{i}}$ the size of Jordan blocks in decreasing order. A map $\mathcal{T}$ assigns an non-negative integer $r_{j}$ to each Jordan block $J_{j}^{i}$ in the following way:
\[\mathcal{T}(\Delta_{\bowtie}\cap E_{\lambda_{i}})=(r_{1},\ldots,r_{l_{i}}) \quad\textrm{which satisfies}\quad 0\leq r_{j}\leq k_{j} \quad\forall j=1,\ldots, l_{i}.\] 

This means that the space $\Delta_{\bowtie}\cap E_{\lambda_{i}}$ can be spanned by the first $r_{1}$ vectors of the Jordan chain of the block $J_{1}^{i}$, to the first $r_{l_{i}}$ vectors of the Jordan chain of the block $J_{l_{i}}^{i}$.
Note that if $k_{c}=k_{d}$ for some $c,d$, then the assignment of non-negative integers are not unique. However, this non-uniqueness becomes immaterial when we consider the number of basis vectors as follows.

Thus $\Delta_{\bowtie}\cap E_{\lambda_{i}}$ is spanned by $\displaystyle t_{i}=\sum_{j=1}^{l_{i}}r_{j}$ vectors corresponding to the eigenvalue $\lambda_{i}$. Let $\{\mathbf{v}_{1},\ldots,\mathbf{v}_{q}\}$ be the $q$ basis vectors of a $q$-dimensional synchrony subspace $\Delta_{\bowtie}$, where $\displaystyle q=\sum_{i=1}^{k}t_{i}$.

Hence the matrix $A$ with respect to the basis $\{\mathbf{v}_{1},\ldots,\mathbf{v}_{q}\}$ with its complement can be written as the following block structure:
\[\left[\begin{array}{cc} A_{\bowtie} & R \\ 0_{(n-q)\times q} & B\end{array}\right]\]
where $R$ is a $q\times (n-q)$ matrix, $B$ is a $(n-q)\times(n-q)$ matrix, and $A_{\bowtie}$ has the $t_{i}$ of eigenvalue $\lambda_{i}$ for $i=1,\ldots,k$.

Therefore, for $\Delta_{\bowtie}\in V_{\mathcal{G}}^{P}$ with $\mathcal{T}(\Delta_{\bowtie})=(r_{1},\ldots,r_{m})\in\mathcal{P}_{A}(k_{1},\ldots,k_{m})$, the corresponding node in $\mathcal{E}_{A}$ has the form $\mathcal{V}(\Delta_{\bowtie})=(t_{1},\ldots,t_{k})$ for any representation $\mathcal{P}_{A}$ associated with $V_{\mathcal{G}}^{P}$.
\qed
\end{proof}

\begin{Def}
\rm
\label{def:covering_relation}
Let $A$ be the adjacency matrix of a given regular network $\mathcal{G}$ with  total phase space $P$.  Let $\mathcal{T}(S)=s$ where $S \in V_{\mathcal{G}}^{P}$ and $s\in\mathcal{P}_{A}$. We say that $\mathcal{P}_{A}$ and $V_{\mathcal{G}}^{P}$ have {\it the same covering relation}, if for any $s\in \mathcal{P}_{A}$, the set $\{r_1,\dots, r_m\}$ of immediate followers (resp. leaders) $r_i=\mathcal{T}(R_i)$ of $s$ coming from a synchrony subspace $R_i\supset S$ (resp. $R_i\subset S$) in $V_{\mathcal{G}}^{P}$ for $i=1,\dots, m$, has the following property: for any immediate follower (resp. leader) $r$ of $s$ in $\mathcal{P}_{A}$, there exists a $r_i$ such that $r_i$ and $r$ have the same tuple representation.
\END
\end{Def}

\begin{rmk}
\rm
In Example \ref{ex:4_cell_network343}, we identified that Type $1$ $\mathcal{P}_{A}$ is the expected structure. Note that only Type $1$ $\mathcal{P}_{A}$ satisfies the covering relation defined in Definition \ref{def:covering_relation}. {For example, consider the left $\mathcal{P}_{A}$ of Type $2$ in Example \ref{ex:4_cell_network343}. Let $s=(1,0,1)\in\mathcal{P}_{A}$. Then the set of immediate followers $r_{i}=\mathcal{T}(R_{i})$ of $s$ coming from $R_{i}\supset S$ in $V_{\mathcal{G}}^{P}$ is given by $\{(1,1,1)\}$. However, the immediate follower $(2,0,1)\in\mathcal{P}_{A}$ of $s$ is different from $(1,1,1)$.}
\END
\end{rmk}

\subsection{Index on $\mathcal{E}_{A}$}
\label{sec:index_QA}

\begin{Def}
\rm
\label{def:ind_EA}
We define an index $\Ind_{\mathcal{E}}(s)$ recursively as follows:
\[\Ind_{\mathcal{E}}(s)=|s|-\sum_{e<s, \Ind_{\mathcal{E}}(e)\geq 0}\Ind_{\mathcal{E}}(e)\]
\END
\end{Def}

\begin{rmk}
\rm
Note that $\Ind_{\mathcal{E}}(s)<0$ for some $s\in\mathcal{E}_{A}$, which is a property the indices on $\mathcal{L}_{M}$, $\mathcal{P}_{A}$ don't satisfy.
\END
\end{rmk}

\begin{proposition}
\label{prop:positivity_Ind_2}
Suppose that $\mathcal{P}_{A}$ and $V_{\mathcal{G}}^{P}$ have the same covering relation defined in Definition \ref{def:covering_relation} for a regular network associated with the adjacency matrix $A$. Then, 
\begin{itemize}
\item[(i)] $\Ind_{\mathcal{P}}(s)\ge\Ind_{\mathcal{E}}(s)$;
\item[(ii)] $\Ind_{\mathcal{P}}(s)>\Ind_{\mathcal{E}}(s)$ for some $S\in V_{\mathcal{G}}^{P}$ if and only if there exists $S_1,S_2 \subsetneq\, S\in V_{\mathcal{G}}^{P}$ of the same rank such that $s_1=s_2$.
\end{itemize}
\end{proposition}

\begin{proof}
By Lemma \ref{lem:ind_PA} (iii), we have
\begin{equation}
\label{eq:ind_PS_1}
\Ind_{\mathcal P}(s)=|s|-\underset{\text{distinct $e$}}{\sum_{e<s}} \Ind_{\mathcal{P}}(e).
\end{equation}

On the other hand, by Definition \ref{def:ind_EA}, we have
\begin{equation}\label{eq:ind_PS_2}
\Ind_{\mathcal{E}}(s)=|s|-\sum_{e<s, \Ind_{\mathcal{E}}(e)\geq 0}\Ind_{\mathcal{E}}(e),
\end{equation}
where identical $e$'s will be summed repeatedly in case there are different synchrony subspaces that are associated with the same tuple $e$.

It follows that 
\begin{equation}\label{eq:ind_PS_3}
\Ind_{\mathcal{E}}(s)\le |s|-\underset{\text{distinct $e$}}{\sum_{e<s,\Ind_{\mathcal{E}}(e)\geq 0}} \Ind_{\mathcal{E}}(e)\le |s|-\underset{\text{distinct $e$}}{\sum_{e<s}} \Ind_{\mathcal{E}}(e).
\end{equation}
Therefore, if we define
	\[\Ind_{\mathcal{P}}(s):=|s|-\underset{\text{distinct $e$}}{\sum_{e<s}} \Ind_{\mathcal{P}}(e),\q s\in \mathcal{P}_{A},\]
recursively on 	$\mathcal{P}_{A}$ with the initial value $\Ind_{\mathcal{P}}(\top)=|\top|$ for the top element $\top\in\mathcal{P}_{A}$, then (\ref{eq:ind_PS_1}) and (\ref{eq:ind_PS_3}) imply 
\begin{equation}
\label{eq:ind_PS_4}
\Ind_{\mathcal{E}}(s)\le \Ind_{\mathcal{P}}(s).
\end{equation}  
Therefore, (i) holds.

Also, the equality ``$=$'' attains in (\ref{eq:ind_PS_4}) if and only if there are no leader synchrony subspaces of $s$ having  the same tuple $e$ for which $\Ind_{\mathcal E}(e)\ge 0$ and there is no leader $e$ such that $\Ind_{\mathcal{E}}(e)<0$. We refer this statement as (S).

To show (ii), consider $s$ such that $\Ind_{\mathcal{P}}(s)>\Ind_{\mathcal{E}}(s)$. Assume to the contrary that there are no two different leader synchrony subspaces of $s$ having the same tuple representation. Then, by (S), there must be a leader $e$ such that  $\Ind_{\mathcal{E}}(e)<0$. Without loss of generality, we can assume that $e$ is the smallest tuple between the top element and $s$ with negative $\Ind_{\mathcal{E}}$. Then, $\Ind_{\mathcal{E}}(e)<0\le \Ind_{\mathcal{P}}(e)$ and by (S) again, there are two different leader synchrony subspaces of $e$ having the same tuple representation, which is a contradiction. Thus, given $s$ such that $\Ind_{\mathcal{P}}(s)>\Ind_{\mathcal{E}}(s)$, we have shown that there must be two different leader synchrony subspaces of $s$ having the same tuple representation.

To show the other implication of (ii), consider $s$ such that there are two different leader synchrony subspaces $S_1,S_2$ of $S$ having the same tuple representation, say $e$. If  $\Ind_{\mathcal{E}}(e)\ge 0$, then by (S), we have  $\Ind_{\mathcal{P}}(s)>\Ind_{\mathcal{E}}(s)$. Otherwise, if  $\Ind_{\mathcal{E}}(e)<0$, we also have $\Ind_{\mathcal{P}}(s)\ge 0>\Ind_{\mathcal{E}}(s)$. Therefore, in both cases, we have $\Ind_{\mathcal{P}}(s)>\Ind_{\mathcal{E}}(s)$.
\qed
\end{proof}

\begin{cor}
\label{cor:positivity_Ind}
If there exists $S\in V_{\mathcal{G}}^{P}$ such that $\Ind_{\mathcal{E}}(s)< 0$ then there exists $S_1,S_2\in V_{\mathcal{G}}^{P}$ such that $s_1=s_2$.
\end{cor}
\begin{proof}
Assume to the contrary that there are no two different synchrony subspaces $S_1,S_2\in V_{\mathcal{G}}^{P}$ with the same tuples $s_1=s_2$. Then, $V_{\mathcal{G}}^{P}$ (and thus $\mathcal E_A$) can be viewed as a subset of $\mathcal P_A/=$. Without loss of generality, we can assume that $\Ind_{\mathcal{E}}(e)\ge 0$ for all $e<s$. Thus, by Definition \ref{def:ind_EA}, we have
		\begin{equation}\label{eq:cor_pf1}
		\Ind_{\mathcal{E}}(s)= |s|-\underset{\text{distinct $e$}}{\sum_{e<s {\text{\, in $\mathcal E_A$}}}} \Ind_{\mathcal{E}}(e)\ge |s|-\underset{\text{distinct $e$}}{\sum_{e<s {\text{\, in $\mathcal P_A$}}}} \Ind_{\mathcal{E}}(e),
		\end{equation}
where the right-hand side expression coincides with $\Ind_{\mathcal P}(s)$ (by induction). This leads to a contradiction to our assumption $\Ind_{\mathcal{E}}(s)< 0$, since $\Ind_{\mathcal{P}}(s)\ge 0$ by Lemma \ref{lem:ind_PA} (i). 
\qed
\end{proof}

\subsection{Quotient $\mathcal{E}_{A}/{\sim}$}
\label{subsec:Quotient_QA}
Now we define the matrix representations of $\mathcal{E}_{A}$ and $\mathcal{E}_{A}/{\sim}$, which are the mathematical object used in the computer algorithm.

\subsubsection{Matrix Representation of $\mathcal{E}_{A}$}
{We define the matrix representation of $\mathcal{E}_{A}$, which is analogous  to the matrix representation of $\mathcal{P}_{A}$ as defined in Definition \ref{def:connectivity_matrix}. Note that the connectivity of $\mathcal{E}_{A}$ is given by the lattice of synchrony subspaces $V_{\mathcal{G}}^{P}$, which is in general different from the connectivity of $\mathcal{P}_{A}$.}

\begin{Def}
\label{def:connectivity_matrix_E}
\rm
Let $s$ be the number of nodes in $\mathcal{E}_{A}$. Define a $s\times s$ adjacency matrix $A_{E}=(a_{ij})=[A_{1} \ldots A_{s}]$, where  $A_{j}$ is the $j$-th column vector of $A_{E}$ for $j=1,\ldots, s$, of $\mathcal{E}_{A}$ with
\[a_{ij}=\begin{cases}1,\q\text{if $i=j$ or node $i$ is immediately connected to node $j$}, \\ 0,\q\text{otherwise.}\end{cases}\]

Define a $s\times s$ matrix $E=(e_{ij})=[E_{1}\ldots E_{s}]$, where  $E_{j}$ is the $j$-th column vector of $E$ for $j=1,\ldots, s$, as a matrix representation of $\mathcal{E}_{A}$ as follows. Let $(t_{1}^{(j)},\ldots, t_{d}^{(j)})$ be the tuple representation of the node $q_{j}\in \mathcal{E}_{A}$, where $t_{k}$ for $k=1,\ldots, d$ is the number of eigenvalue $\lambda_{k}$ in the node $q_{i}\in \mathcal{E}_{A}$. Then $E_{j}=(t_{1}^{(j)},\ldots, t_{d}^{(j)})A_{j}$. More specifically, $i$-th element of a column vector $E_{j}$, $e_{ij}$, is given by
\[
e_{ij} = \begin{cases}
(t_{1}^{(j)},\ldots,t_{d}^{(j)}),\q\text{if $a_{ij}=1$,} \\
(0,\ldots,0), \q\textrm{if $a_{ij}=0$.}
\end{cases}
\]
\END
\end{Def}

\begin{rmk}
\label{rmk:additional_edges}
\rm 
Note that $A_{P}-A_{E}\neq\mathbf{0}$ in general since $\mathcal{P}_{A}$ has additional covering relation in general due to the partial order among their tuple representation. As a result, $A_{P}-A_{E}$ is a nonnegative matrix. 

\END
\end{rmk}

\subsubsection{Construction of $\mathcal{E}_{A}/{\sim}$}
We now consider an equivalence relation on $\mathcal{E}_{A}$ to obtain a quotient set $\mathcal{E}_{A}/{\sim}$. Our goal is to identify the structure of {$\mathcal{P}_{A}/{=}$} using the matrix representation of $\mathcal{E}_{A}$.

\begin{Def} 
\rm
Let $S=\{1,\ldots,s\}$ be a set of nodes on $\mathcal{E}_{A}$ and let $E$ be the corresponding matrix whose columns we denote by $E_{1},\ldots, E_{s}$. Let $\bowtie$ be an equivalence relation on $S$ with classes $I_{1},\ldots,I_{p}$. Denote by $\overline{E}$ the $s\times p$ matrix whose columns $\overline{E}_{1},\ldots \overline{E}_{p}$ are defined by
\[
\overline{E}_{j}=(t_{1}^{(j)},\ldots,t_{d}^{(j)})\vee_{i\in I_{j}}A_{i},
\]
where $(t_{1}^{(j)},\ldots,t_{d}^{(j)})$ is the {identical tuple representation of all nodes $j\in I_{j}$}, and $A_{i}$ is the $i$-th column of the adjacency matrix $A_{E}$.

We say that the matrix $E$ is $\bowtie$-balanced if for each $j=1,\ldots, p$, the rows for $i\in I_{j}$ of $\overline{E}$ are identical. {Let $\sim$ be the equivalence relation on $\mathcal{E}$ associated with $\bowtie$ on $E$}. By keeping only one representative row from each equivalence class, we define the $p\times p$ matrix which represents a quotient {$\mathcal{E}_{A}/\sim$}.
\END
\end{Def}

\begin{lem}
\label{lem:balanced_connectivity_matrix_QA}
Suppose $\mathcal{P}_{A}$ and $V_{\mathcal{G}}^{P}$ have the same covering relation {defined in Definition \ref{def:covering_relation}} for an $n$-cell regular network $\mathcal{G}$. Let $C$ be the matrix representation of $\mathcal{P}_{A}$, and let $E$ be the matrix representation of $\mathcal{E}_{A}$. Let $\bowtie$ be an equivalence relation on $\mathcal{P}_{A}$ which assigns the quotient poset $\mathcal{P}_{A}/{=}$. Then $E$ is $\bowtie$-balanced and $s\in\mathcal{E}_{A}/{\sim}$, where $\mathcal{E}_{A}/{\sim}$ corresponds to $\bowtie$-balanced $E$, satisfies the following conditions:
\begin{enumerate}
\item[(i)]
$\displaystyle \Ind_{\mathcal{E}_{A/\sim}}\geq 0$,
\item[(ii)]
$\displaystyle \sum\Ind_{\mathcal{E}_{A/\sim}}=n$.
\end{enumerate}  
\end{lem}

\begin{proof}
Let {$p_{k}= p_{l}$} for $p_{k},p_{l}\in\mathcal{P}_{A}$. Let $C_{k}=[c_{k1},\ldots, c_{ks}]$ and $C_{l}=[c_{l1},\ldots, c_{ls}]$ be the $k$th and $l$th rows of the $s\times s$ matrix representation $C$, respectively. {Since $C$ is $\bowtie$-balanced, we have}
\[c_{kj}=c_{lj},\quad\forall j\in\{1,\ldots,s\}\setminus\{k,l\},\] 
with $c_{kk}=c_{ll}$. This gives $\overline{C}_{k}$ and $\overline{C}_{l}$ to be identical where $\overline{C}_{k}$ and $\overline{C}_{l}$ are $k$th and $l$th rows of the $s\times p$ matrix.
Note that we have
\[(t_{1},t_{2},\ldots,t_{d}) = (\sum_{i=1}^{j_{1}}r_{1i}, \sum_{i=1}^{j_{2}}r_{2i},\ldots, \sum_{i=1}^{j_{d}}r_{di}),\quad\textrm{where}\quad \sum_{i=1}^{d}j_{i}=m.\]
Thus
\[(r_{1}^{(k)},\ldots, r_{m}^{(k)})=(r_{1}^{(l)},\ldots, r_{m}^{(l)}) \Rightarrow (t_{1}^{(k)},\ldots, t_{d}^{(k)})=(t_{1}^{(l)},\ldots, t_{d}^{(l)}).\]

Let $\overline{E}_{k}$ and $\overline{E}_{l}$ be the $k$th and $l$th rows of the $s\times p$ matrix representation, respectively. Since $\mathcal{P}_{A}$ and $V_{\mathcal{G}^{P}}$ have the same covering relation, if $\bowtie$ is the equivalence relation on $\mathcal{P}_{A}$, so it is on $\mathcal{E}_{A}$. Thus, it immediately follows that if $\overline{C}_{k}$ and $\overline{C}_{l}$ are identical, then the corresponding rows $\overline{E}_{k}$ and $\overline{E}_{l}$ are also identical.

By Proposition \ref{prop:PA_equal}, notice that $\mathcal{P}_{A}/{=}$ corresponds to $\bowtie$-balanced $C$. Furthermore, $\mathcal{P}_{A}/{=}$ is a closed subset of $\mathcal{L}_{M}$. Since {$\mathcal{P}_{A}/{=}$} is a closed subset of $\mathcal{L}_{M}$, by Lemma \ref{lem:ind_PA}, $s\in\mathcal{P}_{A}/{=}$ satisfies the following conditions:
\begin{enumerate}
\item[(i)]
$\displaystyle \Ind_{\mathcal{P}_{A}/=}\geq 0$,
\item[(ii)]
$\displaystyle \sum\Ind_{\mathcal{P}_{A}/=}=|k|=n$,
\end{enumerate}
where $k:=(k_{1},\ldots,k_{m})$ and $k_{1},\ldots,k_{m}$ are Jordan block sizes of $A$.

Let $\overline{A_{P}}$ and $\overline{A_{E}}$ be the adjacency matrices of $\mathcal{P}_{A}/{=}$ and $\mathcal{E}_{A}/{\sim}$, respectively. By Proposition \ref{prop:index_Psim}, an index on $\mathcal{P}_{A}/{=}$ is uniquely determined by the structure $\overline{A_{P}}$. Since we have $\overline{A_{P}}=\overline{A_{E}}$, with Definition \ref{def:ind_EA}, the required conditions immediately follow. 
\qed
\end{proof}

\begin{thm}
\label{thm:reduced_EA}
There exists $\mathcal{E}_{A}/{\sim}$ such that $\Ind_{\mathcal{E}_{A}/\sim}(s)\geq 0$ for all $s$ and $\sum\Ind_{\mathcal{E}_{A}/\sim}(s)=n$ for an $n$-cell regular network if and only if $\mathcal{P}_{A}$ and $V_{\mathcal{G}}^{P}$ have the same covering relation {defined in Definition \ref{def:covering_relation}}.
\end{thm}

\begin{proof}
%
%

Assume that  $\mathcal{P}_{A}$ and $V_{\mathcal{G}}^{P}$ have the same covering relation. Then, by Lemma \ref{lem:balanced_connectivity_matrix_QA}, $\Ind_{\mathcal{E}_{A}/\sim}(s)\geq 0$ for all $s$ and $\sum\Ind_{\mathcal{E}_{A}/\sim}(s)=n$ hold.

Assume to the contrary that $\mathcal{P}_{A}$ and $V_{\mathcal{G}}^{P}$ have different covering relation. Then, there exists a missing connection in $\mathcal{E}_{A}/{\sim}$ compared to $\mathcal{P}_{A}/{\sim}$, due to Definition \ref{def:covering_relation}. Let $r,s\in \mathcal{L}_{M}$ be such that $r<s$ in $\mathcal{P}_{A}/{\sim}$, but not in  $\mathcal{E}_{A}/{\sim}$. {Then, we have $\Ind_{\mathcal{E}_{A}/\sim}(s)>\Ind_{\mathcal{P}_{A}/\sim}(s)$}, which leads to 
\[\sum\Ind_{\mathcal{E}_{A}/\sim}(s)>\sum\Ind_{\mathcal{P}_{A}/\sim}(s){\geq} n,\]
where the last equality used Lemma \ref{lem:ind_PA} (iii) for $\mathcal{P}_{A}/{=}$.
\qed
\end{proof}

 \begin{rmk}
\rm
In our computer algorithm, we consider the tuple representation equality on $\mathcal{E}_{A}$ as an equivalence relation and consider all possible partitions in each equivalence class. Since the converse of Lemma \ref{lem:balanced_connectivity_matrix_QA} is not true in general, there might be more than one such equivalence relations on $\mathcal{E}_{A}$. 
\END
\end{rmk}

\section{Examples}
\label{sec:examples}

We show how all the possible reduced posets $\mathcal{E}_{A}/{\sim}$ can be obtained. The first example has a unique $\mathcal{E}_{A}/{\sim}$. The second example shows four possible $\mathcal{E}_{A}/{\sim}$ for the given regular network, from which a topologically unique reduction can be inferred. Finally we show {a counter example} where reduction is not possible due to the lack of the property defined in Definition \ref{def:covering_relation}.

\subsection{Unique $\mathcal{E}_{A}/{\sim}$}
\begin{ex}
\rm
\label{ex:4_cell_network18_reduction}
Consider the $4$-cell regular network {depicted} in Figure \ref{fig:4_cell_network18} with a repeated eigenvalue with algebraic multiplicity $2$ and geometric multiplicity $2$. 
\begin{figure}[htb!]
\begin{center}
\begin{tabular}{ccll}
Network $\mathcal{G}$ & Adjacency matrix $A$ & eigenvalues & eigenvectors\\
\hline
\multirow{4}{*}{\includegraphics[scale=0.25]{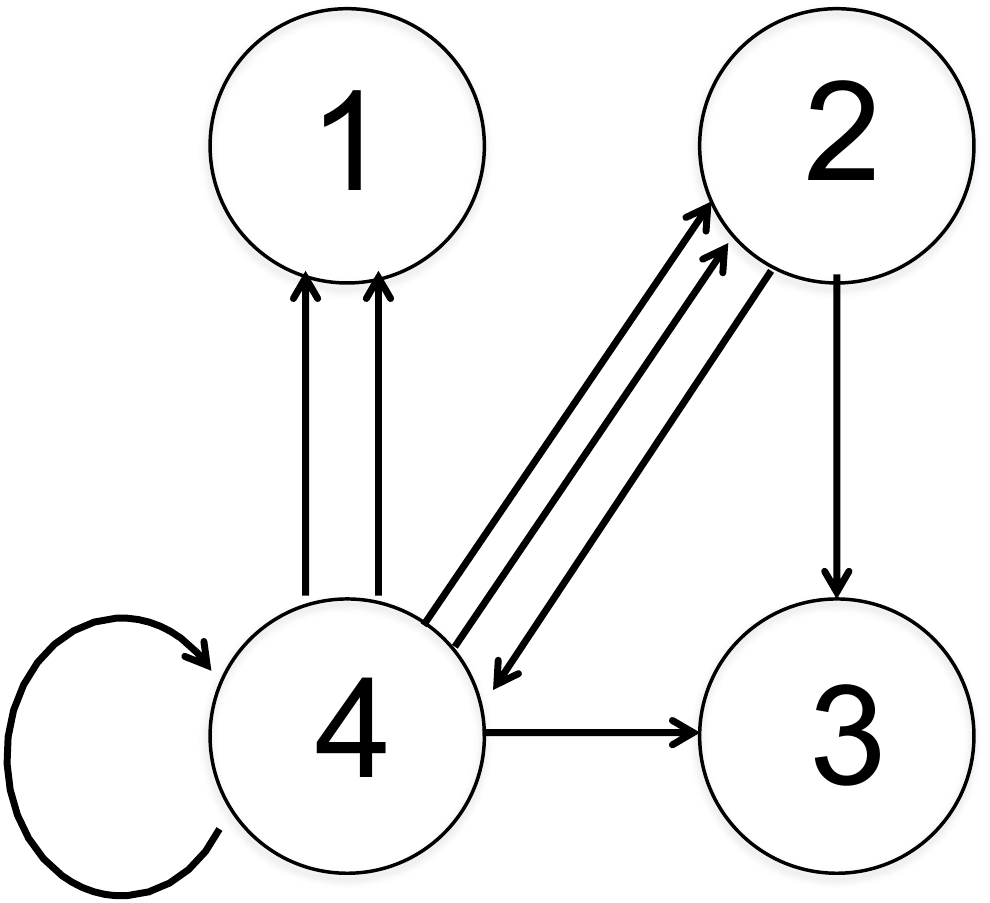}} 
&
\multirow{4}{*}{$\left(\begin{array}{cccc}
0 & 0 & 0 & 2\\
0 & 0 & 0 & 2 \\
0 & 1 & 0 & 1 \\
0 & 1 & 0 & 1  
\end{array}\right)$} &
$\lambda_{1}=0$ & $(1,0,0,0)$\\
& & $\lambda_{2}=0$ & $(0,0,1,0)$\\
& & $\lambda_{3}=-1$ & $(-2,-2,1,1)$\\
& & $\lambda_{4}=2$  & $(1,1,1,1)$ \\
& &
\end{tabular}
\end{center}
\caption{$4$-cell regular network $\mathcal{G}$ with corresponding adjacency matrix $A$ and its eigenvalues. The repeated eigenvalue $\lambda_{1}=\lambda_{2}=0$ has algebraic multiplicity $2$ and geometric multiplicity $2$.}
\label{fig:4_cell_network18}
\end{figure}

Using the information from the lattice of synchrony subspaces $V_{\mathcal{G}}^{P}$ in Figure \ref{fig:network18_EA_reduction}, we obtain the unique $\mathcal{E}_{A}$ shown in Figure \ref{fig:network18_EA} (a). 

\begin{figure}[h!]
\begin{center}
\begin{tabular}{cc}
(a) $\mathcal{E}_{A}$ & (b) Index on  $\mathcal{E}_{A}$\\
\\
\includegraphics[scale=0.3]{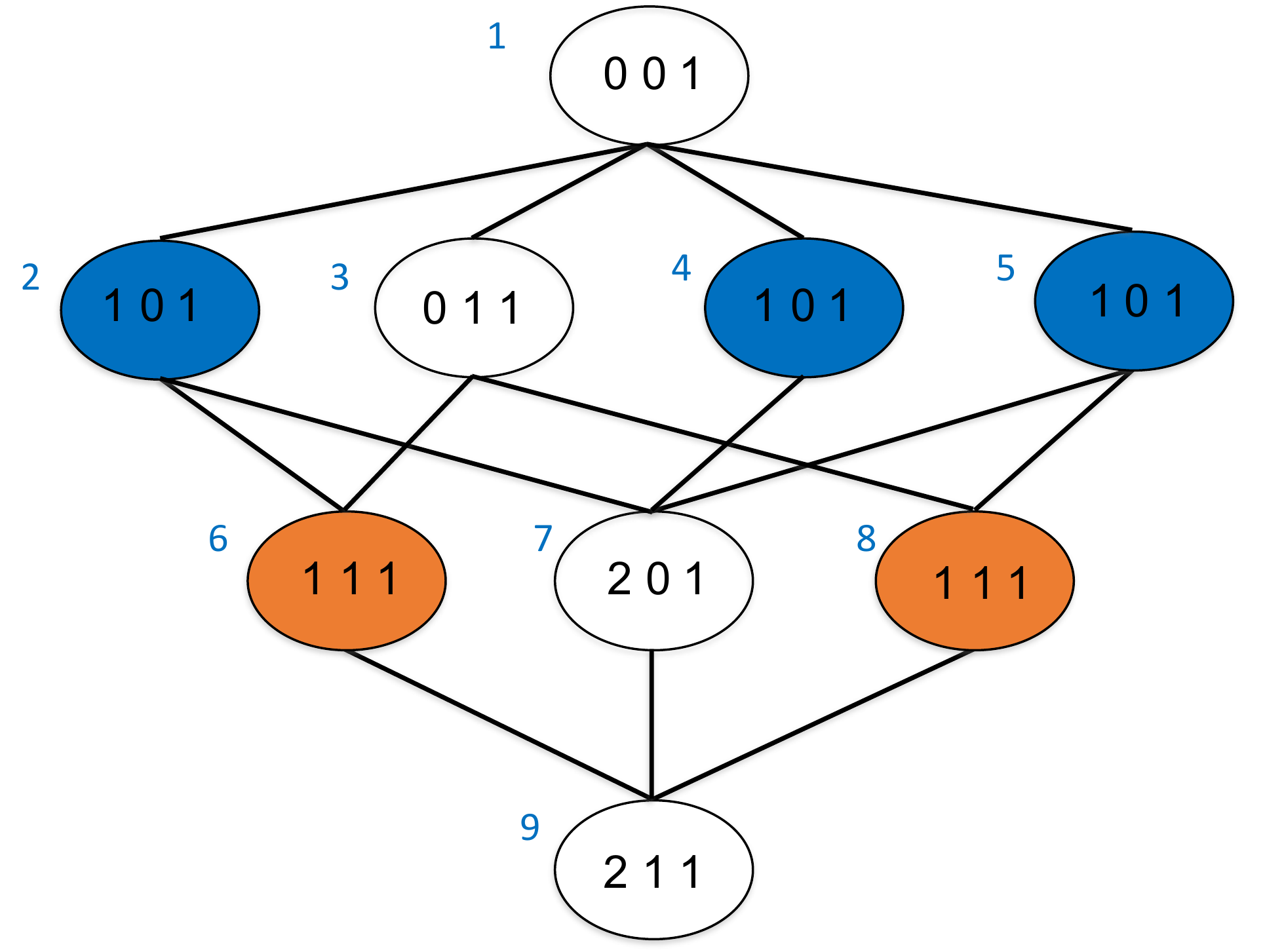}  &
\includegraphics[scale=0.3]{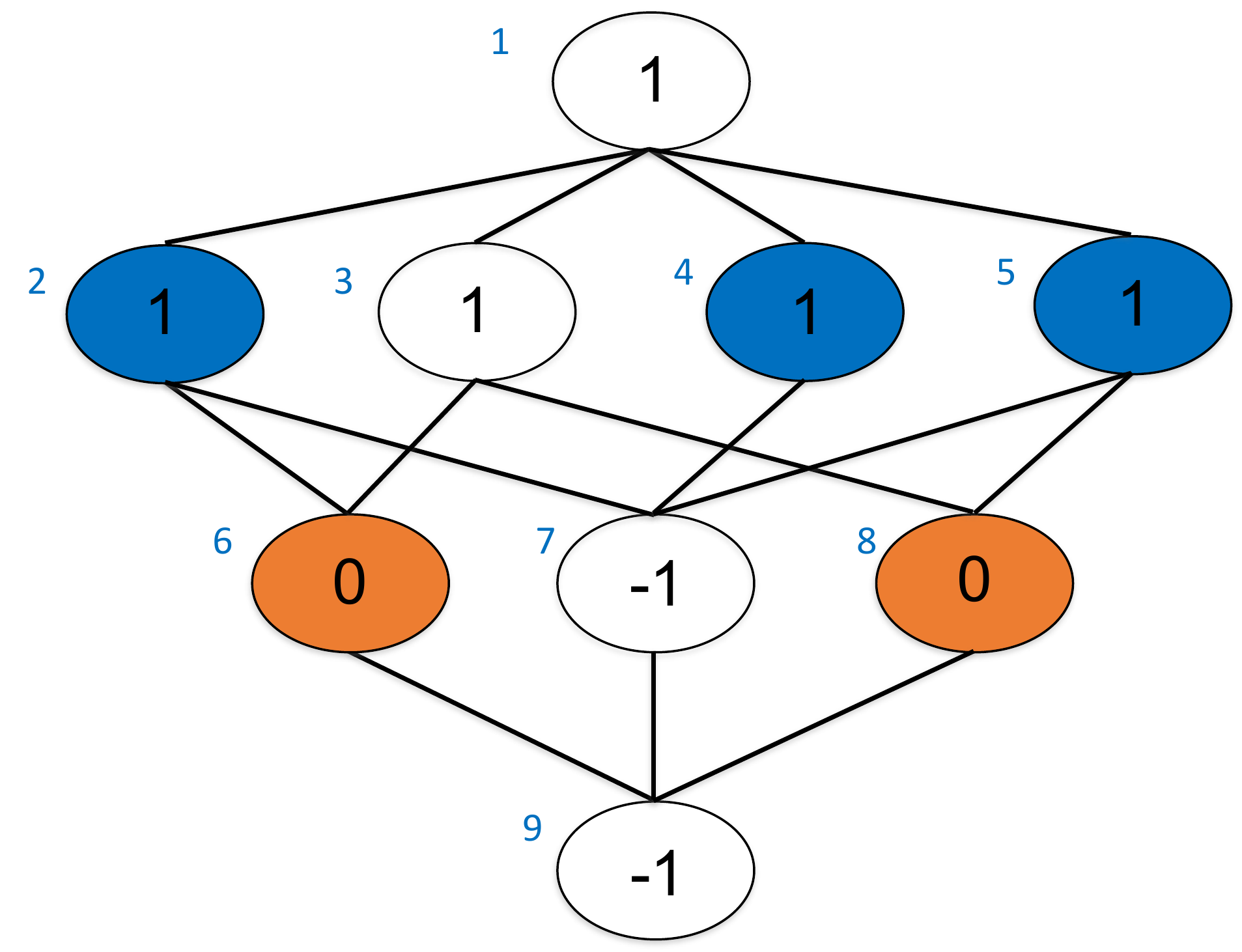}
\end{tabular}
\end{center}
\caption{(a) Let $(s_{1}, s_{2}, s_{3})\in \mathcal{E}_{A}$. Non-negative integers $s_{1}$, $s_{2}$ and $s_{3}$ represent the number of distinct eigenvalues $\lambda=0$, $\lambda=-1$ and $\lambda=2$ in the corresponding quotient network, respectively. Identical tuple representations are colored the same. (b) Using Definition \ref{def:ind_EA}, indices are assigned with each node in $\mathcal{E}_{A}$.}
\label{fig:network18_EA}
\end{figure}

Since $\Ind_{\mathcal{E}}(2,0,1)<0$ and $\Ind_{\mathcal{E}}(2,1,1)<0$ as in Figure \ref{fig:network18_EA} (b), this suggests a reduction (cf. Corollary \ref{cor:positivity_Ind}). Table \ref{tab:network18_equivalence_relations} shows $10$ possible equivalence relations on $\mathcal{E}_{A}$, where some of them are associated with the equivalence relations on $\mathcal{P}_{A}$.
\begin{table}[h!]
\begin{center}
\begin{tabular} {r|r|r}
Equivalence relations on $\mathcal{E}_{A}$ & Associated $\mathcal{P}_{A}$ & $\mathcal{E}_{A}/{\sim}$ \\
\hline
(1)(2)(3)(4)(5)(6)(7)(8)(9)&                                 & invalid\\
\hline
(1)(24)(3)(5)(6)(7)(8)(9) & Type 1-A, Type 1-D & invalid\\
(1)(25)(3)(4)(6)(7)(8)(9) &                                & invalid\\
(1)(2)(3)(45)(6)(7)(8)(9) & Type 1-B, Type 1-C  & invalid\\
\hline
(1)(245)(3)(6)(7)(8)(9)   &                                 & invalid\\
\hline
(1)(2)(3)(4)(5)(68)(7)(9) &                                & invalid \\
\hline
(1)(24)(3)(5)(68)(7)(9)   &                                & invalid\\
(1)(25)(3)(4)(68)(7)(9)   & Type 2-A, Type 2-B  & valid\\
(1)(2)(3)(45)(68)(7)(9)   &                                 & invalid\\
\hline
(1)(245)(3)(68)(7)(9)     & Type 3-A, Type 3-B  & invalid 
\end{tabular}
\end{center}
\caption{$10$ possible equivalence relations on $\mathcal{E}_{A}$. Some equivalence relations associate with $\mathcal{P}_{A}$ in Figure \ref{fig:network18_PA_reduction}. Note that Type 2-A and Type 2-B satisfy the covering relation defined in Definition \ref{def:covering_relation}, which gives valid $\mathcal{E}_{A}/{\sim}$ satisfying $\bowtie$-balanced $E$ as well as index properties.}
\label{tab:network18_equivalence_relations}
\end{table}

Our algorithm finds only the equivalence relation $\bowtie=(1)(25)(3)(4)(68)(7)(9)$ gives a valid $\mathcal{E}_{A}/{\sim}$ (i.e., which gives $\bowtie$-balanced {$E$} and indices {satisfy} the condition in Lemma \ref{lem:balanced_connectivity_matrix_QA}). Lattice reduction is summarised in Figure \ref{fig:network18_EA_reduction}. 

\begin{figure}[h!]
\small
\begin{center}
\begin{tabular}{cc|cc}
\multicolumn{2}{c|}{Before Reduction} & \multicolumn{2}{c}{After Reduction} \\
\hline
\hline
\multicolumn{2}{c|}{$V_{\mathcal{G}}^{P}$} &
\multicolumn{2}{c}{$V_{\mathcal{G}}^{P}/{\sim}$}
\\
& & & \\
\multicolumn{2}{c|}{\includegraphics[scale=0.3]{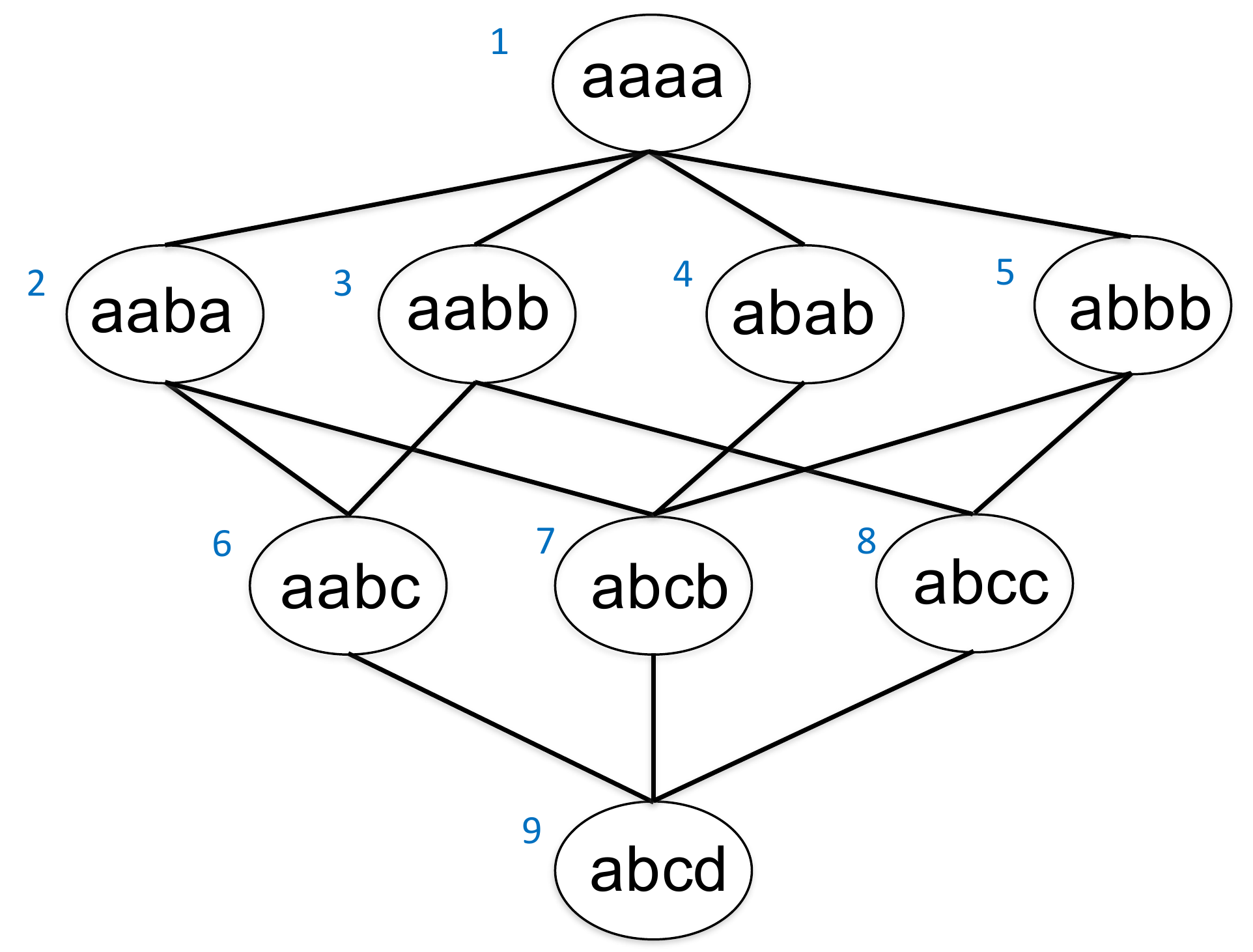}} &
\multicolumn{2}{c}{\includegraphics[scale=0.3]{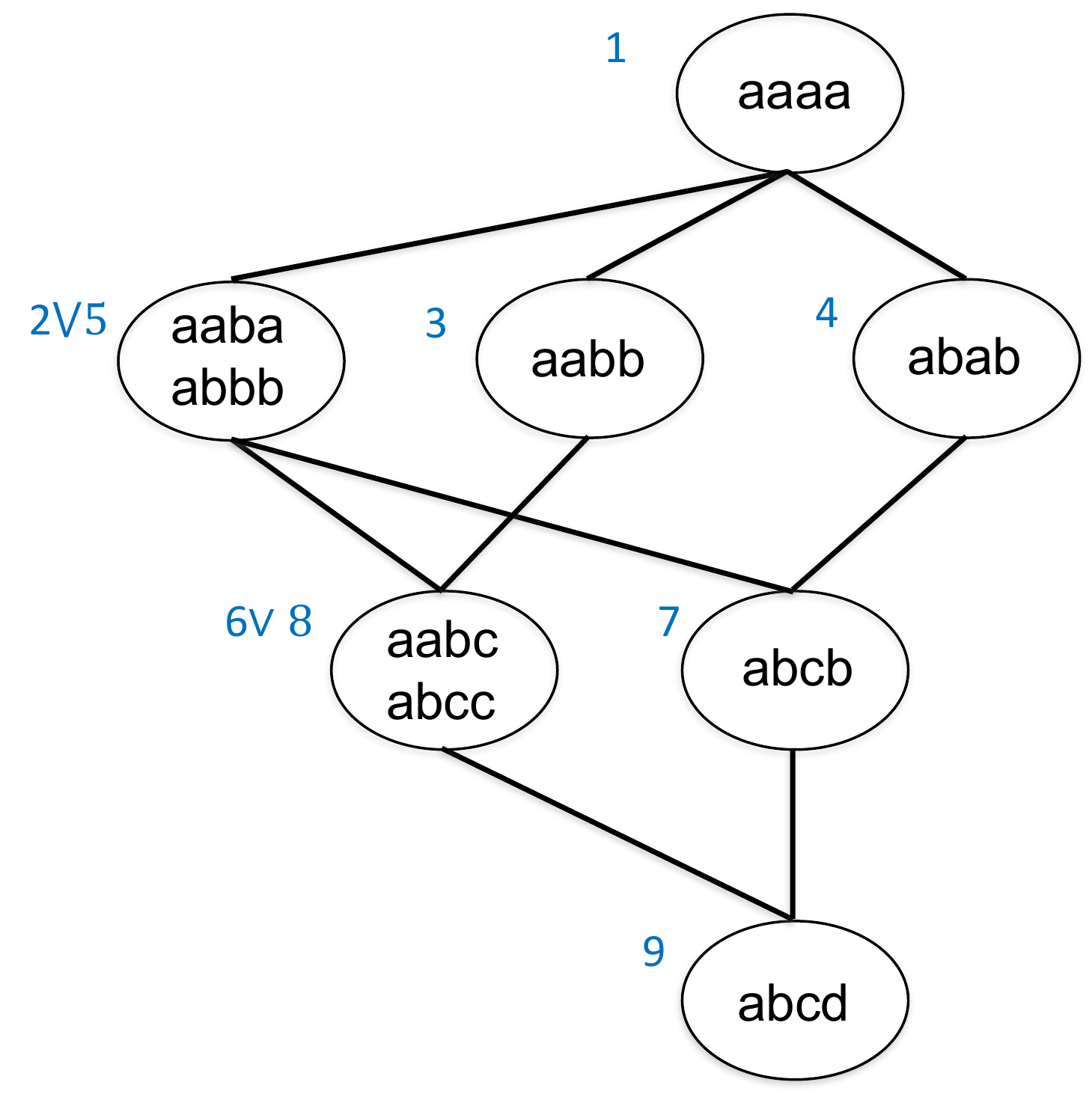}}
\\
& & & \\
& & & \\
\multicolumn{2}{c|}{$\mathcal{E}_{A}$} &
\multicolumn{2}{c}{$\mathcal{E}_{A}/{\sim}$} \\
& & & \\
\multicolumn{2}{c|}{\includegraphics[scale=0.3]{network18_EA.pdf}} & 
\multicolumn{2}{c}{\includegraphics[scale=0.3]{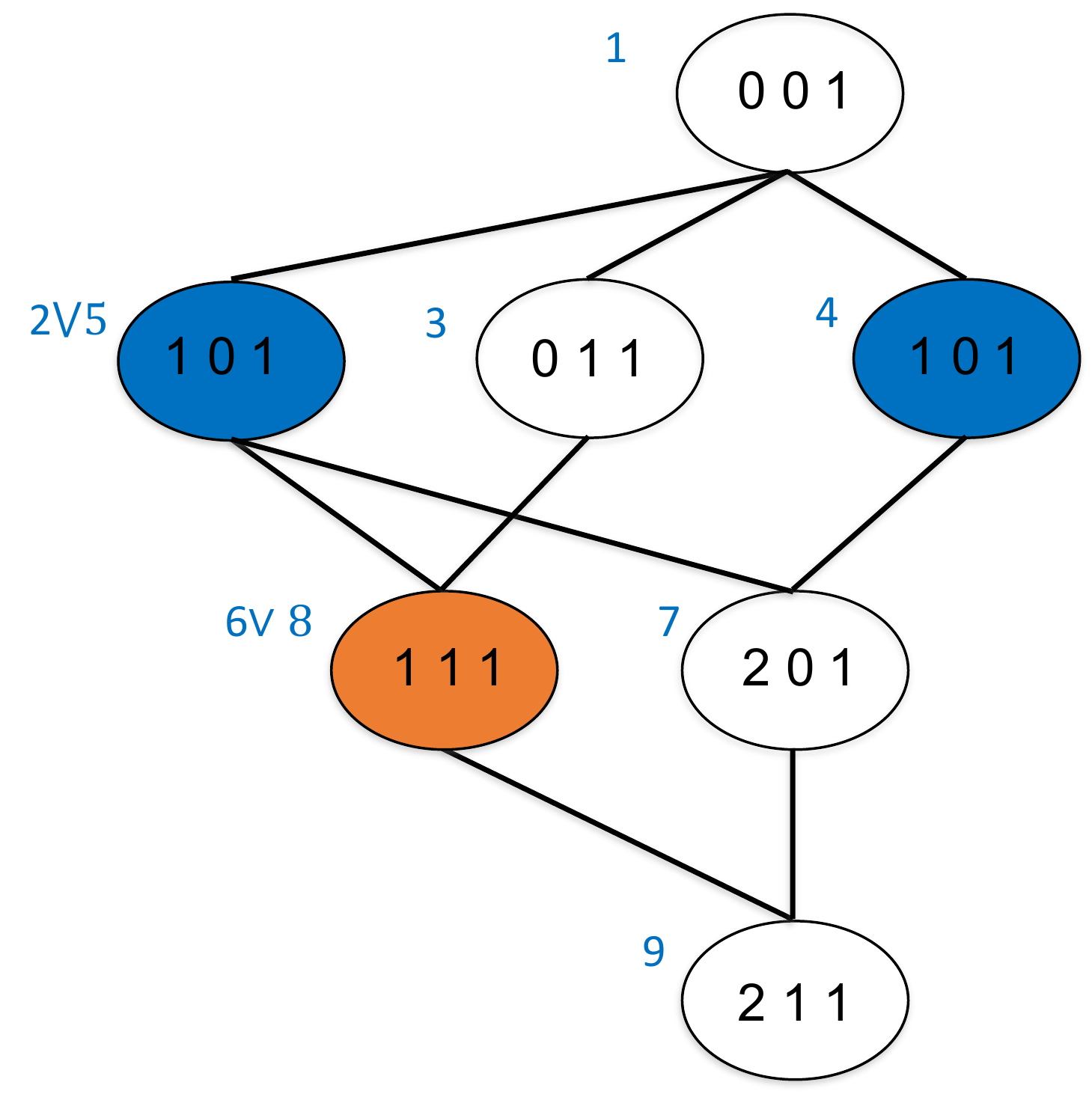}} \\
& & & \\
& & & \\
Type 2-A $\mathcal{P}_{A}$ &
Type 2-B $\mathcal{P}_{A}$ &
Type 2-A {$\mathcal{P}_{A}/{=}$} &
Type 2-B {$\mathcal{P}_{A}/{=}$} \\
& & & \\
\includegraphics[scale=0.27]{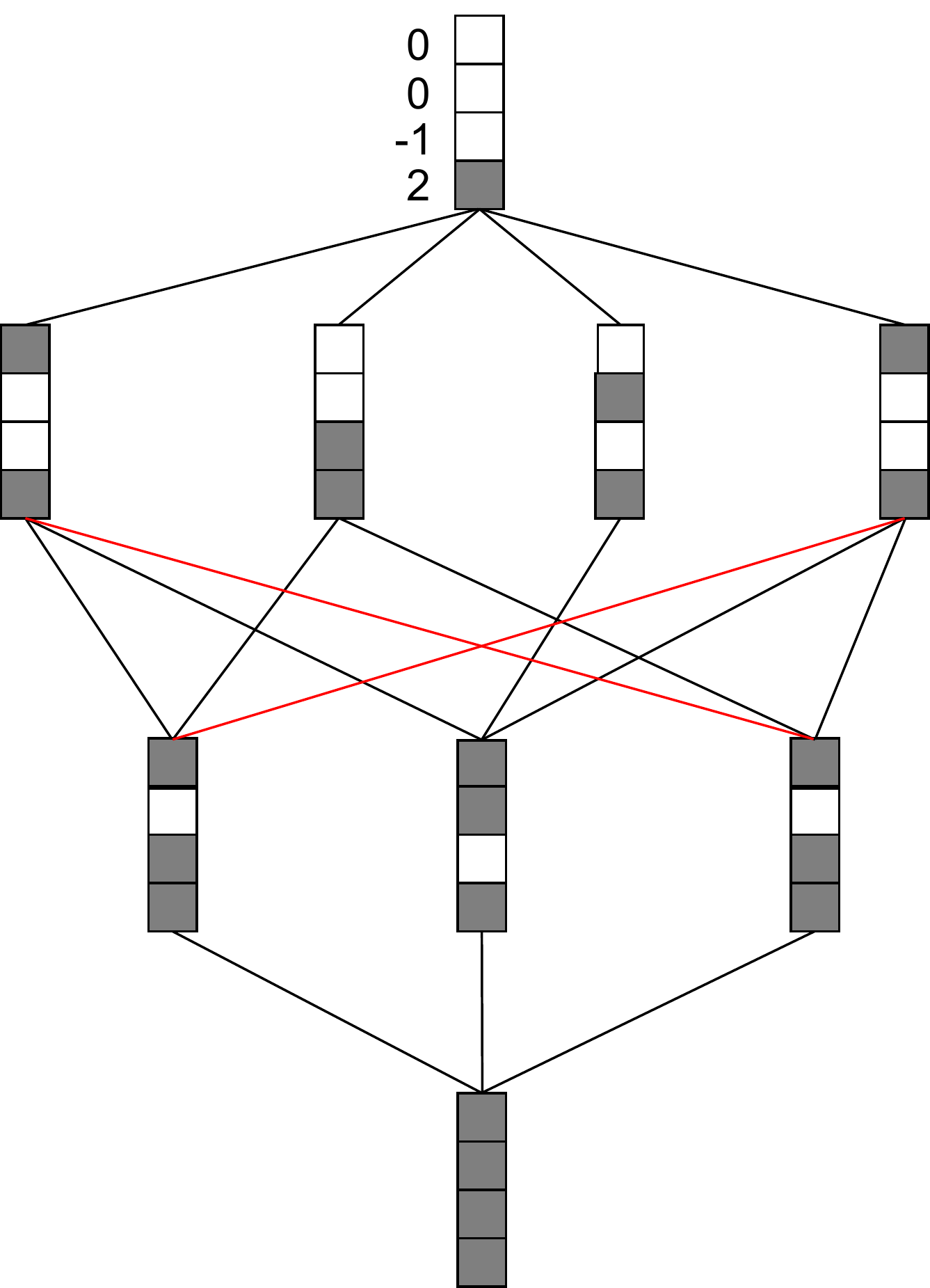} &
\includegraphics[scale=0.27]{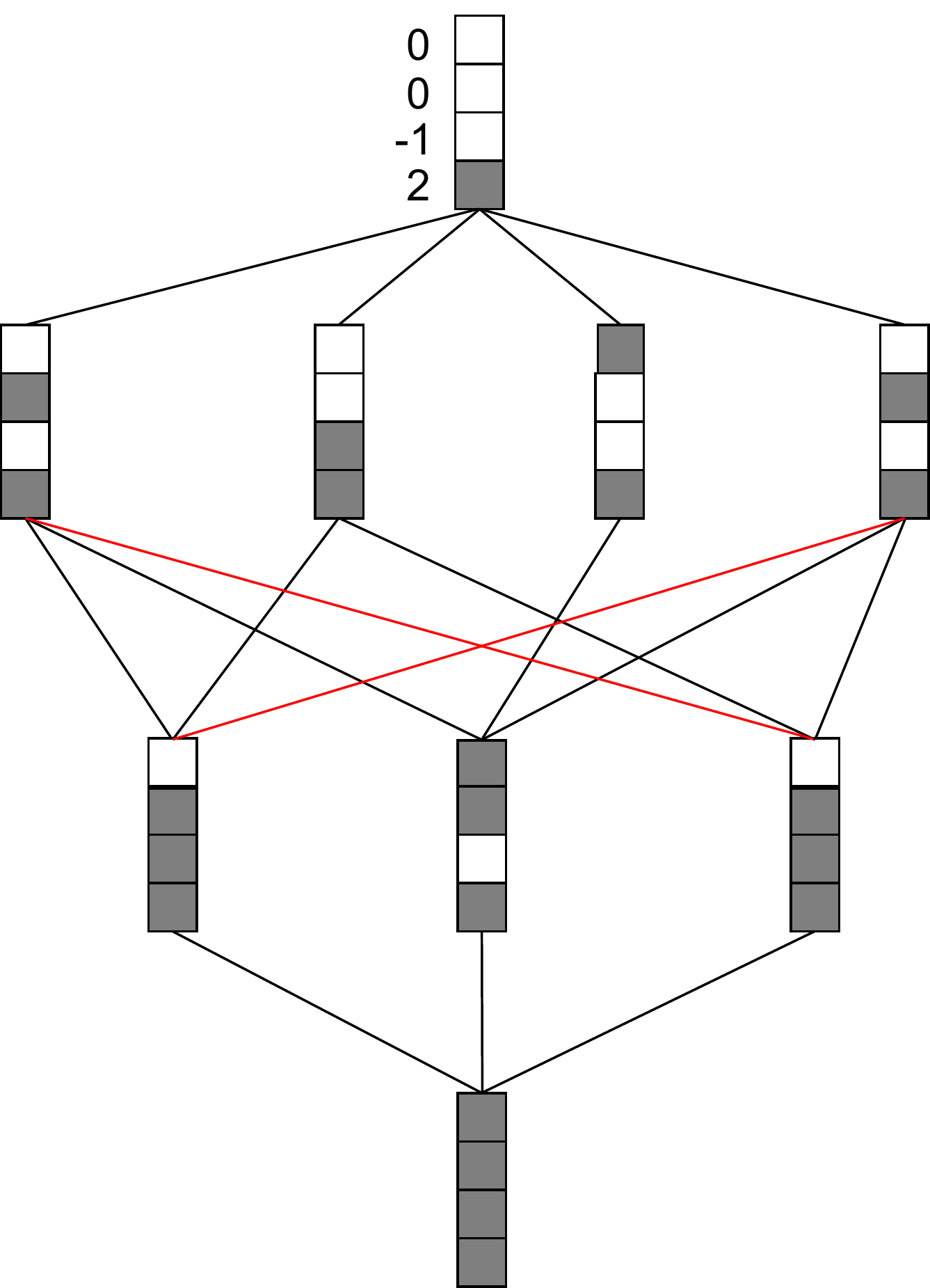}  &
\includegraphics[scale=0.27]{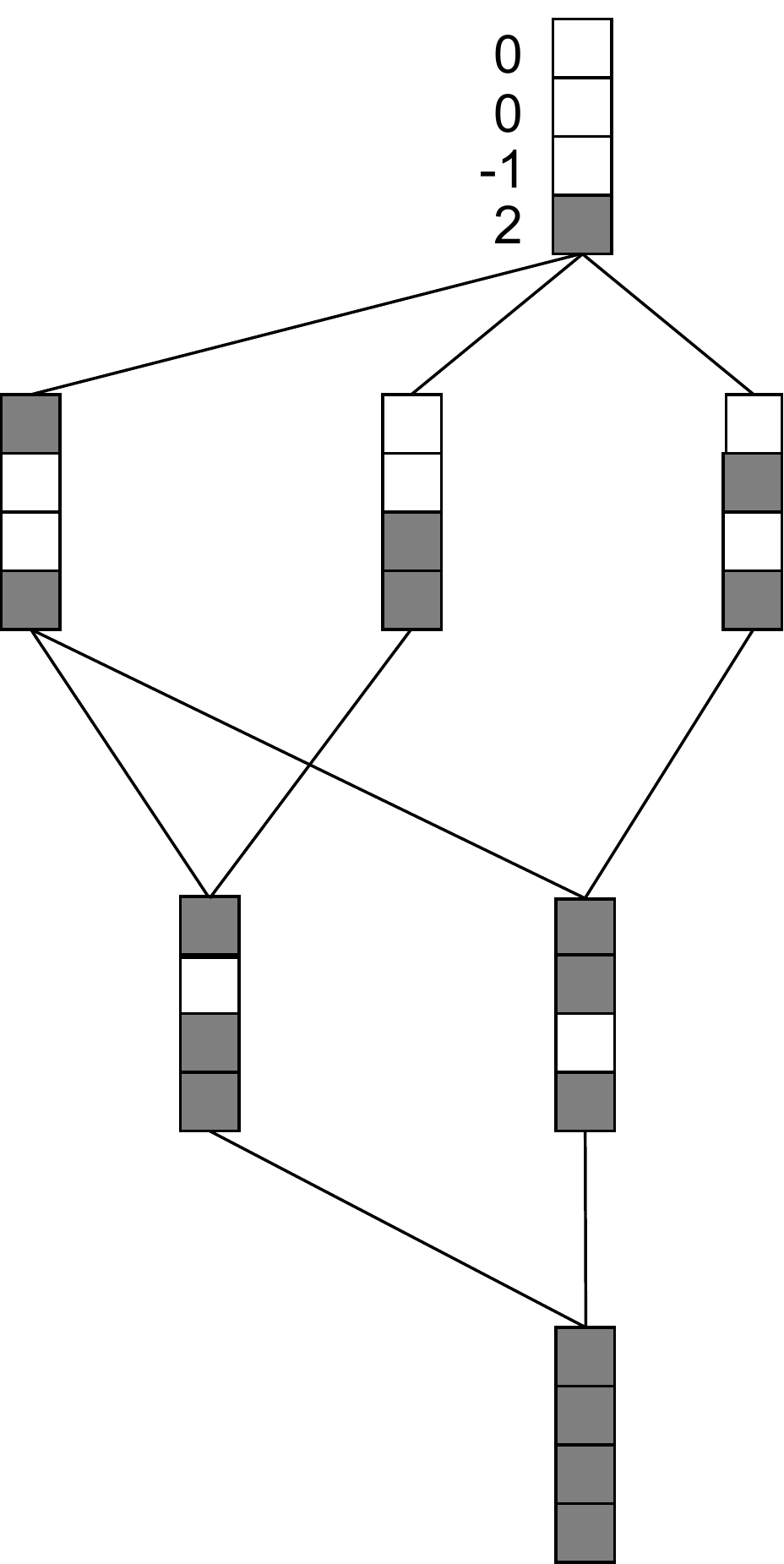} &
\includegraphics[scale=0.27]{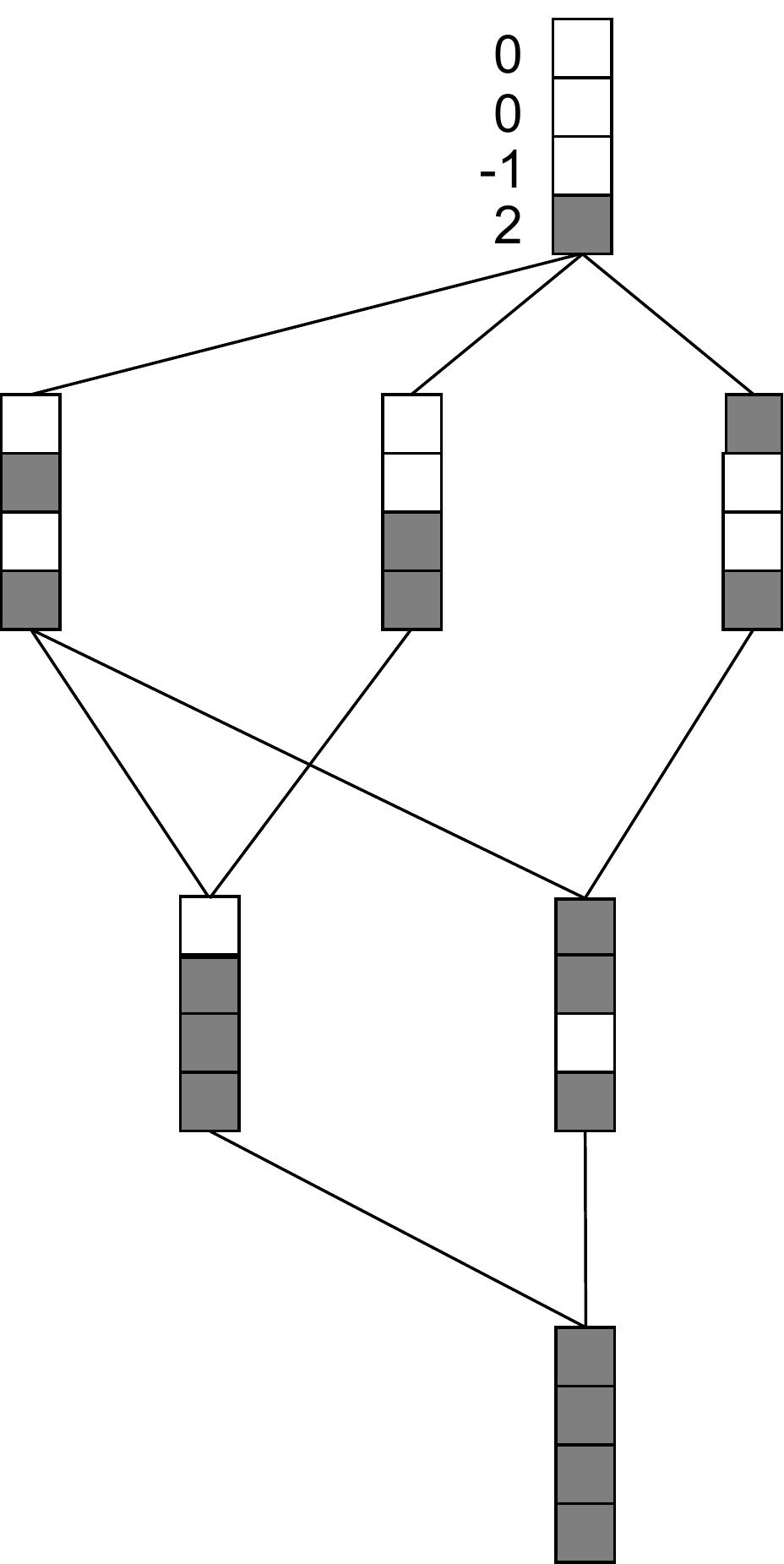} 
\end{tabular}
\end{center}
\caption{Before and after reduction on the lattice of synchrony subspaces $V_{\mathcal{G}}^{P}$, $\mathcal{E}_{A}$ and $\mathcal{P}_{A}$. For $(x_{1},x_{2},x_{3},x_{4})\in V_{\mathcal{G}}^{P}$, cell coordinate equality is given by the same symbol, e.g., $(x_{1},x_{2},x_{3},x_{4})=(a,a,b,b)$ means $x_{1}=x_{2}$ and $x_{3}=x_{4}$. For example in $V_{\mathcal{G}}^{P}/{\sim}$, the most left node at rank $2$ corresponds to two synchrony subspaces $(a,a,b,a)$ and $(a,b,b,b)$ in $V_{\mathcal{G}}^{P}$. Note that there are two possible $\mathcal{P}_{A}$ associated with the unique $V_{\mathcal{G}}^{P}$ as well as {$\mathcal{E}_{A}$}. This is due to the equal size of Jordan blocks associated with the eigenvalue $\lambda=0$.}
\label{fig:network18_EA_reduction}
\end{figure}

There are in total $8$ different possible $\mathcal{P}_{A}(1,1,1,1)$, which can be classified into three distinct types as shown in Figure \ref{fig:network18_PA_reduction}. Notice that only Type $2$ $\mathcal{P}_{A}$ satisfies the covering relation defined in Definition \ref{def:covering_relation}. Accordingly the equivalence relations $\bowtie$ associated with other types of $\mathcal{P}_{A}$ do not give $\bowtie$-balanced $E$.

\begin{sidewaysfigure}
\begin{center}
\begin{tabular}{cccc|cc}
\hline
\multicolumn{4}{l|}{Type 1 $\mathcal{P}_{A}$} & \multicolumn{2}{l}{Type 1 {$\mathcal{P}_{A}/{=}$}} \\
& & & & & \\
A & B & C & D & A=B=C=D & \\
\includegraphics[scale=0.22]{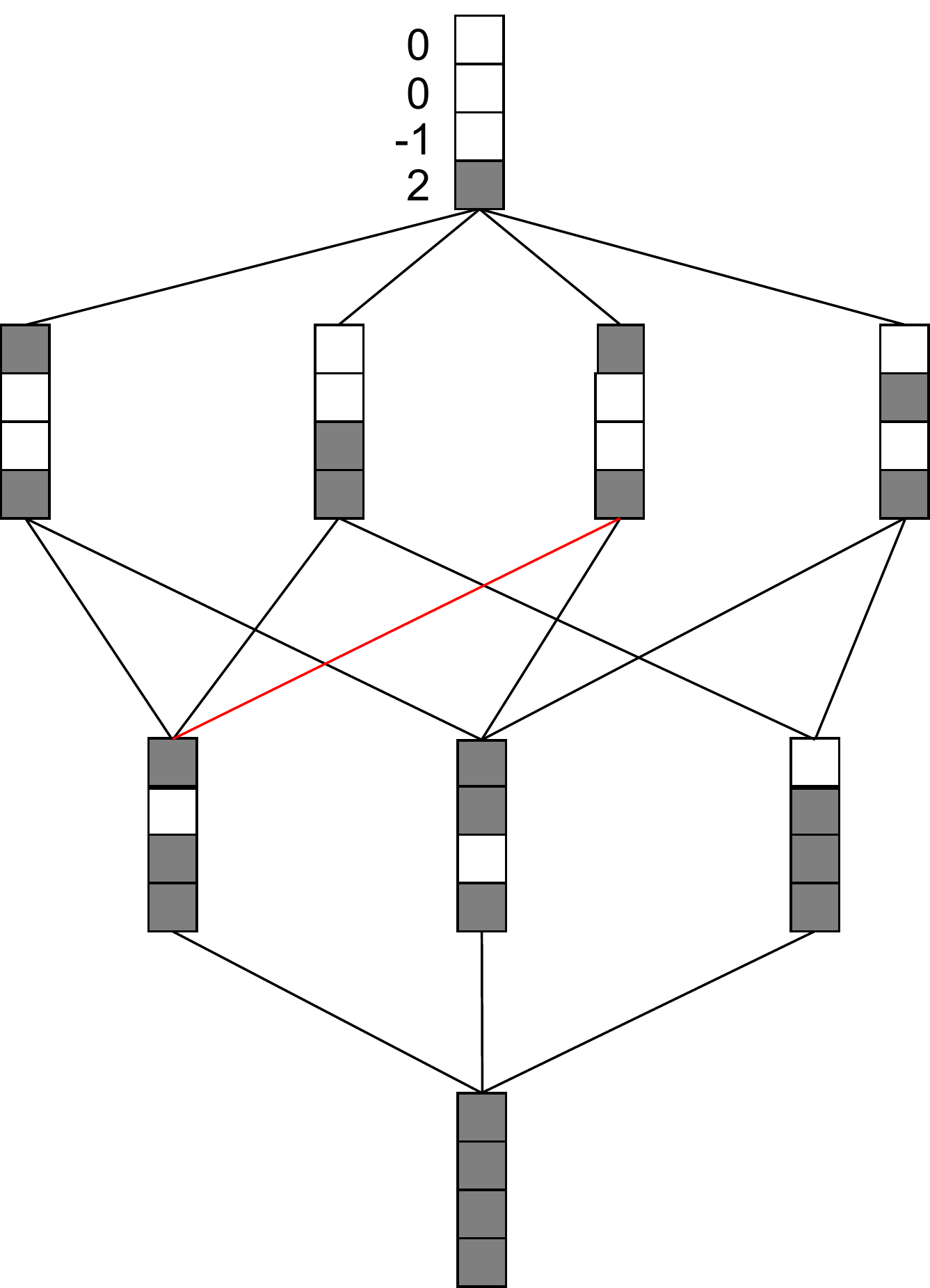} & 
\includegraphics[scale=0.22]{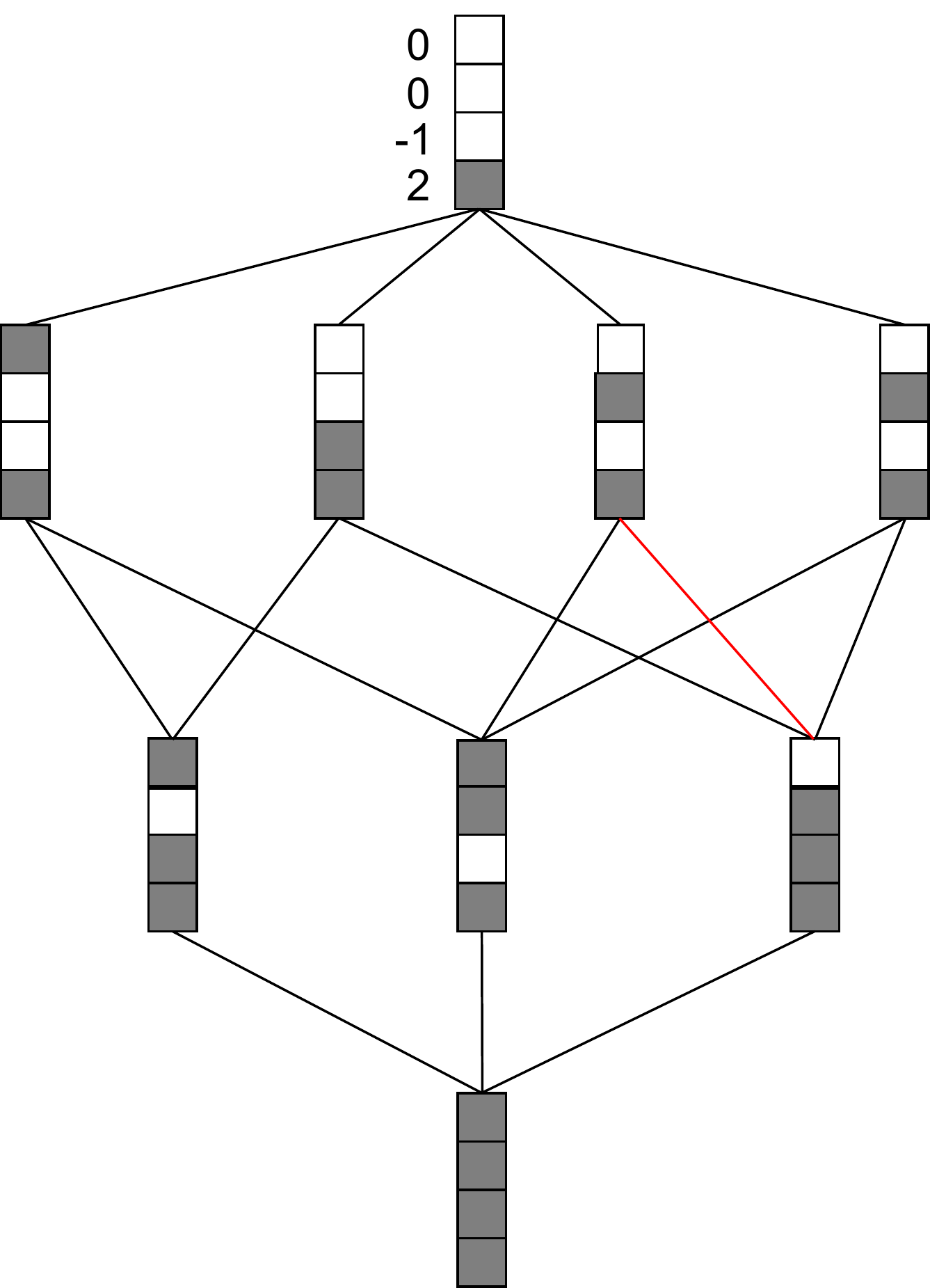} & 
\includegraphics[scale=0.22]{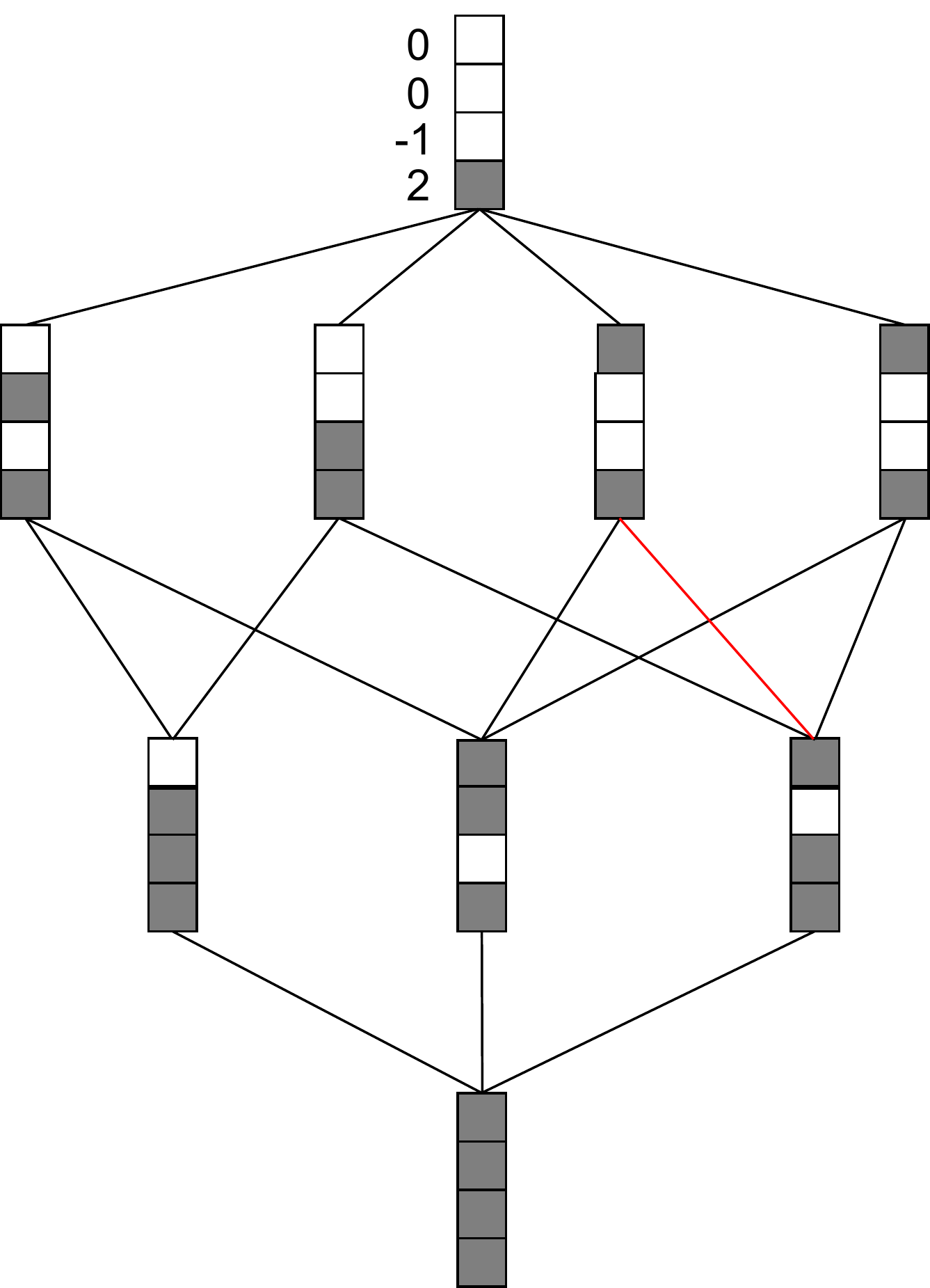} & 
\includegraphics[scale=0.22]{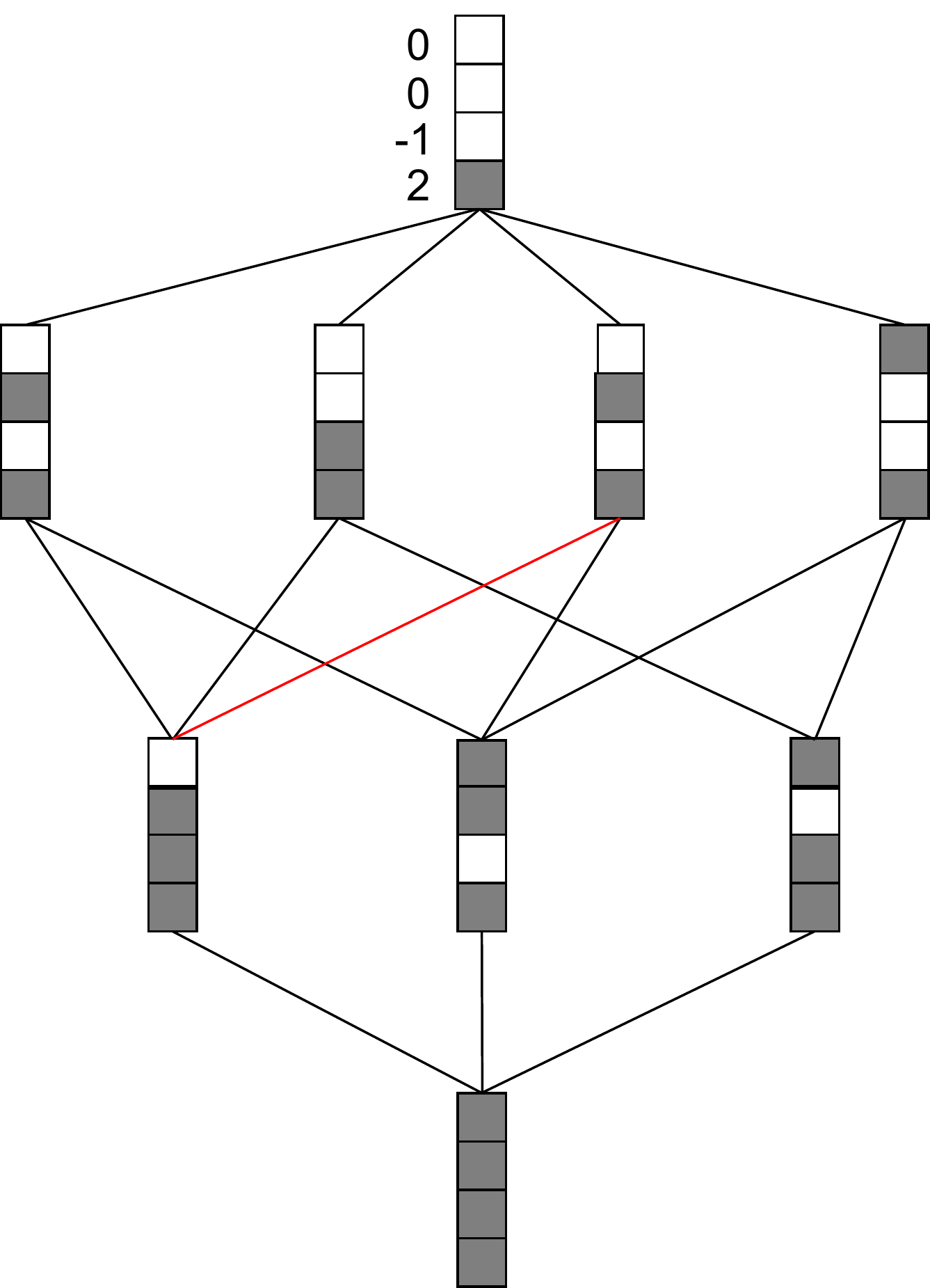}  &
\includegraphics[scale=0.22]{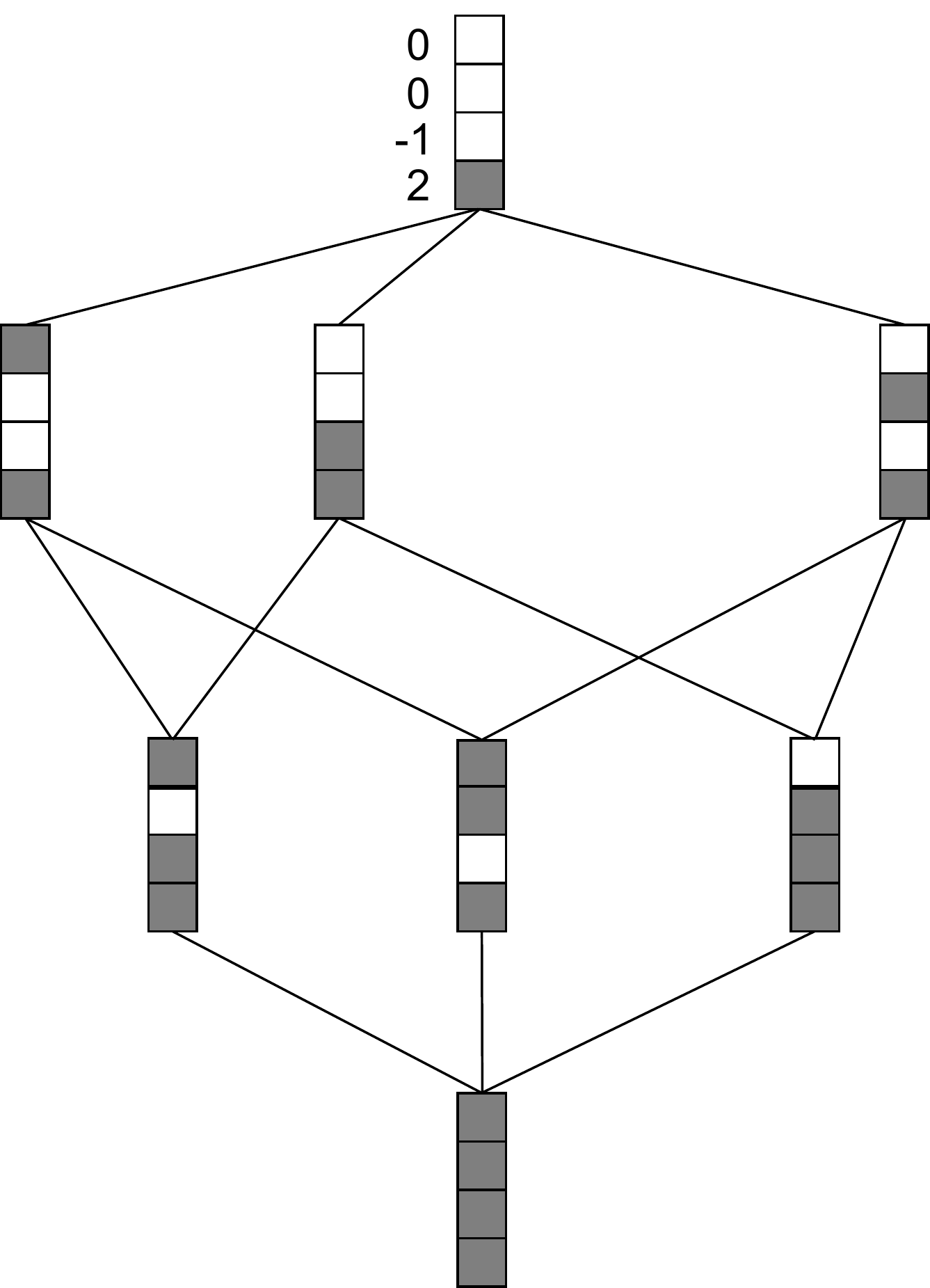}  &
\\
\hline
\hline
\multicolumn{4}{l|}{Type 2 $\mathcal{P}_{A}$} & \multicolumn{2}{l}{Type 2 {$\mathcal{P}_{A}/{=}$}} \\
& & & & & \\
A & B & & & A & B \\
\includegraphics[scale=0.22]{network18_PA_2_1} &
\includegraphics[scale=0.22]{network18_PA_2_2} &  
&  
&
\includegraphics[scale=0.22]{network18_PA_2_1_reduction} &
\includegraphics[scale=0.22]{network18_PA_2_2_reduction} 
\\
\hline
\hline
\multicolumn{4}{l|}{Type 3 $\mathcal{P}_{A}$} & \multicolumn{2}{l}{Type 3 {$\mathcal{P}_{A}/{=}$}} \\
& & & & & \\
A & B & & & A & B\\
\includegraphics[scale=0.22]{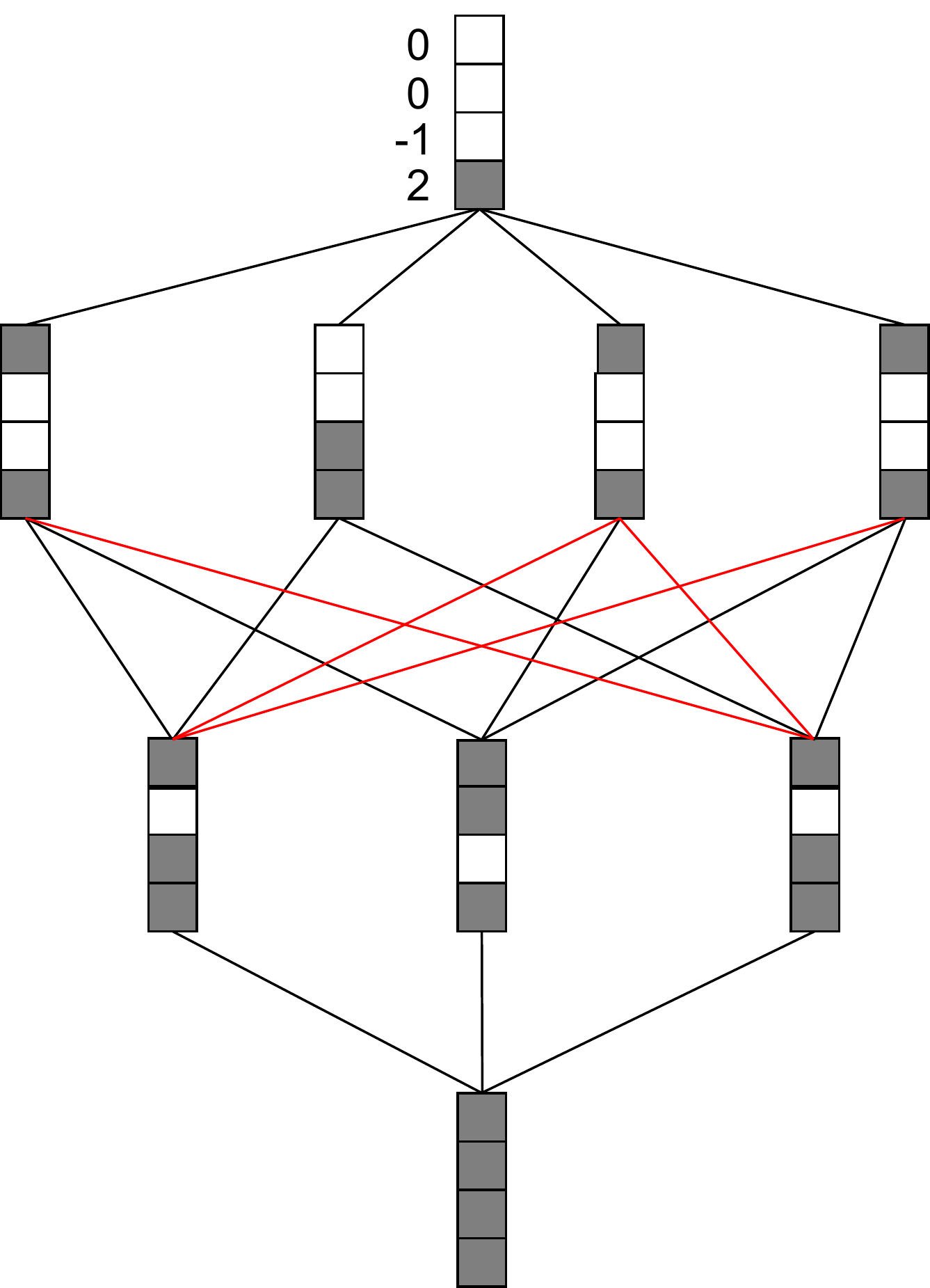} & 
\includegraphics[scale=0.22]{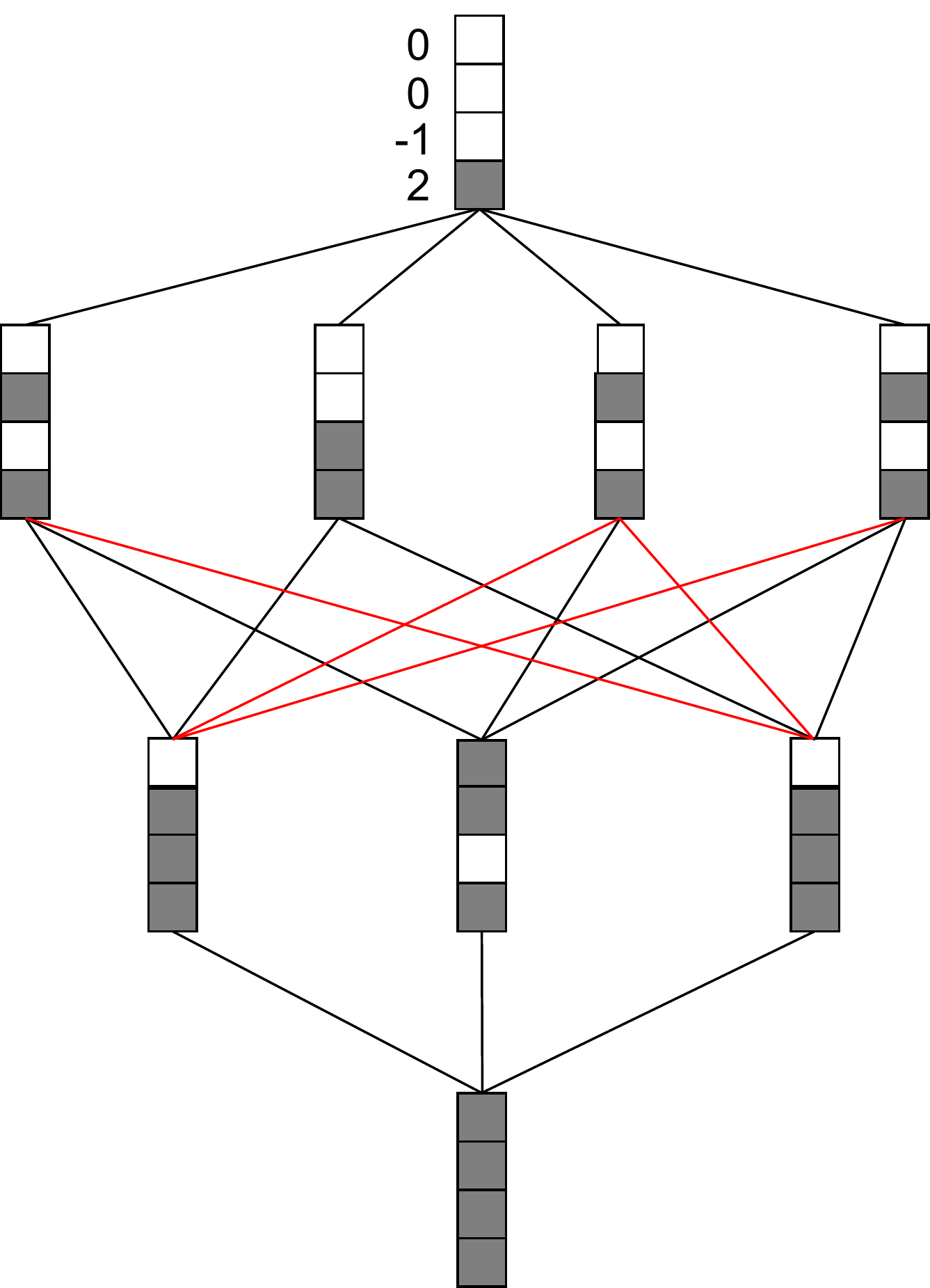} &  
&
&
\includegraphics[scale=0.22]{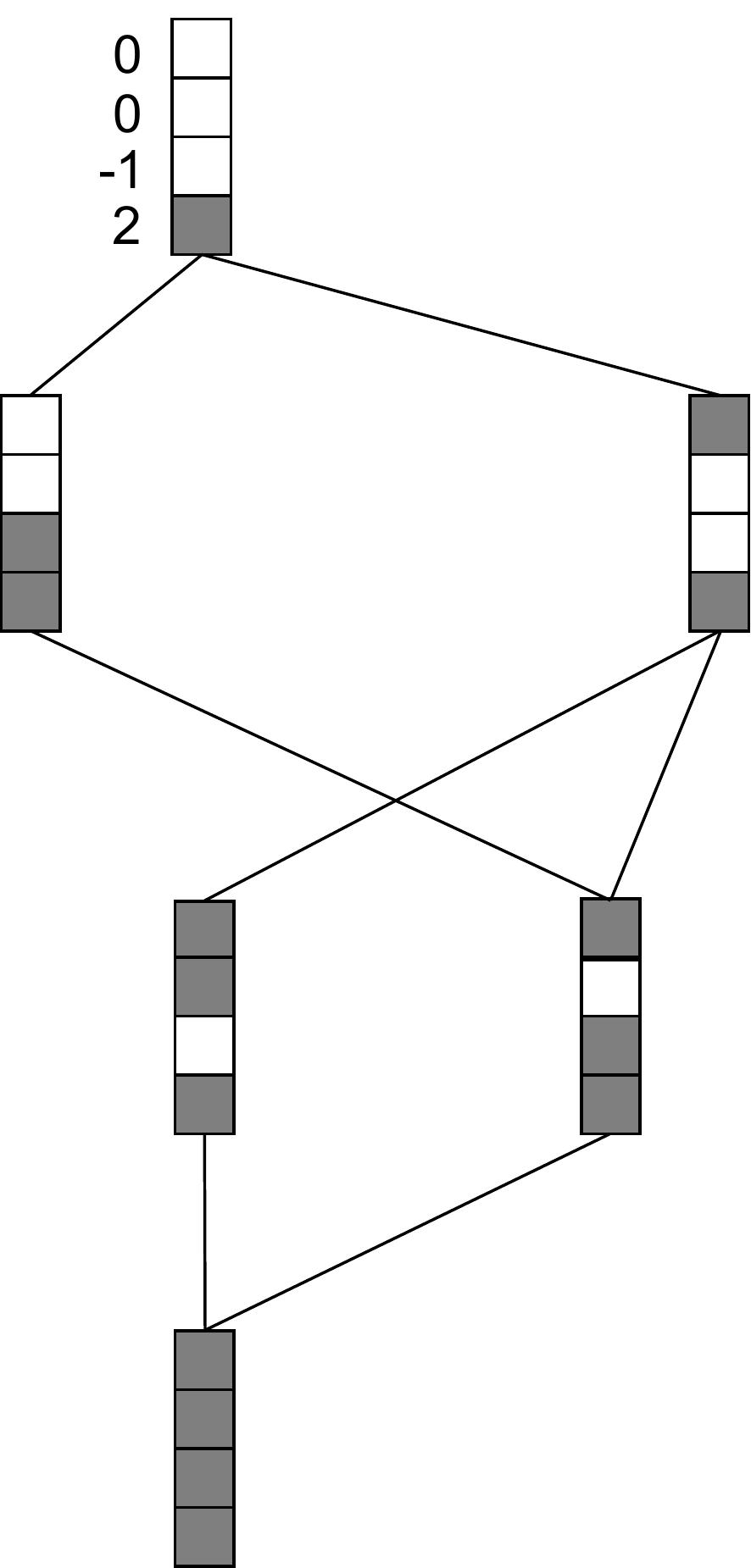} &
\includegraphics[scale=0.22]{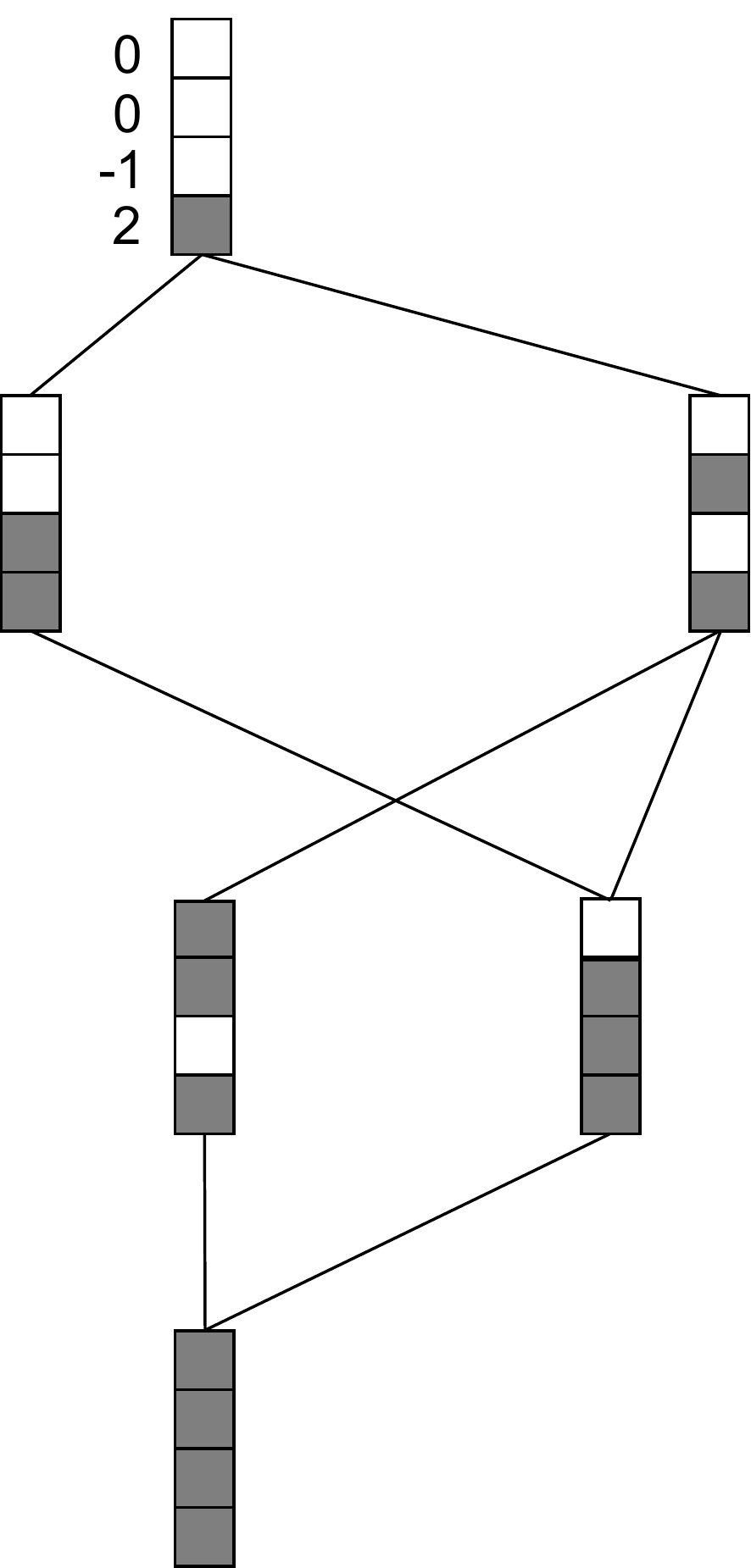} \\
\hline
\end{tabular}
\end{center}
\caption{Three distinct types of $\mathcal{P}_{A}$ and its associated {$\mathcal{P}_{A}/{=}$}. Only Type $2$ $\mathcal{P}_{A}$ satisfies the covering relation defined in Definition \ref{def:covering_relation}.}
\label{fig:network18_PA_reduction}
\end{sidewaysfigure}

\end{ex}

\clearpage

\subsection{Multiple $\mathcal{E}_{A}/{\sim}$}
\begin{ex}
\rm
\label{ex:4_cell_network2}
Let $A:\bc^4\to \bc^4$ be the adjacency matrix of the $4$-cell network shown in Figure \ref{fig:4_cell_network2}. 
\begin{figure}[h!]
\begin{center}
\begin{tabular}{ccll}
Network $\mathcal{G}$ & Adjacency matrix $A$ & eigenvalues & eigenvectors\\
\hline
\multirow{4}{*}{
	\includegraphics[scale=0.25]{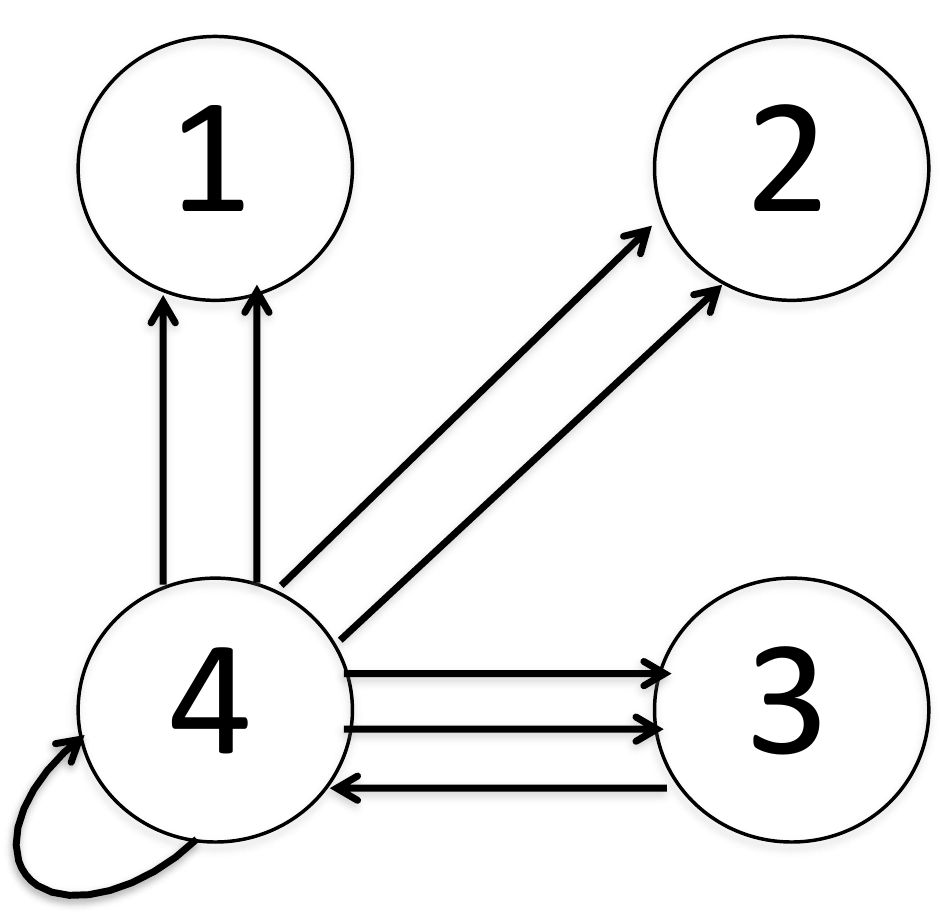}
} &
\multirow{4}{*}{$\left(\begin{array}{cccc}
0 & 0 & 0 & 2\\
0 & 0 & 0 & 2 \\
0 & 0 & 0 & 2 \\
0 & 0 & 1 & 1  
\end{array}\right)$} &
$\lambda_{1}=0$ & $(1,0,0,0)$\\
& & $\lambda_{2}=0$ & $(0,1,0,0)$\\
& & $\lambda_{3}=-1$ & $(-2,-2,-2,1)$\\
& & $\lambda_{4}=2$  & $(1,1,1,1)$ \\
& &
\end{tabular}
\end{center}
\caption{$4$-cell regular network $\mathcal{G}$ with corresponding adjacency matrix $A$ and its eigenvalues. The repeated eigenvalue $\lambda_{1}=\lambda_{2}=0$ has algebraic multiplicity $2$ and geometric multiplicity $2$.}
\label{fig:4_cell_network2}
\end{figure}

There are $20$ possible equivalence relations on $\mathcal{E}_{A}$ to check in order to find $\mathcal{E}_{A}/{\sim}$. Our algorithm finds $4$ equivalence relations on $\mathcal{E}_{A}$ give a candidate $\mathcal{E}_{A}/{\sim}$. $3$ out of $4$ equivalence relations give a topologically equivalent reduced lattice, where a representative reduction is shown in Figure \ref{fig:4_cell_network2_EA_reduction_Type1_1}. Each has two possible $\mathcal{P}_{A}$ giving 6 which we call Type $1$. The remaining equivalence relation is shown in Figure \ref{fig:4_cell_network2_EA_reduction_Type2}, giving two possible $\mathcal{P}_{A}$ which we call Type $2$. All $8$ of $\mathcal{P}_{A}$ satisfy the covering relation defined in Definition \ref{def:covering_relation}, however, we can identify the correct $\mathcal{P}_{A}$ to be Type 1 using the following result which we have proved, and will appear in a successive manuscript.

\begin{proposition}
\label{prop:max_simple_eigenvalue_lattice_n_cell}
Let $A:\bc^n\to \bc^n$ be an adjacency matrix of a regular network $\mathcal{G}$ whose valency is $v$ and $\sigma(A)=\{v,\lambda,\cdots,\lambda\}$, where $\lambda\ne v$ has a  geometric multiplicity $n-1$. Let $\mathcal{L}_{A}(1,\ldots,1):=\{r\in \mathcal{L}_{M}(1,\ldots,1)\,;\, r\ge (0,\dots,0, 1)\}$. Then, $\mathcal{P}_{A}/{=}$ is isomorphic to $\mathcal{L}_{A}$.
\end{proposition}

The structure of $\mathcal{P}_{A}$ for the node $(1,1,0,1)$ in Figure \ref{fig:4_cell_network2_EA_reduction_Type1_1} and Figure \ref{fig:4_cell_network2_EA_reduction_Type2} can be seen by ignoring the third box corresponding to the $1\times 1$ Jordan block associated with the eigenvalue $-1$ as shown in Figure \ref{fig:3cell_reduction_summary}. Type $1$ $\mathcal{P}_{A}$ satisfies $\mathcal{P}_{A}/{=}\, \cong\, \mathcal{L}_{A}(1,1,1)$. However, this is not satisfied with Type $2$ $\mathcal{P}_{A}$. 

\begin{figure}[h!]
\small
\begin{center}
\begin{tabular}{cc|cc}
\multicolumn{2}{c|}{Before Reduction} & \multicolumn{2}{c}{After Reduction} \\
\hline
\hline
Type $1$ $\mathcal{P}_{A}$ & Type $2$ $\mathcal{P}_{A}$ & Type $1$ $\mathcal{P}_{A}/=$ & Type $2$ $\mathcal{P}_{A}/=$\\
& & &  \\
\includegraphics[scale=0.25]{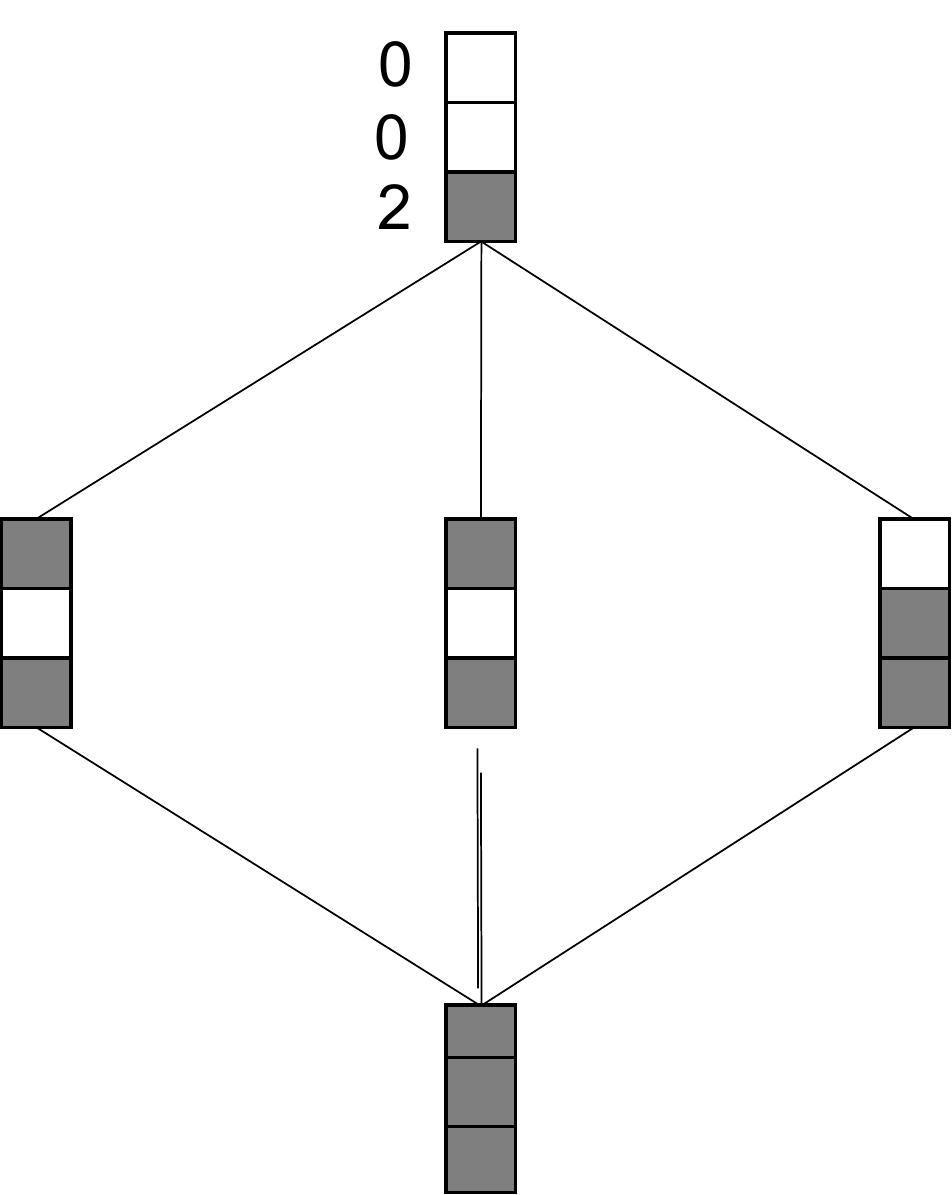} & 
\includegraphics[scale=0.25]{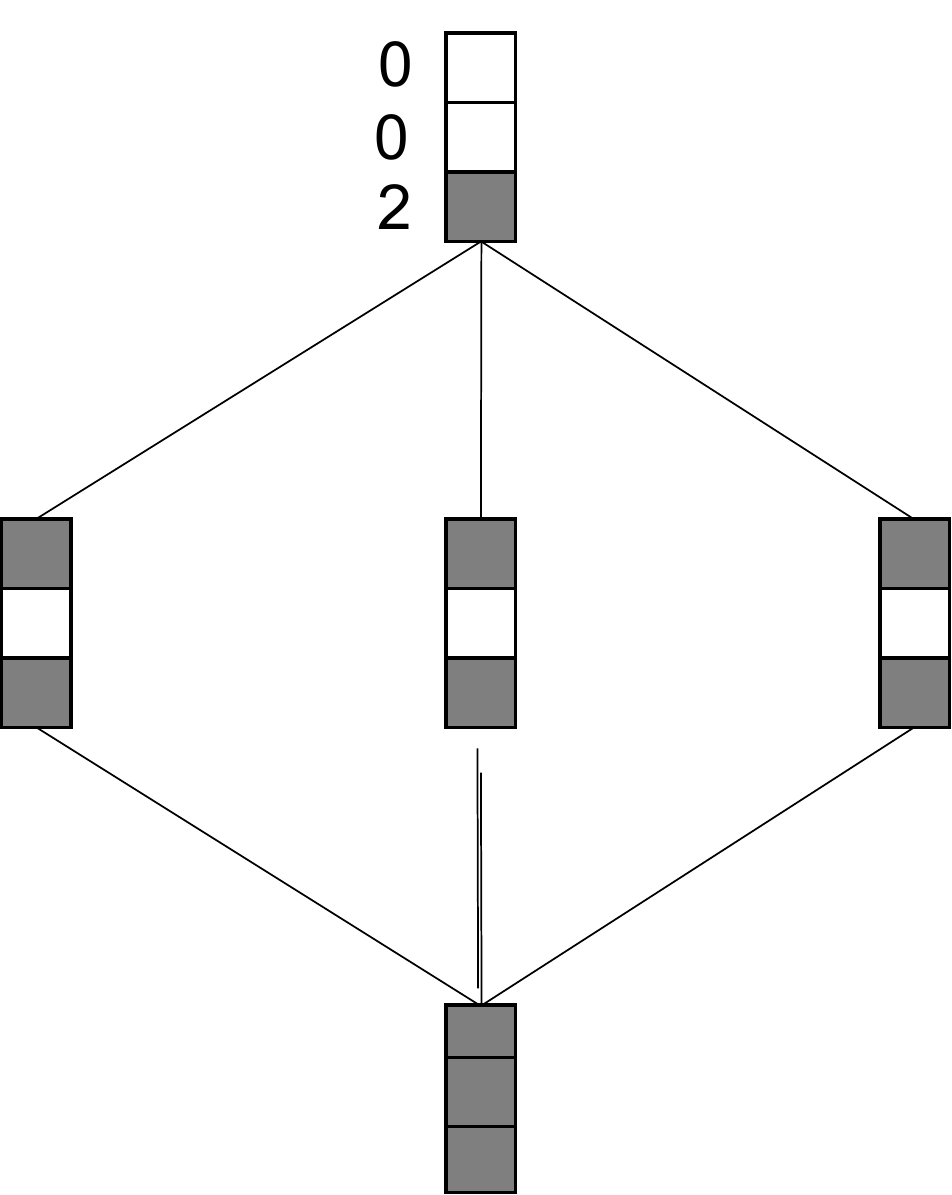} &
\includegraphics[scale=0.25]{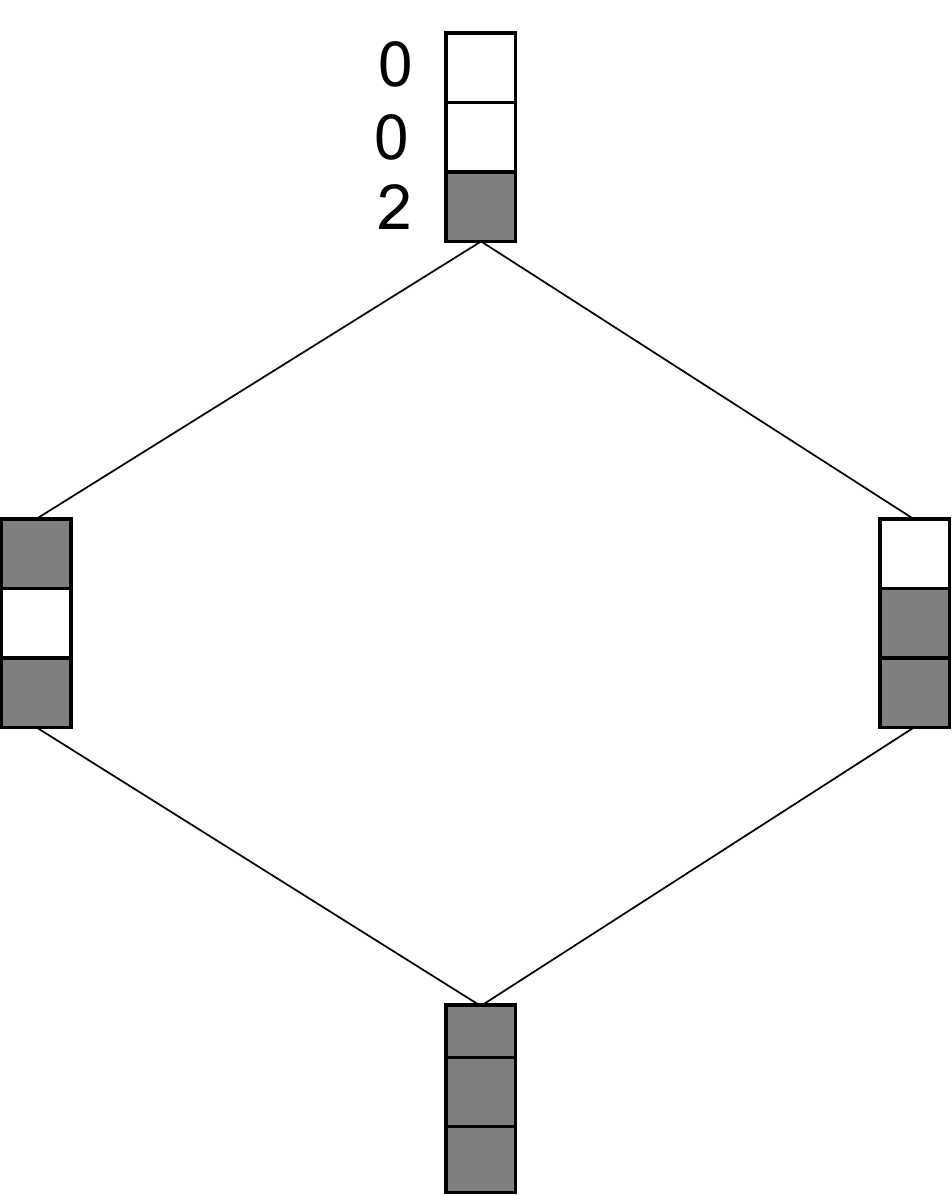} &
\includegraphics[scale=0.25]{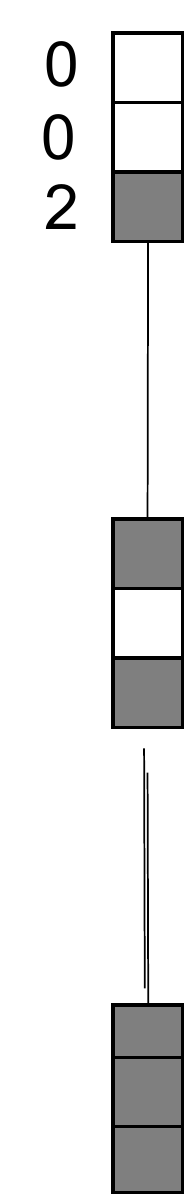}
\end{tabular}
\end{center}
\caption{Before and after reduction on $\mathcal{P}_{A}$ of $3$-cell regular network whose adjacency matrix has the repeated eigenvalues $\lambda_{1}=\lambda_{2}=0$ with the geometric multiplicity $2$ and $\lambda_{3}=2$. Note that only representatives from each type are shown. Type $1$ is the expected reduction since Type $1$ $\mathcal{P}_{A}/{=}\, \cong \,\mathcal{L}_{A}(1,1,1)$.}
\label{fig:3cell_reduction_summary}
\end{figure}

\begin{figure}[h!]
\small
\begin{center}
\begin{tabular}{cc|cc}
\multicolumn{2}{c|}{Before Reduction} & \multicolumn{2}{c}{After Reduction} \\
\hline
\hline
\multicolumn{2}{c|}{$V_{\mathcal{G}}^{P}$} &
\multicolumn{2}{c}{$V_{\mathcal{G}}^{P}/{\sim}$}
\\
& & & \\
\multicolumn{2}{c|}{\includegraphics[scale=0.3]{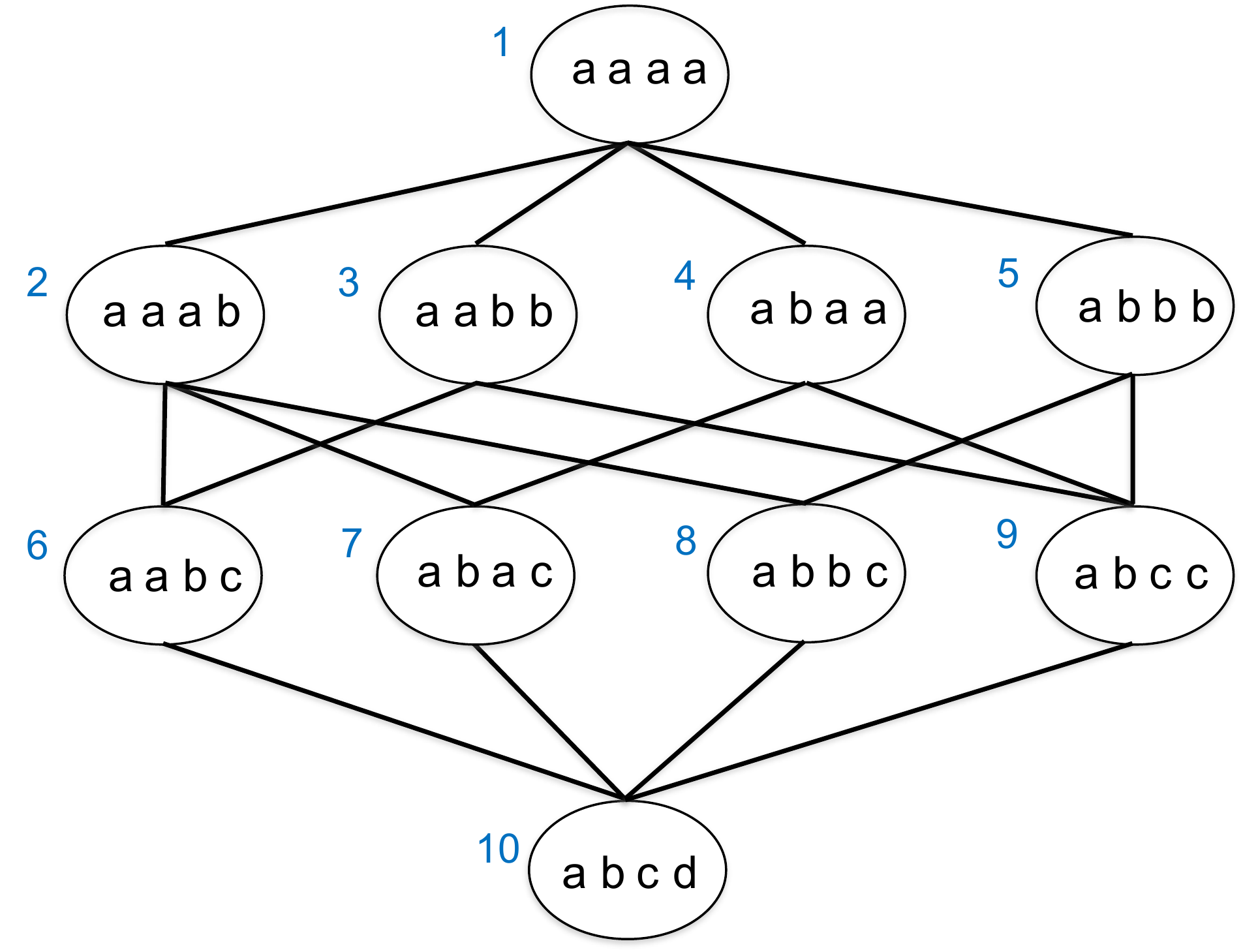}} &
\multicolumn{2}{c}{\includegraphics[scale=0.3]{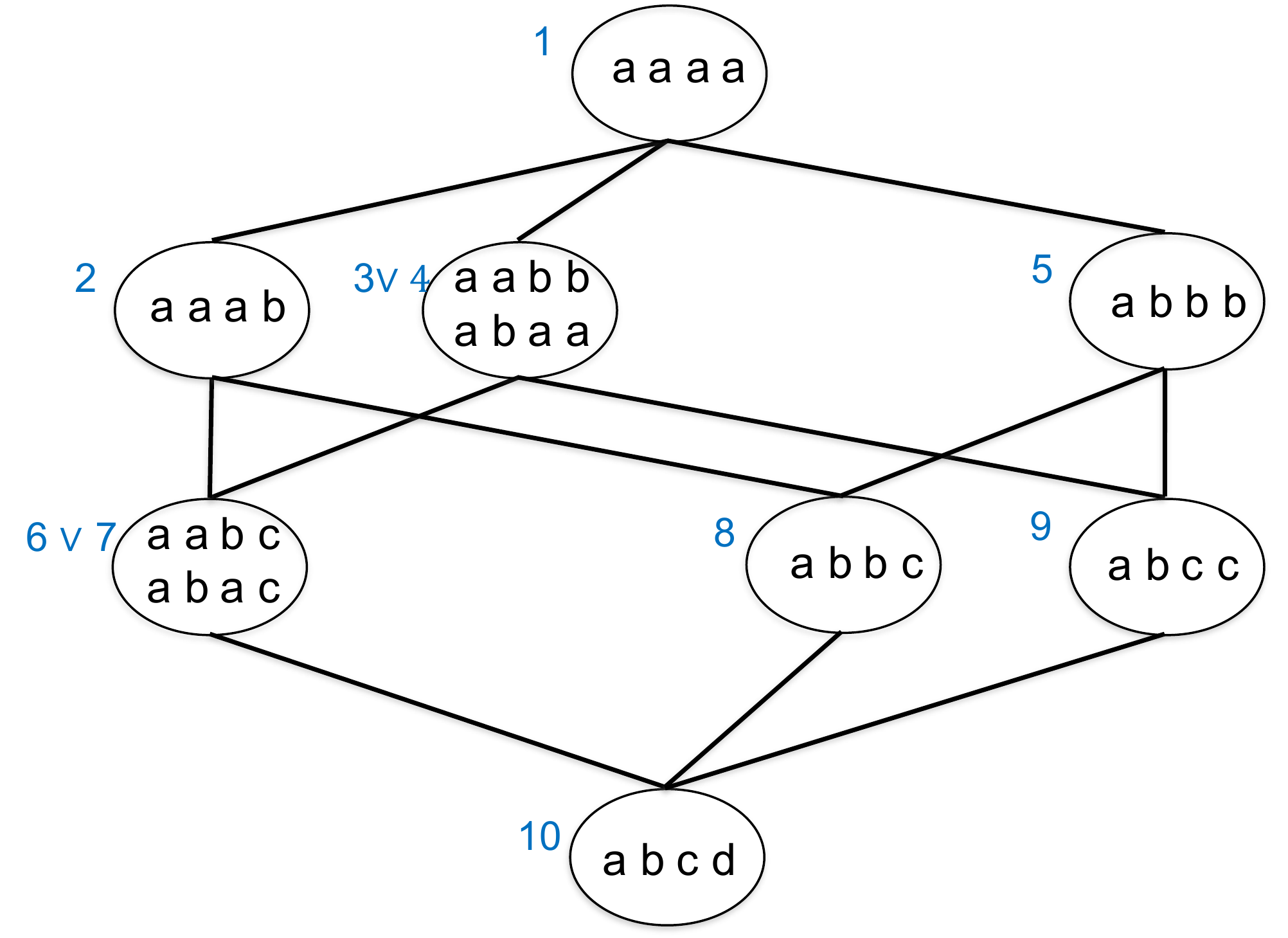}}
\\
& & & \\
& & & \\
\multicolumn{2}{c|}{$\mathcal{E}_{A}$} &
\multicolumn{2}{c}{$\mathcal{E}_{A}/{\sim}$} \\
& & & \\
\multicolumn{2}{c|}{\includegraphics[scale=0.3]{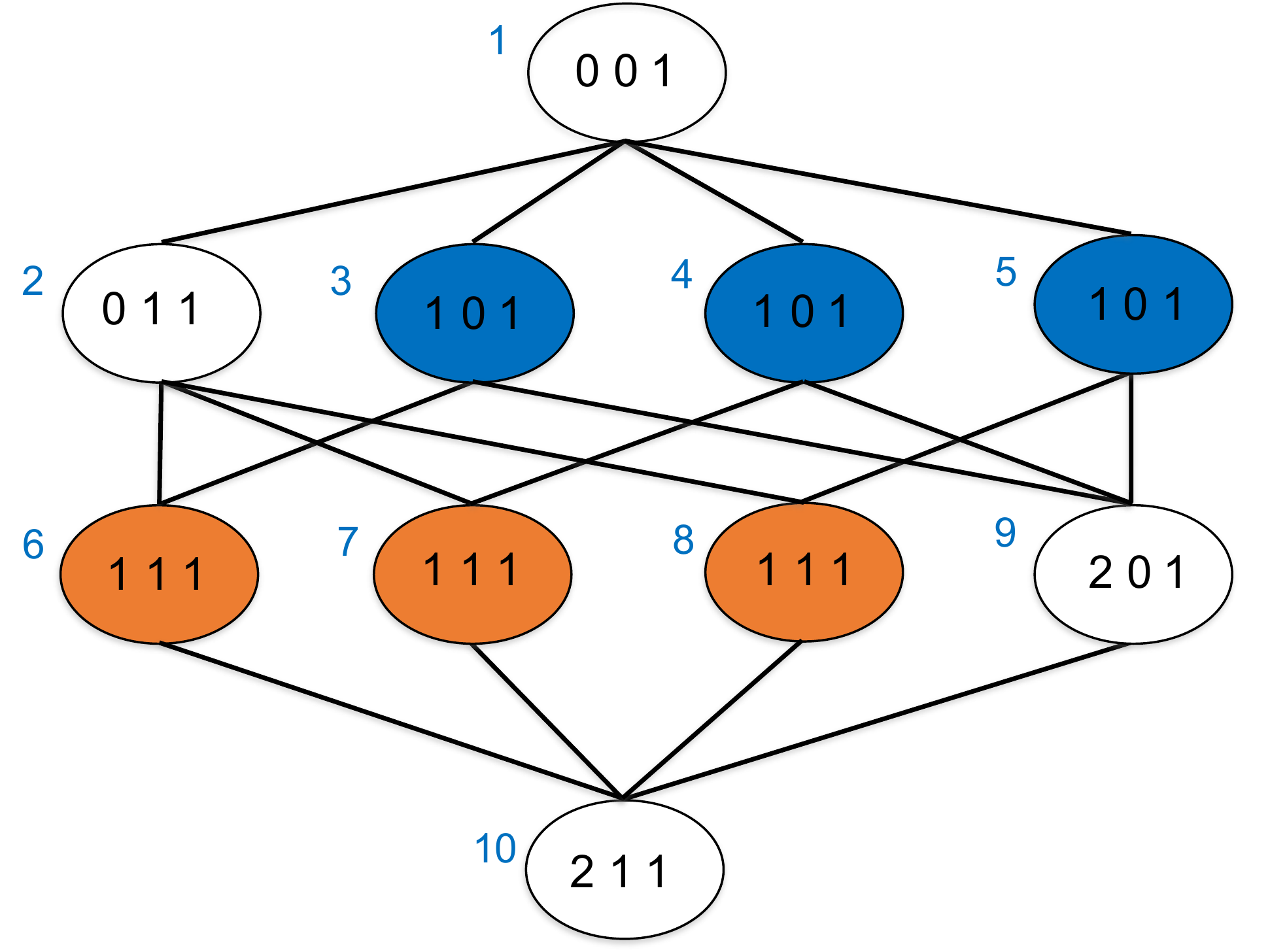}} & 
\multicolumn{2}{c}{\includegraphics[scale=0.3]{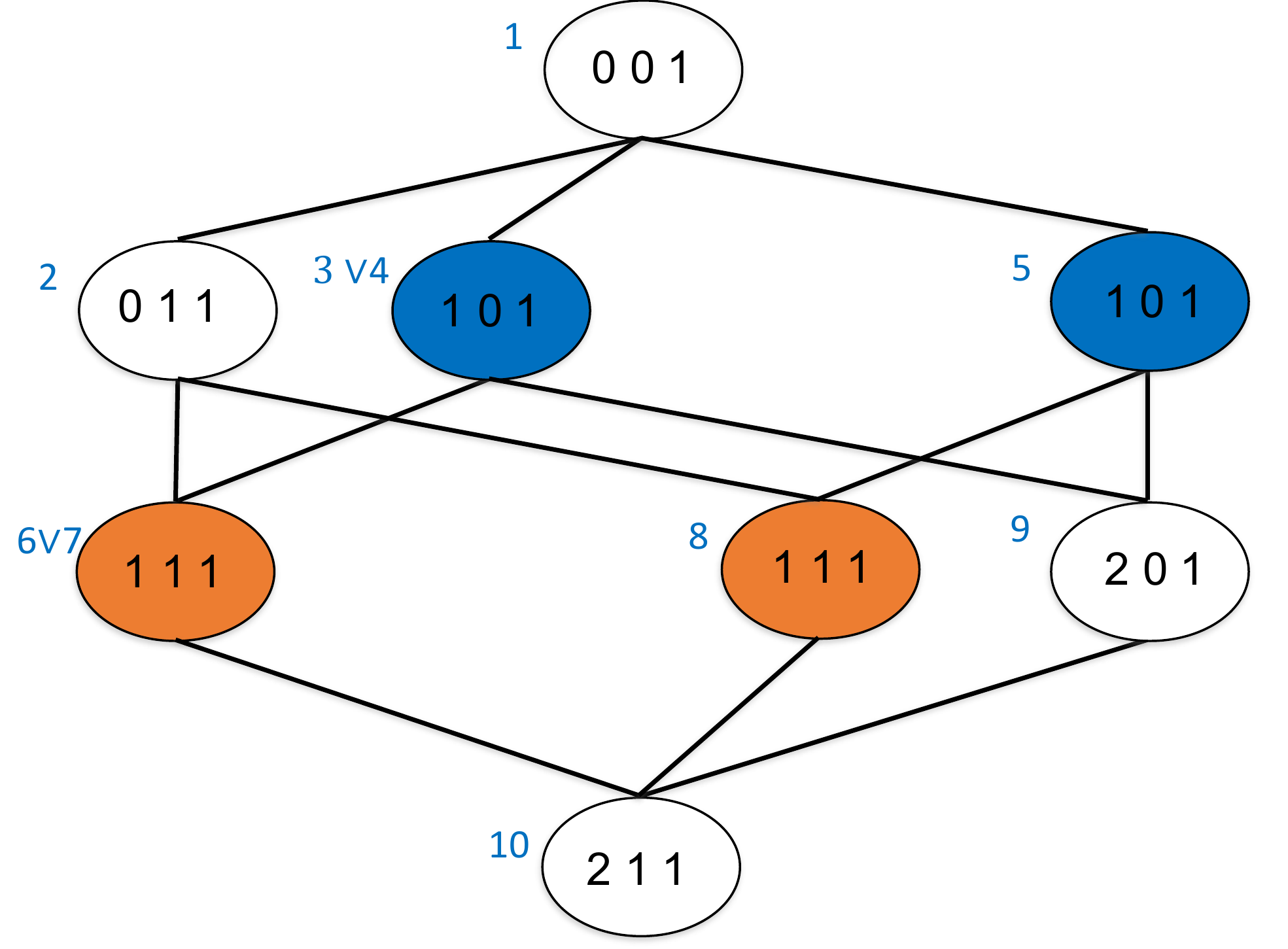}} \\
& & & \\
& & & \\
\multicolumn{2}{c}{{Type 1  $\mathcal{P}_{A}$}} &
\multicolumn{2}{c}{{Type 1  $\mathcal{P}_{A}/{=}$}} \\
& & & \\
\includegraphics[scale=0.25]{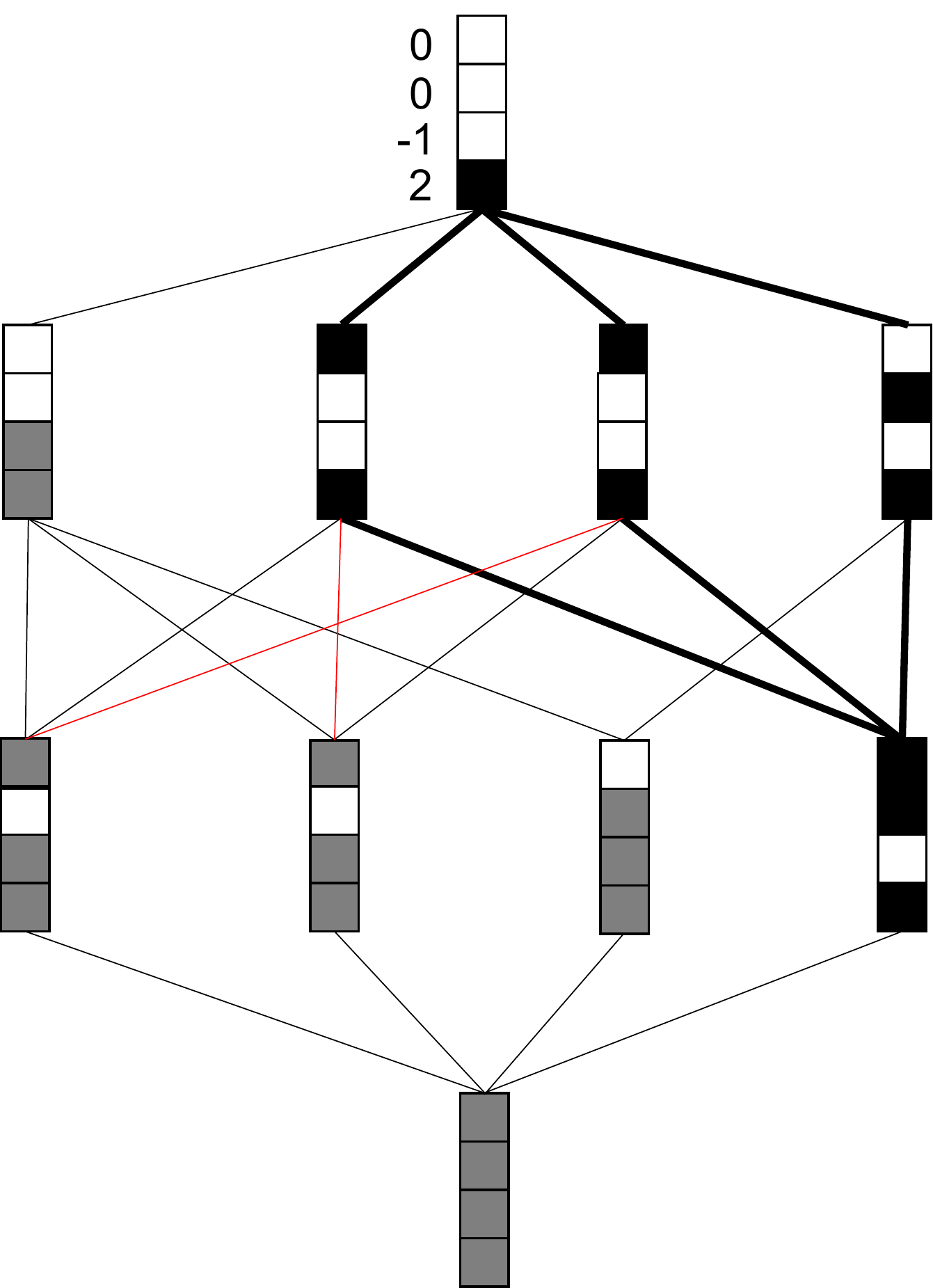} & 
\includegraphics[scale=0.25]{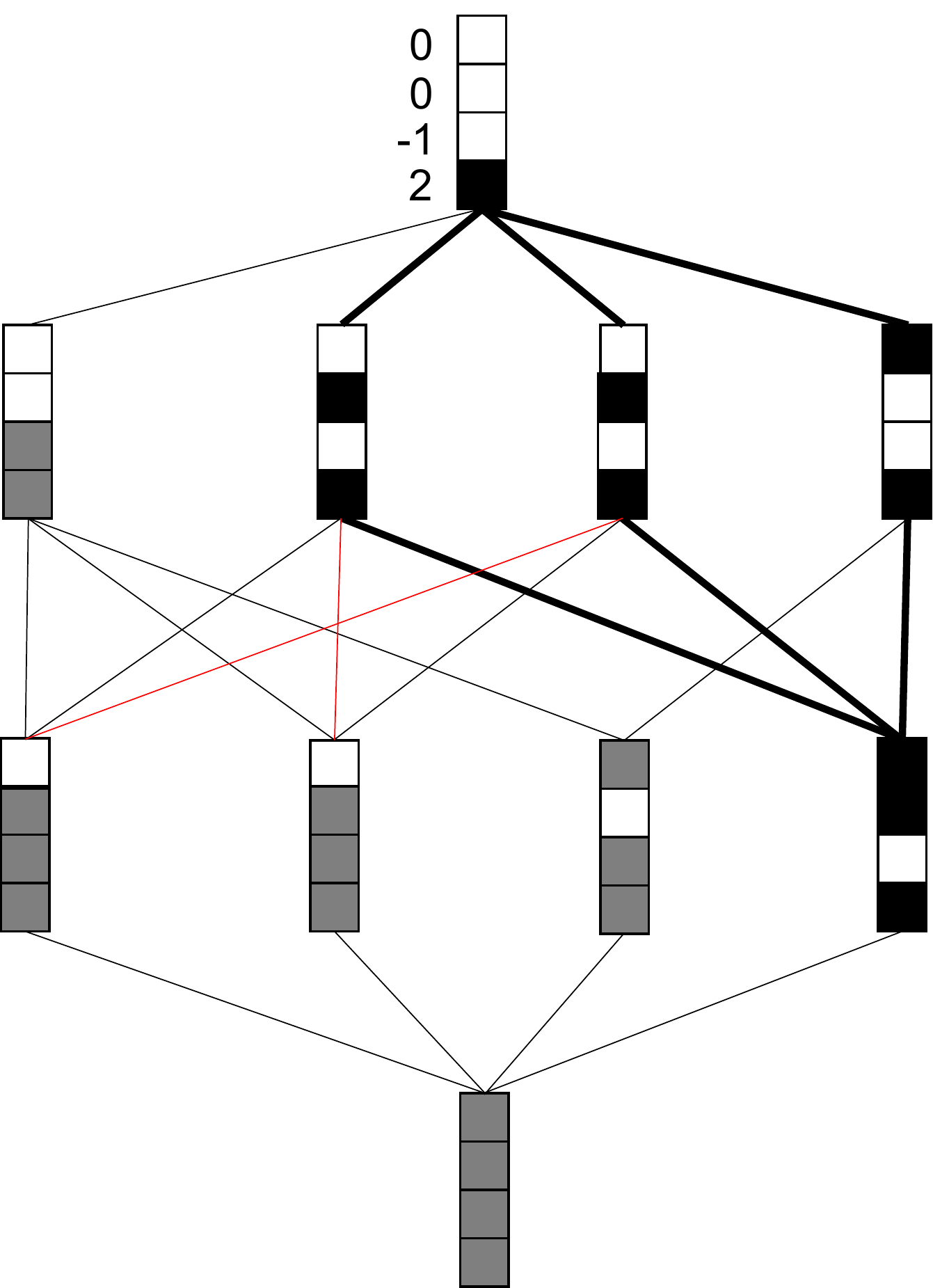} & 
\includegraphics[scale=0.25]{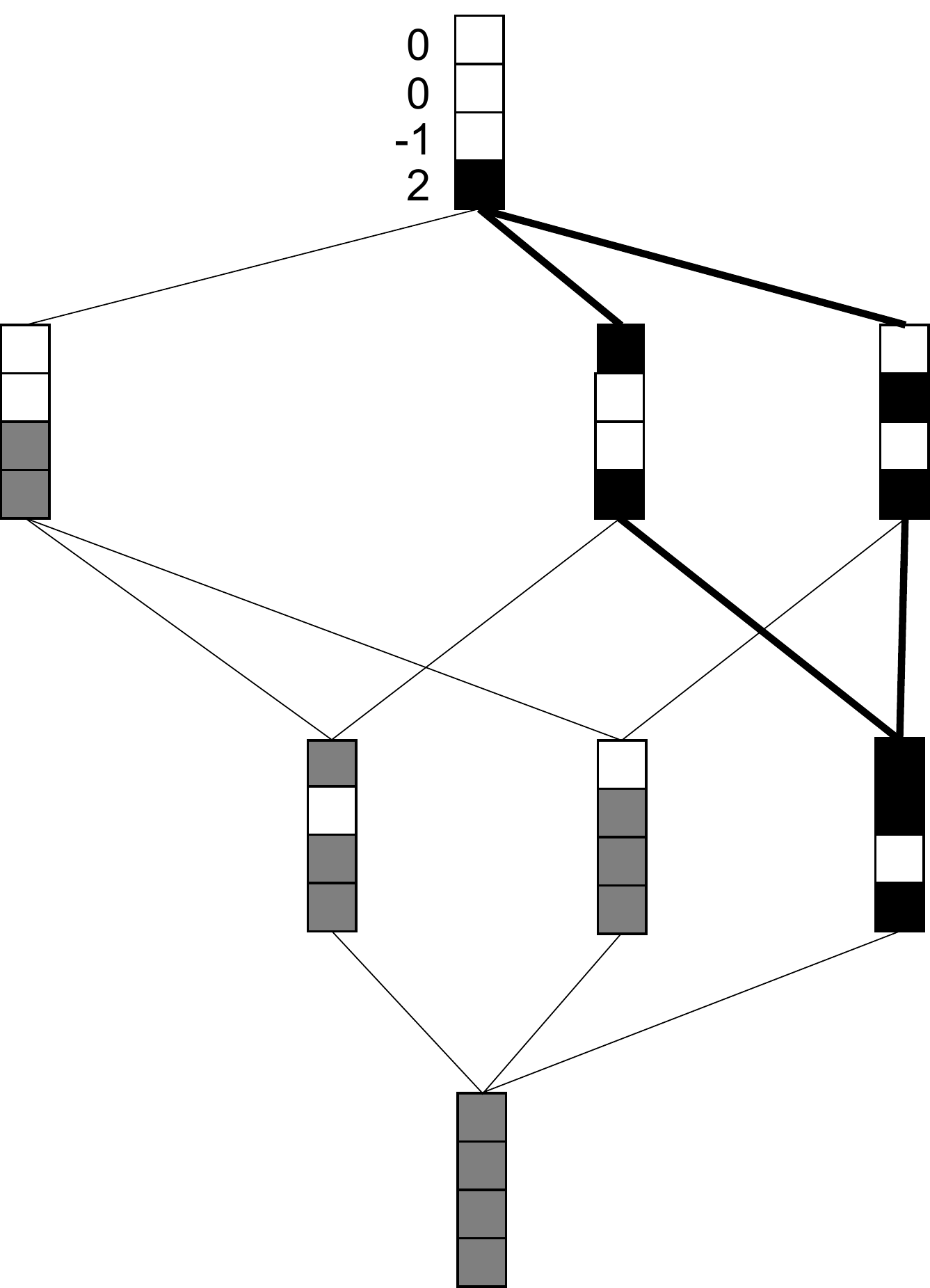} &
\includegraphics[scale=0.25]{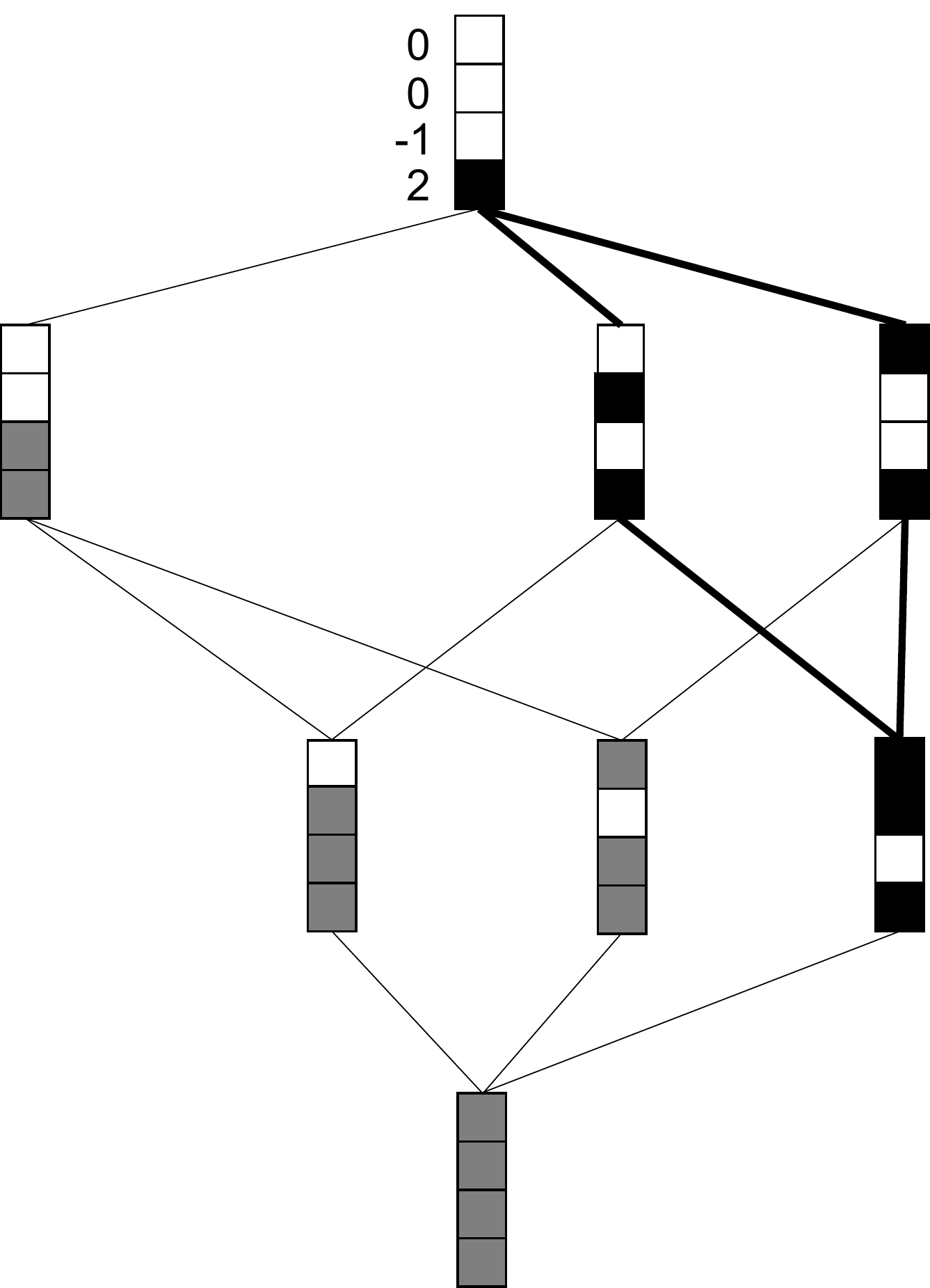}
\end{tabular}
\end{center}
\caption{Before and after reduction on the lattice of synchrony subspaces $V_{\mathcal{G}}^{P}$, $\mathcal{E}_{A}$ and $\mathcal{P}_{A}$ for {Type 1}. Note that there are two possible $\mathcal{P}_{A}$ associated with the unique {$\mathcal{E}_{A}$}. This is due to the equal size of Jordan blocks associated with the eigenvalue $\lambda=0$. The structures of $\mathcal{P}_{A}(1,1,1)$, which are subsets of $\mathcal{P}_{A}(1,1,1,1)$ and correspond to Type $1$ $\mathcal{P}_{A}$ in Figure \ref{fig:3cell_reduction_summary}, are highlighted with darker color.}
\label{fig:4_cell_network2_EA_reduction_Type1_1}
\end{figure}

\begin{figure}[h!]
\small
\begin{center}
\begin{tabular}{cc|cc}
\multicolumn{2}{c|}{Before Reduction} & \multicolumn{2}{c}{After Reduction} \\
\hline
\hline
\multicolumn{2}{c|}{$V_{\mathcal{G}}^{P}$} &
\multicolumn{2}{c}{$V_{\mathcal{G}}^{P}/{\sim}$}
\\
& & & \\
\multicolumn{2}{c|}{\includegraphics[scale=0.3]{4_cell_network2_VG.pdf}} &
\multicolumn{2}{c}{\includegraphics[scale=0.3]{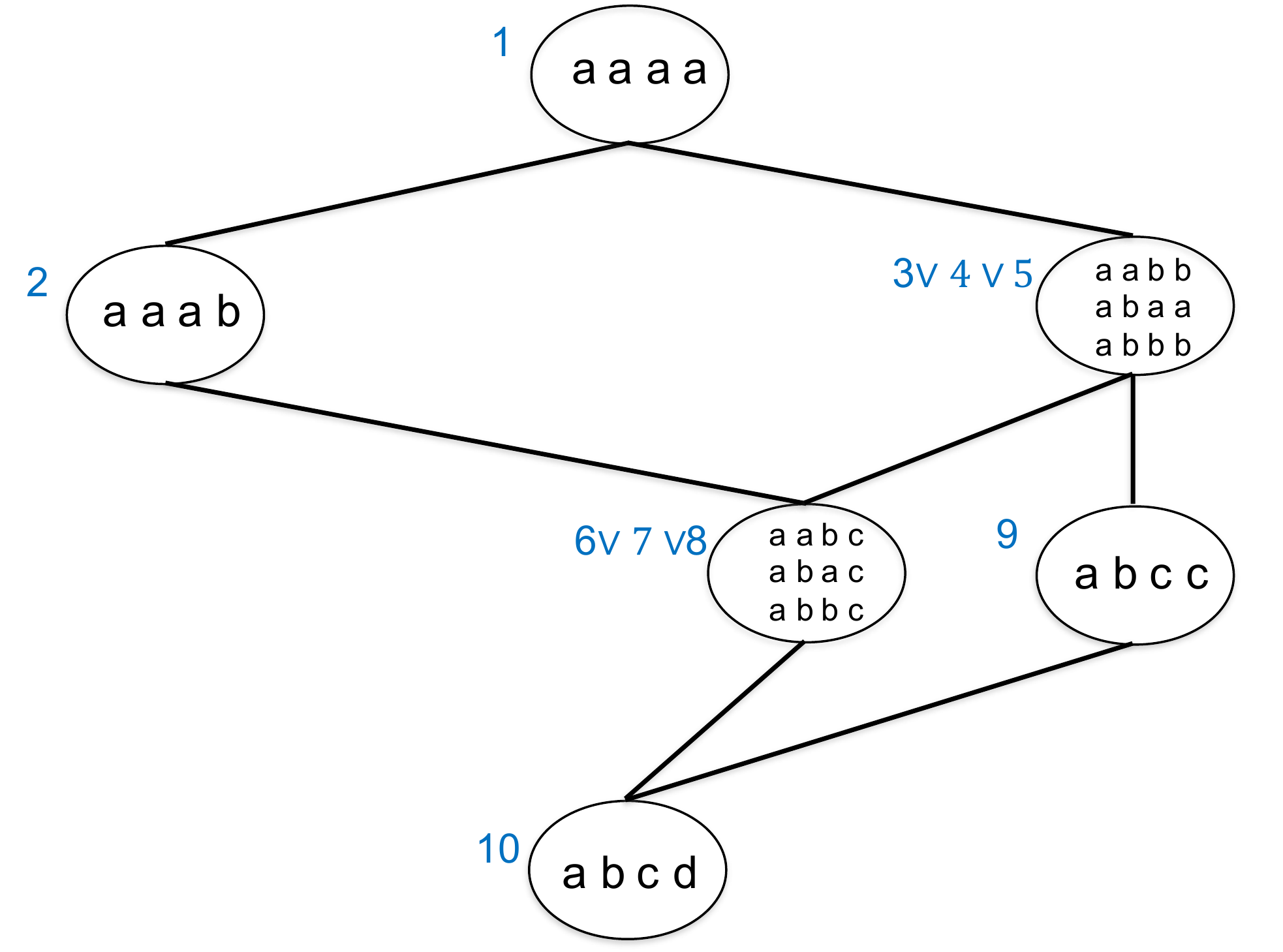}}
\\
& & & \\
& & & \\
\multicolumn{2}{c|}{$\mathcal{E}_{A}$} &
\multicolumn{2}{c}{$\mathcal{E}_{A}/{\sim}$} \\
& & & \\
\multicolumn{2}{c|}{\includegraphics[scale=0.3]{Eigenvalue_Tuple_4_cell_network2.pdf}} & 
\multicolumn{2}{c}{\includegraphics[scale=0.3]{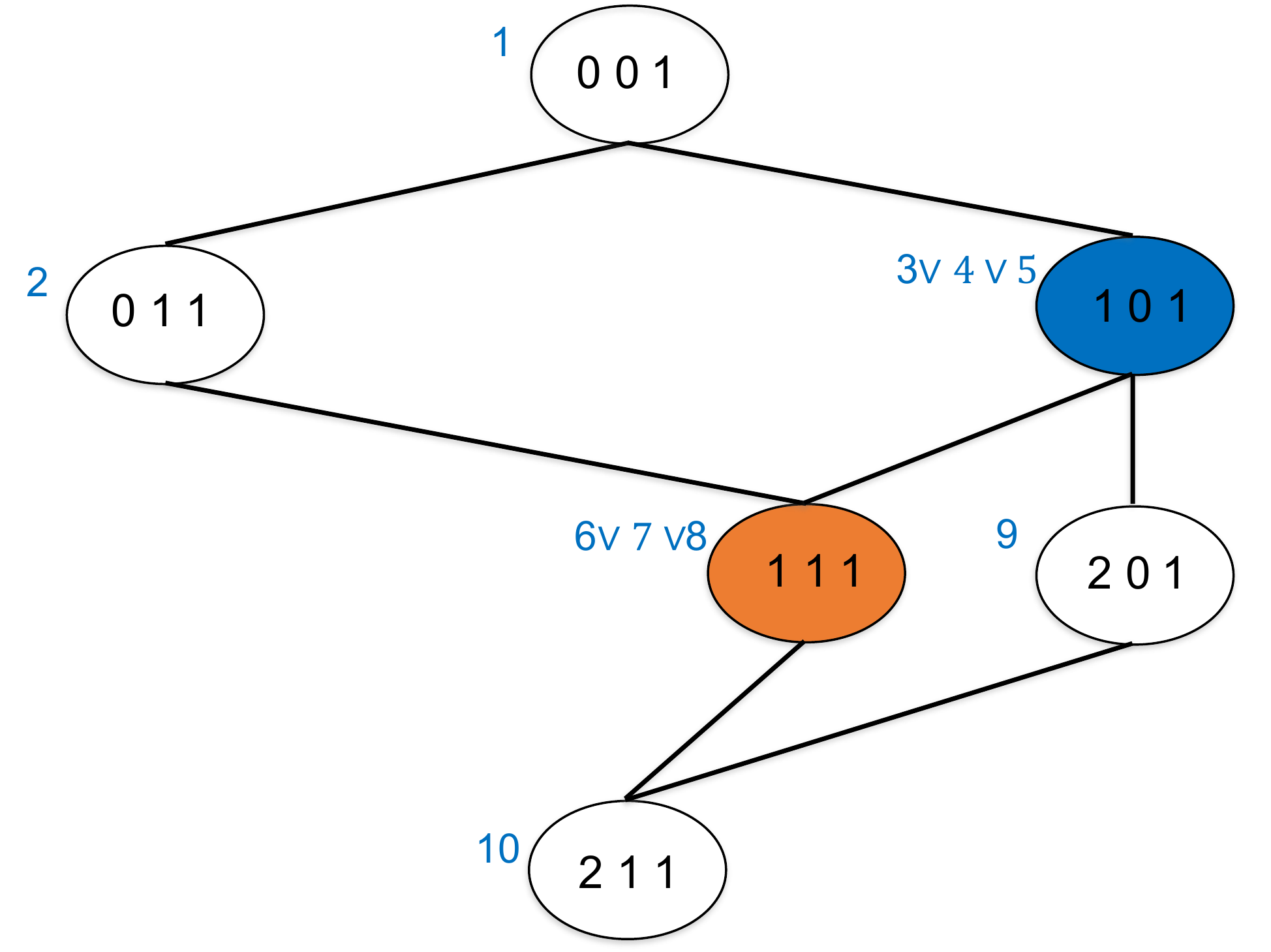}} \\
& & & \\
& & & \\
\multicolumn{2}{c}{Type 2 $\mathcal{P}_{A}$} &
\multicolumn{2}{c}{{Type 2 $\mathcal{P}_{A}/{=}$}} \\
& & & \\
\includegraphics[scale=0.25]{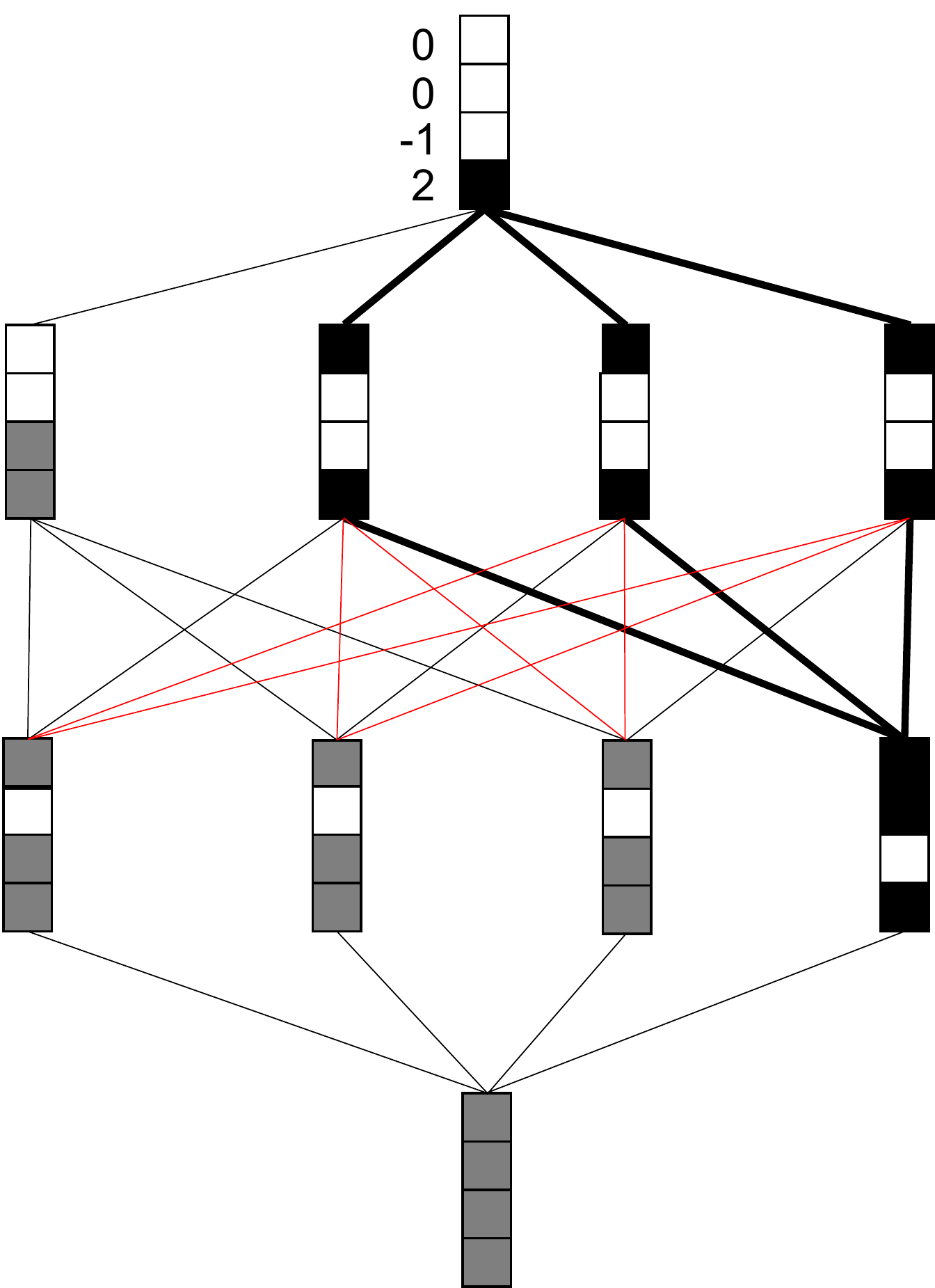} & 
\includegraphics[scale=0.25]{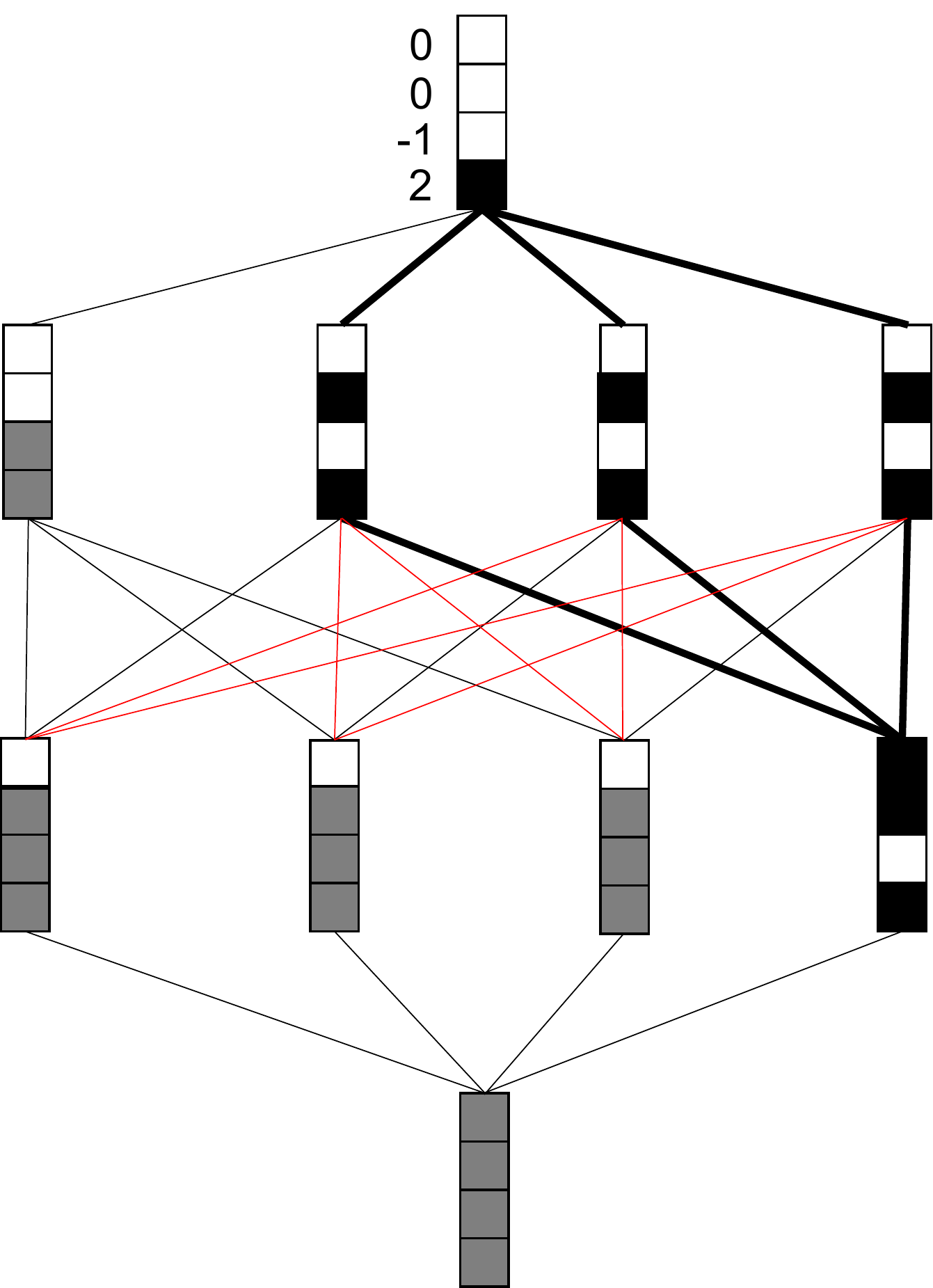} & 
\includegraphics[scale=0.25]{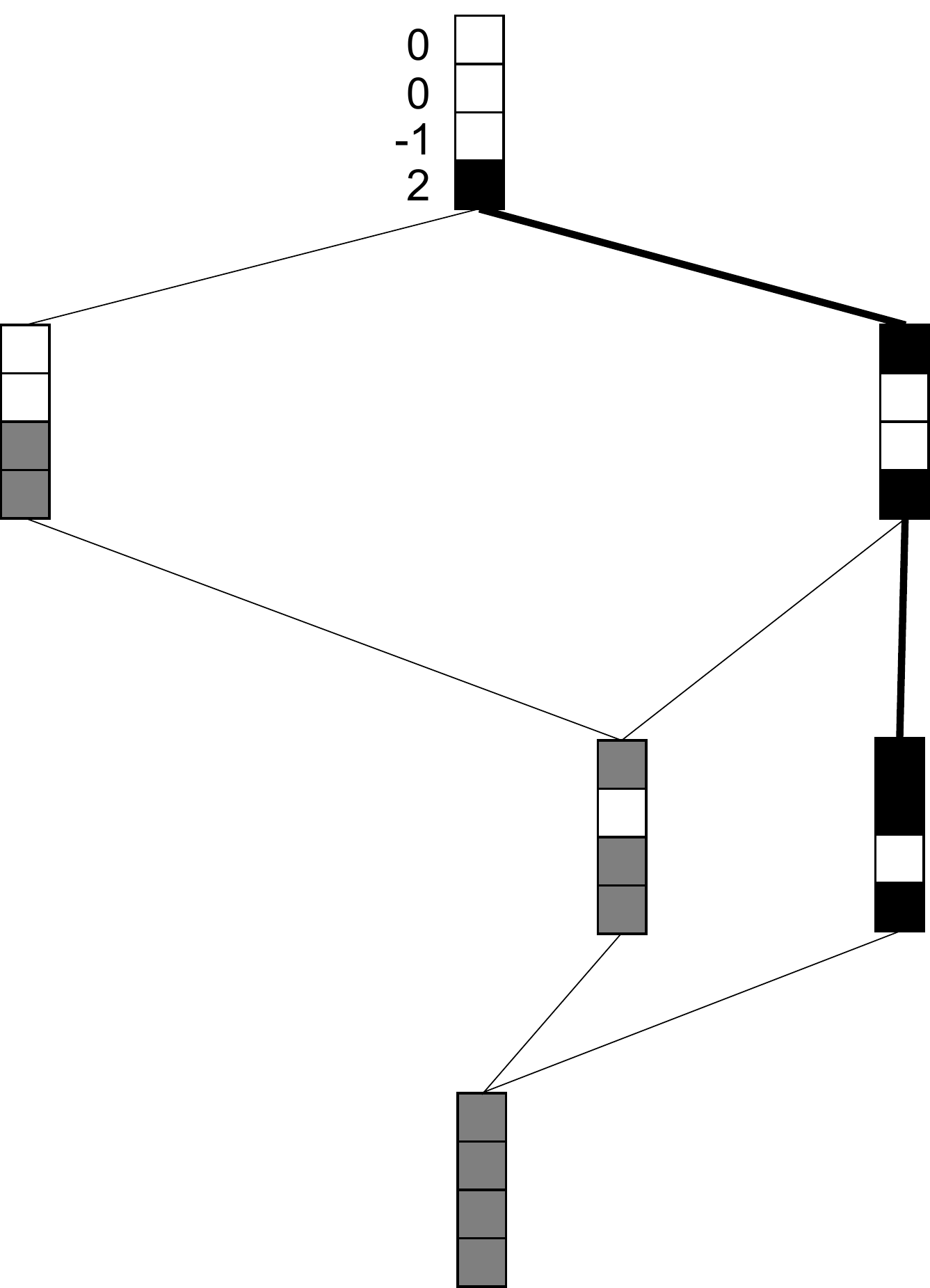} &
\includegraphics[scale=0.25]{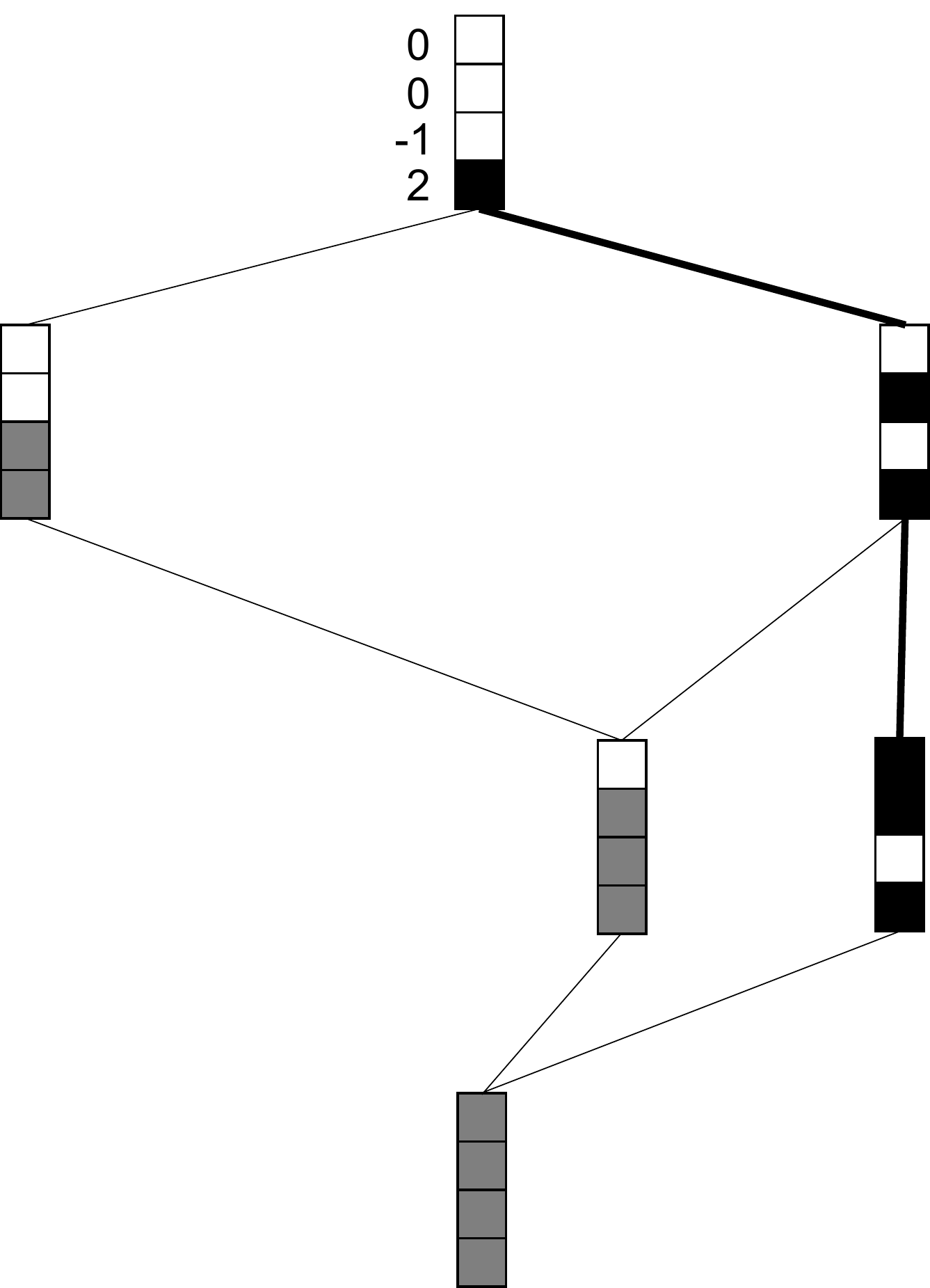}
\end{tabular}
\end{center}
\caption{Before and after reduction on the lattice of synchrony subspaces $V_{\mathcal{G}}^{P}$, $\mathcal{E}_{A}$ and $\mathcal{P}_{A}$ for Type 2. Note that there are two possible $\mathcal{P}_{A}$ associated with the unique $\mathcal{E}_{A}$ due to the equal size of Jordan blocks associated with the eigenvalue $\lambda=0$. The structures of $\mathcal{P}_{A}(1,1,1)$, which are subsets of $\mathcal{P}_{A}(1,1,1,1)$ and correspond to Type $2$ $\mathcal{P}_{A}$ in Figure \ref{fig:3cell_reduction_summary}, are highlighted with darker color.}
\label{fig:4_cell_network2_EA_reduction_Type2}
\end{figure}
\end{ex}


\subsection{Counter example}
We show {a counter example network} where the reduction does not work since the {network does} not satisfy the covering relation defined in Definition \ref{def:covering_relation}. 

\begin{ex}
\rm
\label{ex:5_cell_network4}
Consider the $5$-cell regular network in Figure \ref{fig:5_cell_network4}. The associated adjacency matrix $A$ has a repeated eigenvalue with algebraic multiplicity $3$ and geometric multiplicity $2$. 
\begin{figure}[h!]
\begin{center}
\begin{tabular}{ccl}
Network $\mathcal{G}$ & Adjacency matrix $A$ & eigenvalues \\
\hline
\multirow{5}{*}{
	\includegraphics[scale=0.25]{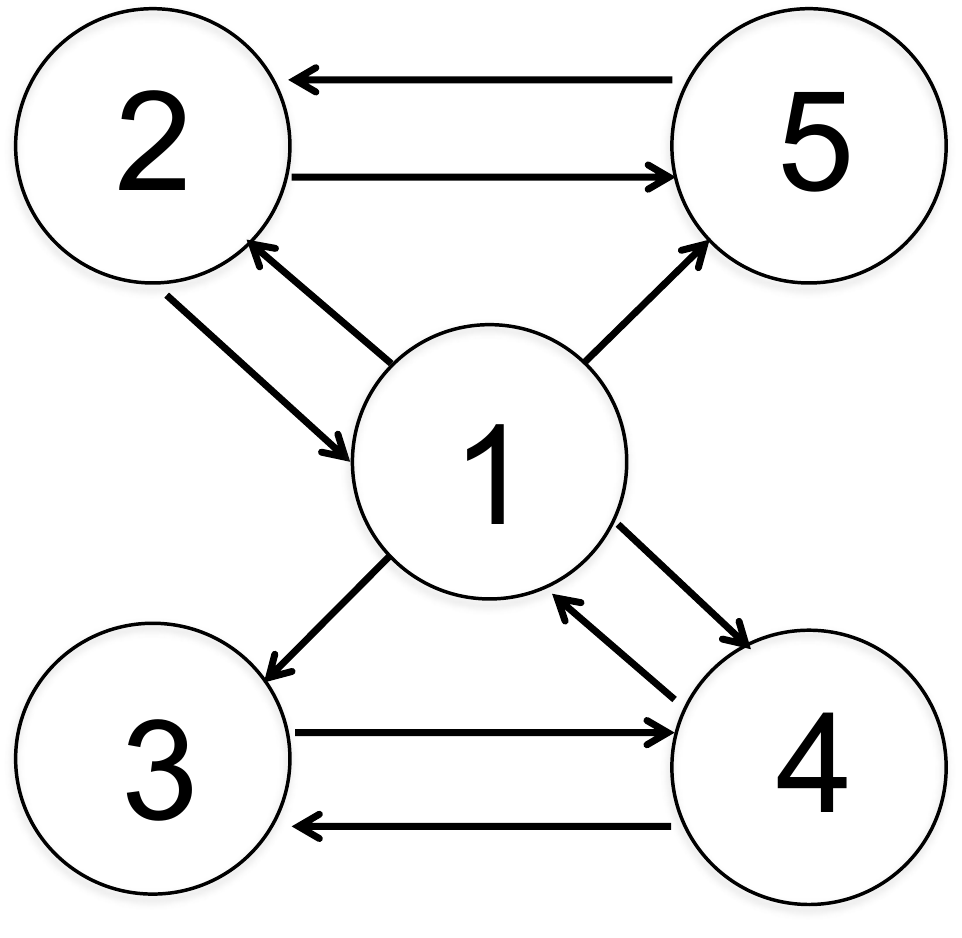}
} &
\multirow{5}{*}{$\left(\begin{array}{ccccc}
0 & 1 & 0 & 1 & 0\\
1 & 0 & 0 & 0 & 1 \\
1 & 0 & 0 & 1  & 0\\
1 & 0 & 1 & 0 & 0 \\
1 & 1 & 0 & 0 & 0  
\end{array}\right)$} &
$\lambda_{1}=-1$ \\
& & $\lambda_{2}=-1$ \\
& & $\lambda_{3}=-1$ \\
& & $\lambda_{4}=1$ \\
& & $\lambda_{5}=2$  
\end{tabular}
\end{center}
\caption{$5$-cell regular network $\mathcal{G}$ with the corresponding adjacency matrix $A$ and its eigenvalues. The repeated eigenvalue $\lambda_{1}=\lambda_{2}=\lambda_{3}=-1$ has algebraic multiplicity $3$ and geometric multiplicity $2$.}
\label{fig:5_cell_network4}
\end{figure}

Figure \ref{fig:5cell_network4_EA_reduction} summarizes lattices of synchrony subspaces $V_{\mathcal{G}}^{P}$, $\mathcal{E}_{A}$, and $\mathcal{P}_{A}$ before and after a reduction attempt. We find the equivalence relation
\[\bowtie=(1)(23)(4)(5)(6)(7)(89)(10)\]
is the only equivalence relation, which gives $\bowtie$-balanced $E$ as shown in Table \ref{tab:Eigenvalue_Tuple_5_cell_network4_reduction}. However, the associated $\mathcal{E}_{A}/{\sim}$ is not a closed subset of $\mathcal{L}_{M}(2,1,1,1)$ since there exists a $s\in\mathcal{E}_{A}/{\sim}$ such that $\Ind_{\mathcal{E}_{A}/{\sim}}(s)<0$. As a result, reduction fails. 

Note that we can determine $\mathcal{P}_{A}$ uniquely by Corollary \ref{cor:map_T_unique} since Jordan blocks have distinct sizes for the repeated eigenvalue $-1$. However, this $\mathcal{P}_{A}$ does not satisfy the covering relation defined in Definition \ref{def:covering_relation}.
\begin{figure}[h!]
\small
\begin{center}
\begin{tabular}{cc|cc}
\multicolumn{2}{c|}{Before Reduction} & \multicolumn{2}{c}{After Reduction} \\
\hline
\hline
\multicolumn{2}{c|}{$V_{\mathcal{G}}^{P}$} &
$V_{\mathcal{G}}^{P}/{\sim}$ & \\
& & & \\
\multicolumn{2}{c|}{\includegraphics[scale=0.21]{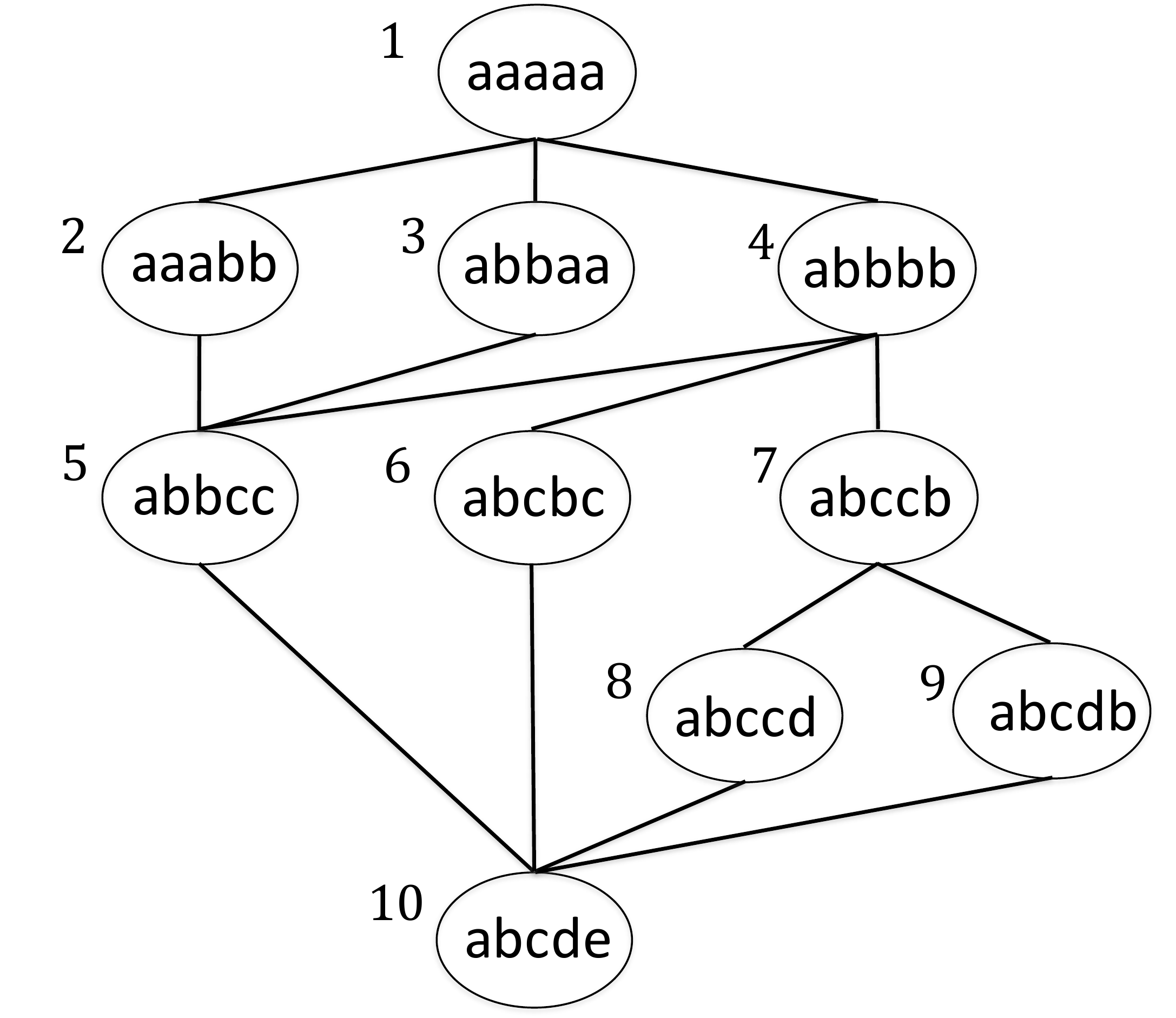}} &
\includegraphics[scale=0.21]{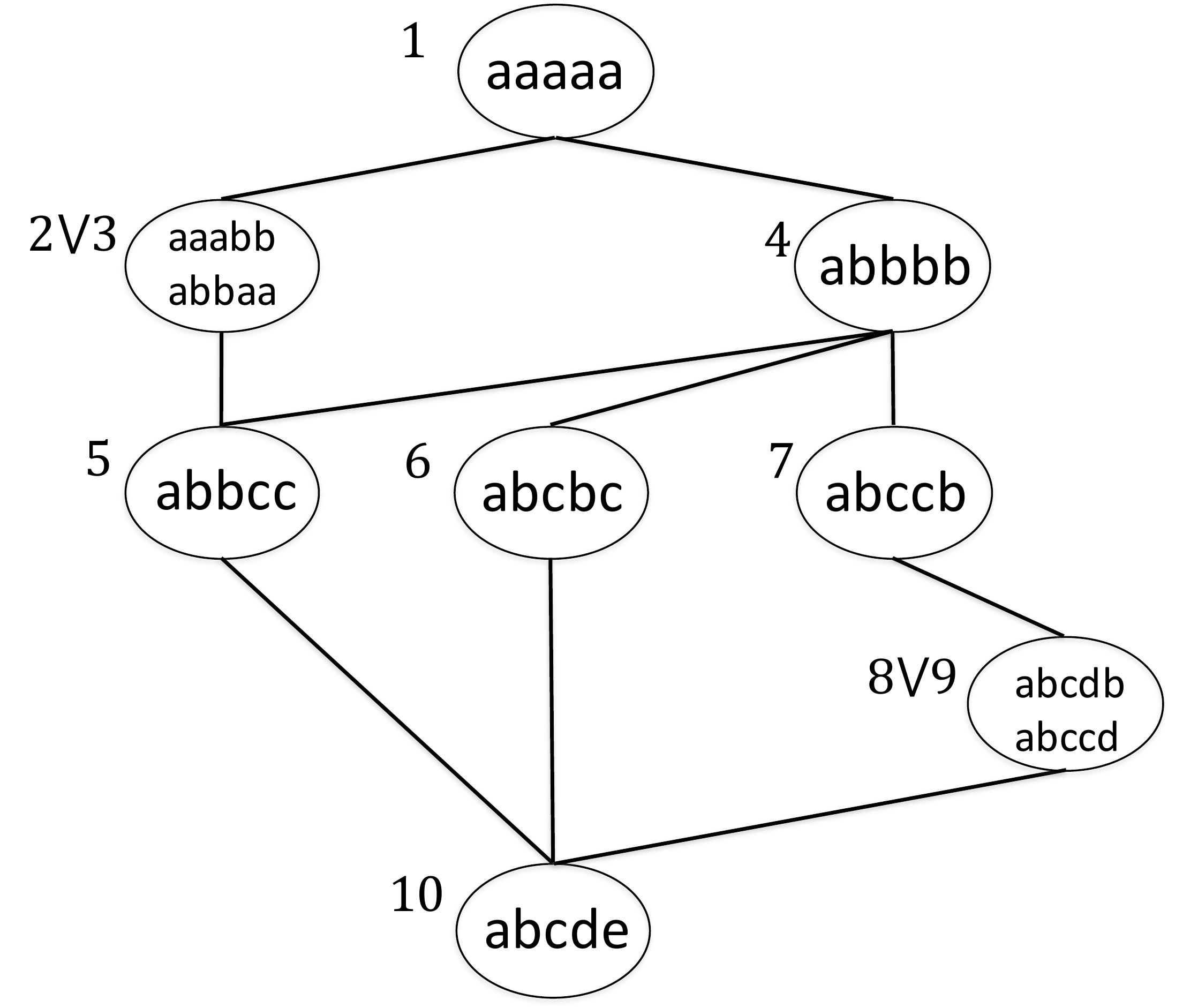} & \\
& & & \\
& & & \\
\multicolumn{2}{c|}{$\mathcal{E}_{A}$} &
$\mathcal{E}_{A}/{\sim}$ &  
Index on $\mathcal{E}_{A}/{\sim}$\\
& & & \\
\multicolumn{2}{c|}{\includegraphics[scale=0.21]{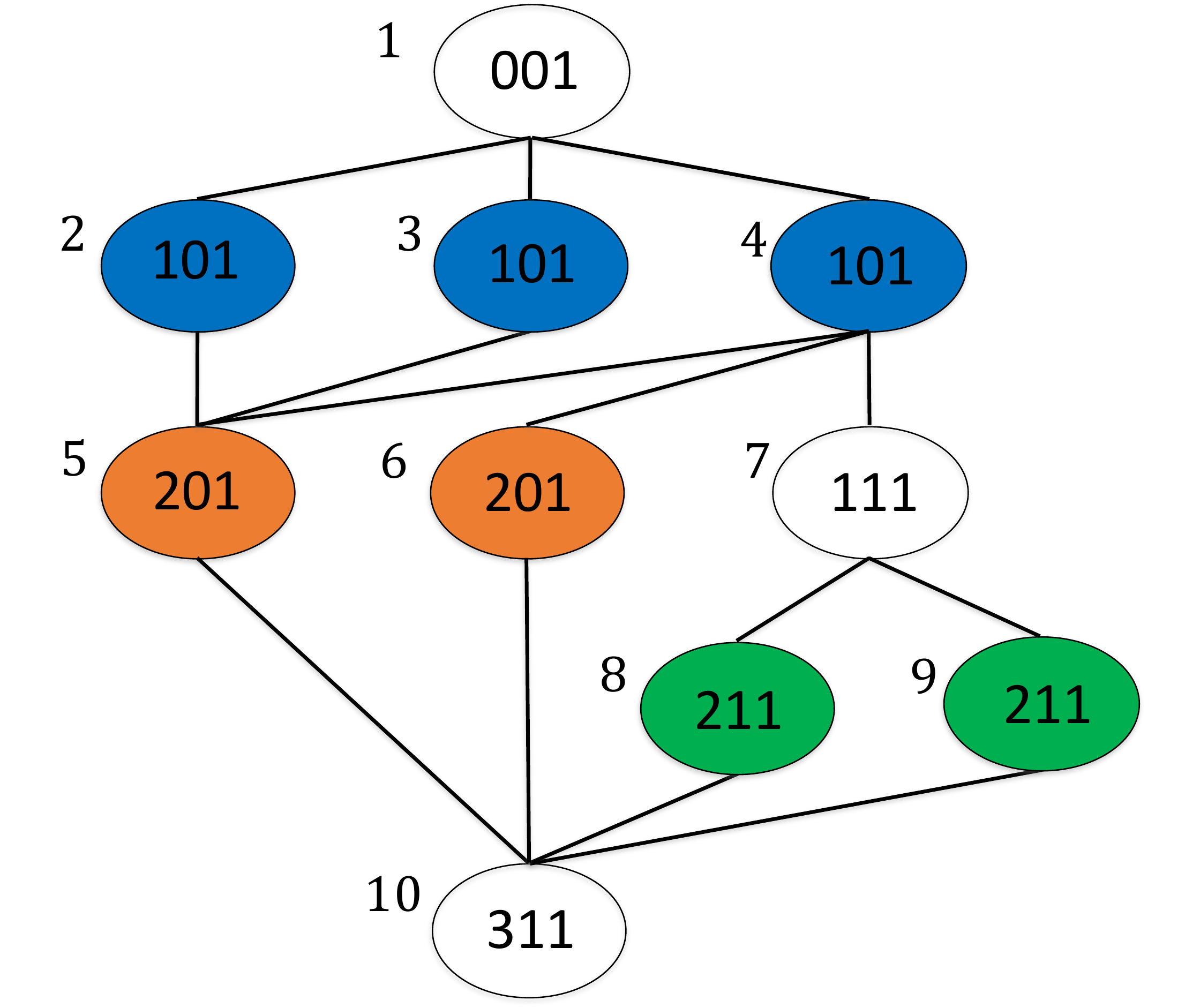}} & 
\includegraphics[scale=0.21]{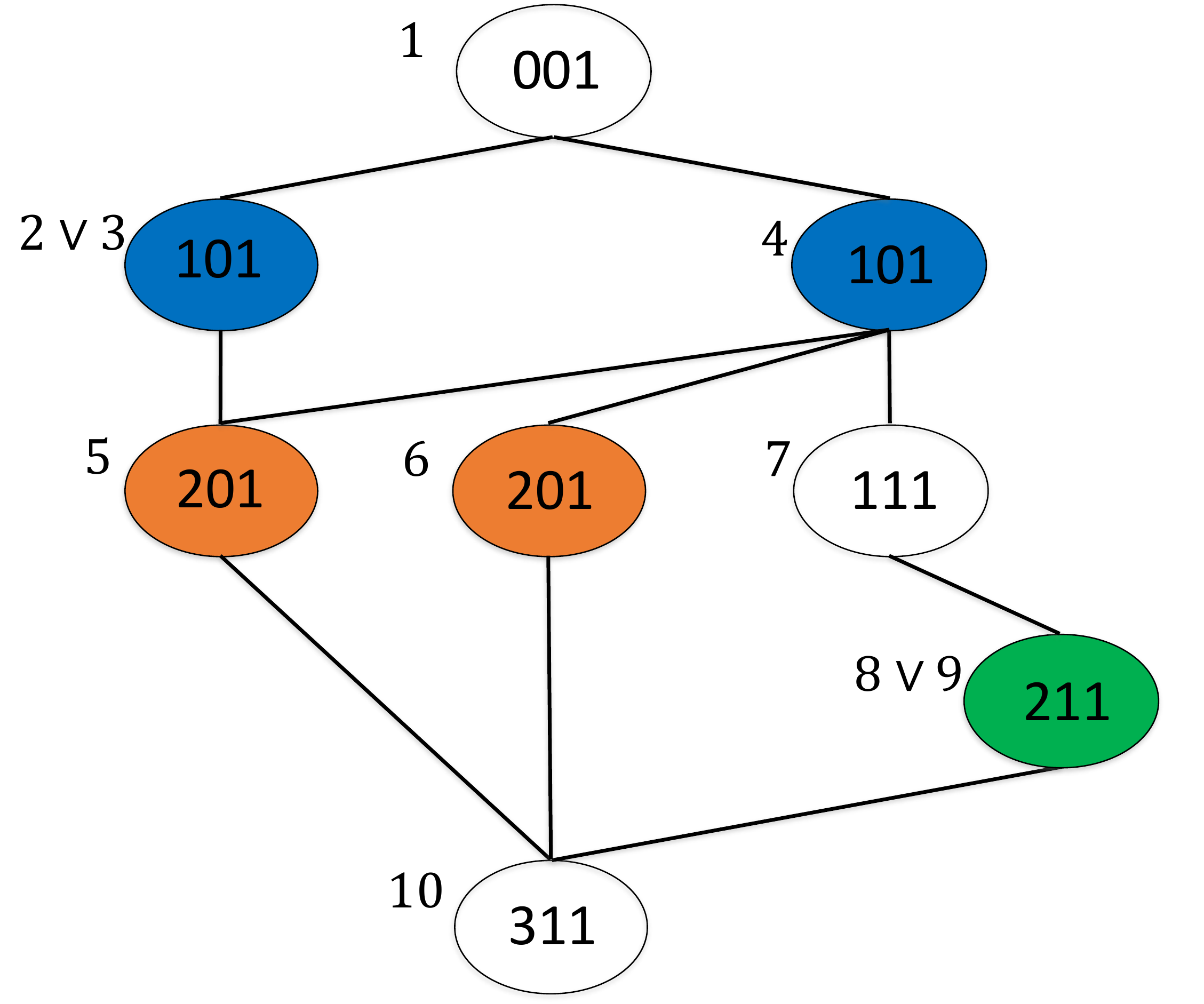} &
\includegraphics[scale=0.21]{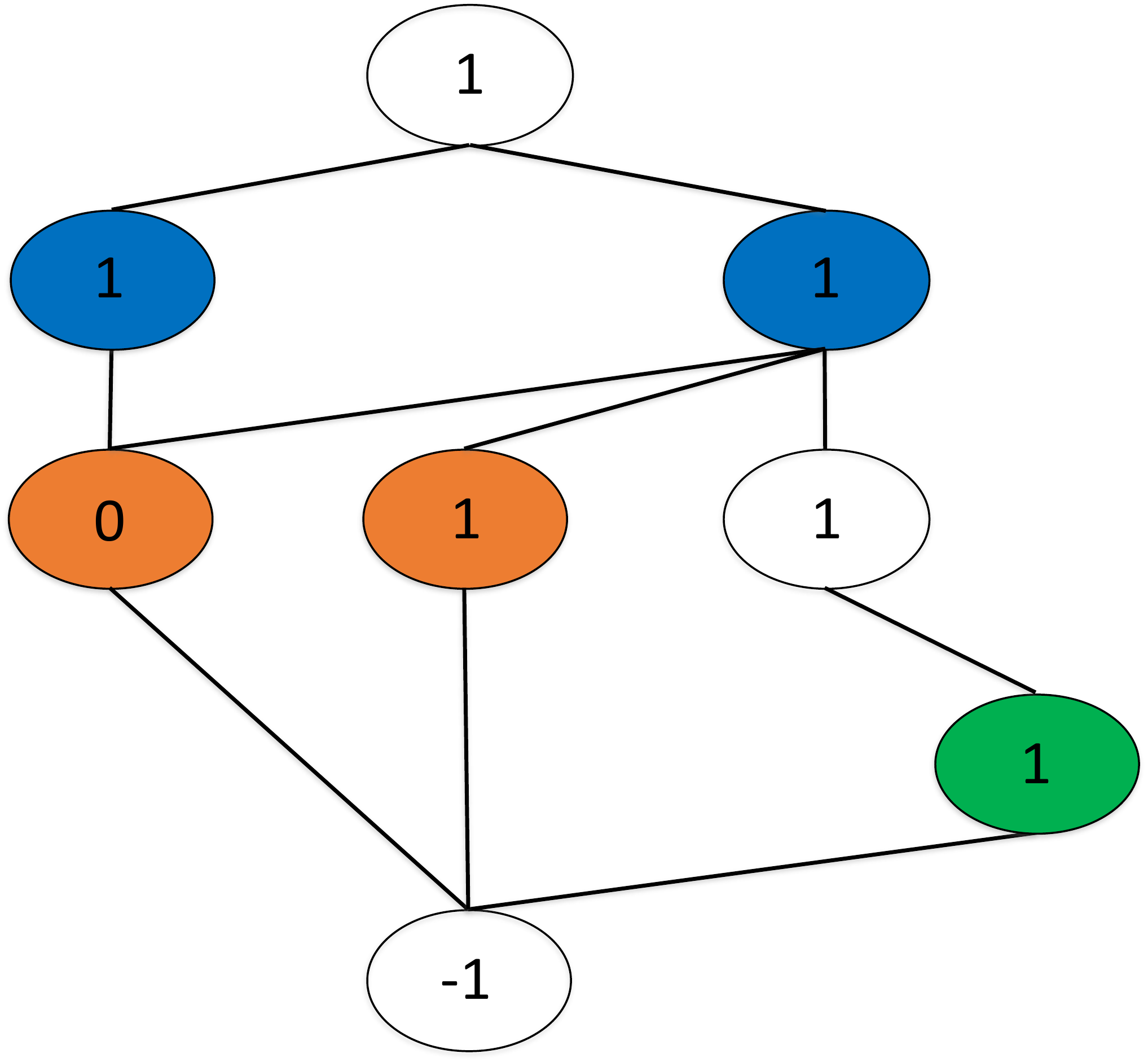}\\
& & & \\
& & & \\
\multicolumn{2}{c|}{$\mathcal{P}_{A}$} &
$\mathcal{P}_{A}/{=}$ & \\
& & & \\
\multicolumn{2}{c|}{\includegraphics[scale=0.25]{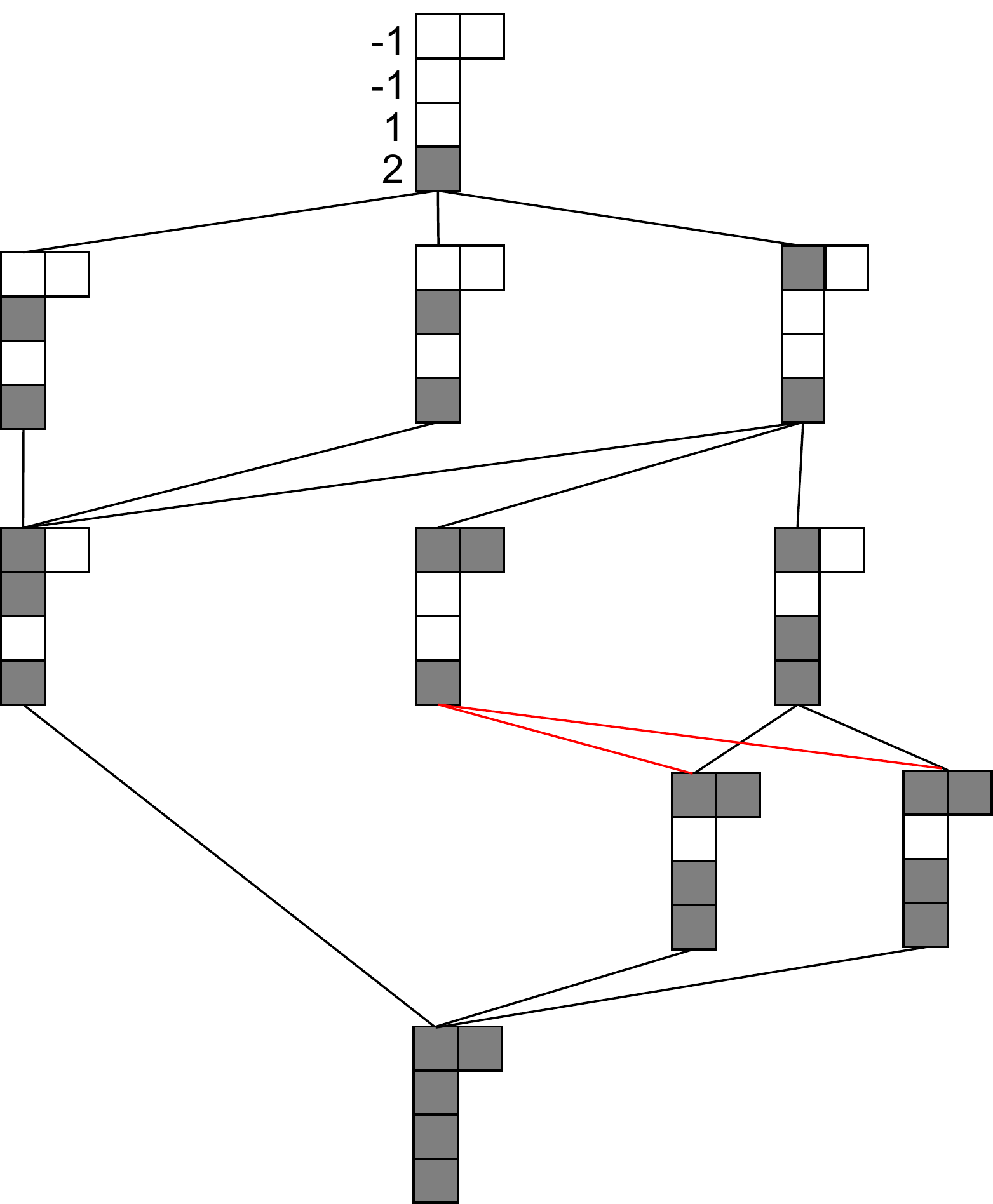}} & 
\includegraphics[scale=0.25]{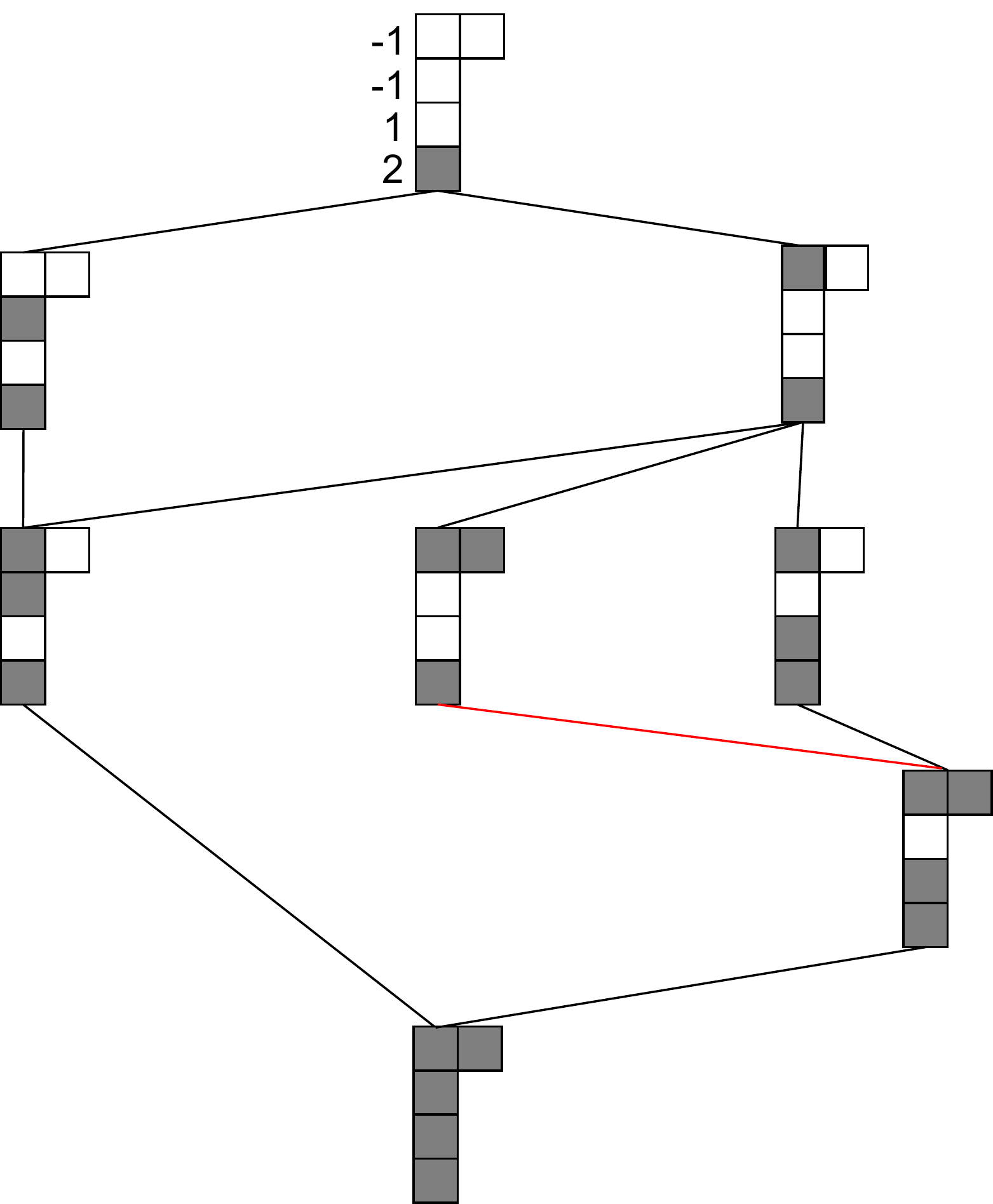} &
\end{tabular}
\end{center}
\caption{Before and after an attempt of the reduction on  $V_{\mathcal{G}}^{P}$, $\mathcal{E}_{A}$ and $\mathcal{P}_{A}$. Note that {there exists a $s\in\mathcal{E}_{A}/{\sim}$ such that $\Ind_{\mathcal{E}_{A}/{\sim}}(s)<0$}, which implies that the reduction does not return a closed subset of $\mathcal{L}_{M}(2,1,1,1)$.}
\label{fig:5cell_network4_EA_reduction}
\end{figure}

\begin{table}[h!]
\begin{tabular}{l}
(a) Matrix representation $E$ of $\mathcal{E}_{A}$ \\
$\begin{array}{c|cccccccccc}
         & 1 & 2 & 3 & 4 & 5 & 6 & 7 & 8 & 9 & 10 \\
\hline 
$1$   & 1 & 1 & 1 & 1 & 0 & 0 & 0 & 0 & 0 & 0  \\
$2$   & 1 & 1 & 0 & 0 & 1 & 0 & 0 & 0 & 0 & 0  \\
$3$   & 1 & 0 & 1 & 0 & 1 & 0 & 0 & 0 & 0 & 0  \\
$4$   & 1 & 0 & 0 & 1 & 1 & 1 & 1 & 0 & 0 & 0  \\
$5$   & 0 & 1 & 1 & 1 & 1 & 0 & 0 & 0 & 0 & 1  \\
$6$   & 0 & 0 & 0 & 1 & 0 & 1 & 0 & 0 & 0 & 1  \\
$7$   & 0 & 0 & 0 & 1 & 0 & 0 & 1 & 1 & 1 & 0  \\
$8$   & 0 & 0 & 0 & 0 & 0 & 0 & 1 & 1 & 0 & 1 \\
$9$   & 0 & 0 & 0 & 0 & 0 & 0 & 1 & 0 & 1 & 1  \\
$10$ & 0 & 0 & 0 & 0 & 1 & 1 & 0 & 1 & 1 & 1  \\
\hline
\textrm{tuple} &  (001) & \color{blue}{(101)} & \color{blue}{(101)} & \color{blue}{(101)} & \color{orange}{(201)} & \color{orange}{(201)} & (111) & \color{green}{(211)} & \color{green}{(211)} & (311)
\end{array}$ \\
\\
(b) $\bowtie$-balanced $E$ \\
$\begin{array}{c|cccccccc}
         & 1 & 2\vee 3     & 4 & 5 & 6 & 7 & 8\vee 9 & 10 \\
\hline 
$1$   & 1 & 1               & 1 & 0 & 0 & 0 & 0          & 0  \\
\color{blue}{2}   & \color{blue}{1} & \color{blue}{1}               & \color{blue}{0} & \color{blue}{1} & \color{blue}{0} & \color{blue}{0} & \color{blue}{0}          & \color{blue}{0}  \\
\color{blue}{3}   & \color{blue}{1} & \color{blue}{1}               & \color{blue}{0} & \color{blue}{1} & \color{blue}{0} & \color{blue}{0} & \color{blue}{0}          & \color{blue}{0}  \\
$4$   & 1 & 0               & 1 & 1 & 1 & 1 & 0          & 0  \\
$5$   & 0 & 1               & 1 & 1 & 0 & 0 & 0          & 1  \\
$6$   & 0 & 0               & 1 & 0 & 1 & 0 & 0          & 1  \\
$7$   & 0 & 0               & 1 & 0 & 0 & 1 & 1          & 0  \\
\color{green}{8}   & \color{green}{0} & \color{green}{0}               & \color{green}{0} & \color{green}{0} & \color{green}{0} & \color{green}{1} & \color{green}{1}          & \color{green}{1} \\
\color{green}{9}   & \color{green}{0} & \color{green}{0}               & \color{green}{0} & \color{green}{0} & \color{green}{0} & \color{green}{1} & \color{green}{1}          & \color{green}{1}  \\
$10$ & 0 & 0               & 0 & 1 & 1 & 0 & 1          & 1  \\
\hline
\textrm{tuple} &  (001) & \color{blue}{(101)} & \color{blue}{(101)} & \color{orange}{(201)} & \color{orange}{(201)} & (111) & \color{green}{(211)} & (311)
\end{array}$
\end{tabular}
\caption{(a) The matrix representation $E$ of a poset $\mathcal{E}_{A}$ by Definition \ref{def:connectivity_matrix_E}. Each node in $\mathcal{E}_{A}$ is given by a tuple representation $(s_{1}, s_{2}, s_{3})$ where $s_{1}$, $s_{2}$ and $s_{3}$ are the number of distinct eigenvalues $-1$, $1$ and $2$, respectively. Identical tuple representations are colored the same. (b) The equivalence relation $\bowtie=(1)(23)(4)(5)(6)(7)(89)(10)$ gives a $\bowtie$-balanced $E$ since, after the column manipulation, rows $2$ and $3$ are the same as well as rows $8$ and $9$. However, the associated $\mathcal{E}_{A}/{\sim}$ is not a a closed subset of $\mathcal{L}_{M}(2,1,1,1)$ due to a negative index for $s=(3,1,1)$ as shown in Figure \ref{fig:5cell_network4_EA_reduction}.}
\label{tab:Eigenvalue_Tuple_5_cell_network4_reduction}
\end{table}
\END
\end{ex}

\clearpage
\section{Conclusion}
\label{sec:conclusion}
We propose an eigenvalue-based algorithm for reducing lattices of synchrony subspaces associated with regular coupled cell networks. The reduced lattice $\mathcal{E}_{A}/{\sim}$, under a covering relation assumption on the lattice, corresponds to a lattice structure for regular networks whose coupling matrices have only simple eigenvalues (Theorem \ref{thm:reduced_EA}). The algorithm generally outputs more than one candidate reduced lattice. We use the tuple representation lattice $\mathcal P_A$ as a filter to identify the reduced lattice (Example \ref{ex:4_cell_network2}). At this moment, we do not know whether such filtering referring to $\mathcal P_A$ always leads to a unique reduced lattice structure, or even if using $\mathcal E_A$ alone is possible. We also give an example for which the algorithm fails to find the reduction (Example \ref{ex:5_cell_network4}). 

We also give an extension of the existing concept of lattice index for general regular networks, by first introducing an index on the tuple representation lattice $\mathcal P_A$ and then extend it to the eigenvalue lattice $\mathcal E_A$. The combination of the two types of indices gives rise to the reduction strategy for our algorithm (Proposition \ref{prop:positivity_Ind_2} and Corollary \ref{cor:positivity_Ind}). As a consequence, the algorithm outputs not only the reduced lattice, but also the lattice index defined on it.  

In relation to the literature \cite{ADS2020,ADGL09,Golubitsky-2009,Soares17}, in future work we aim to employ the algorithm for index-based bifurcation analysis, which can be used for predicting multiple bifurcating branches from a single bifurcation point in regular coupled cell networks. It will be especially interesting to investigate some counter examples for which our algorithm fails, and to examine what our postulated reduction can tell us about generic synchrony-breaking bifurcations.

\section*{Acknowledgements}
The code to produce reduced lattices is provided as a supplementary material for the manuscript. This is based on the code published in \cite{Kamei-2013}, and authors acknowledge Dr. Peter Cock for the implementation of the code. Authors also acknowledge the University of Hamburg for its support and hospitality.

\bibliographystyle{plainnat}
\bibliography{lattice_bibliography}
\end{document}